\documentclass[11pt]{article}
\usepackage{amsmath,enumerate,amsfonts,color,amssymb,amsthm, verbatim}

\usepackage{hyperref}
\usepackage[normalem]{ulem}

\usepackage{todonotes}

\usepackage{slashed}

\usepackage{tikz}
\usetikzlibrary{patterns}

\def\ZZ{{\mathbb Z}}
\def\RR{{\mathbb R}}
\def\Sphere{{\mathbb S}}

\def\Qgl{Q_\eps^{GL}}
 
\newcommand{\f}{\varphi}

\def\bS{{\mathbb S}}
\def\eps{{\varepsilon}}
\def\f{\varphi}

\numberwithin{equation}{section}

\newtheorem{theorem} {\sc  Theorem\rm} [section]

\newtheorem{corollary} [theorem] {\sc  Corollary\rm}
\newtheorem{lemma} [theorem] {\sc  Lemma\rm}
\newtheorem{proposition} [theorem] {\sc  Proposition\rm}

\newtheorem{definition}[theorem]{\sc  Definition\rm}

\newtheorem{remark}[theorem]{\sc  Remark\rm}

\def\nd{\noindent}

\newcounter{marnote}
\newcommand\marginnote[1]{\stepcounter{marnote}$^{\bullet\,\themarnote}$\marginpar{\tiny$\bullet\,\themarnote$:\,#1}}

\DeclareFontFamily{OT1}{rsfs}{}
\DeclareFontShape{OT1}{rsfs}{m}{n}{ <-7> rsfs5 <7-10> rsfs7 <10-> rsfs10}{}
\DeclareMathAlphabet{\mycal}{OT1}{rsfs}{m}{n}

\def\be{\begin{equation}}
\def\ee{\end{equation}}

\newcommand{\R}{\mathbb{R}}

\def\be{\begin{equation}}
\def\ee{\end{equation}}
\def\bea#1\eea{\begin{align}#1\end{align}}

\def\mcA{{\mycal A}}

\def\mcB{{\mycal B}}
\def\mcH{{\mycal H}}
\def\mcM{{\mycal M}}

\def\barfeps{f_\eps}

\def\nespm{\bar m}

\begin{document}
\title{Local minimality of $\RR^N$-valued and $\bS^N$-valued Ginzburg--Landau vortex solutions in the unit ball $B^N$}

\author{Radu Ignat\thanks{Institut de Math\'ematiques de Toulouse \& Institut Universitaire de France, UMR 5219, Universit\'e de Toulouse, CNRS, UPS
IMT, F-31062 Toulouse Cedex 9, France. Email: Radu.Ignat@math.univ-toulouse.fr
}~ and Luc Nguyen\thanks{Mathematical Institute and St Edmund Hall, University of Oxford, Andrew Wiles Building, Radcliffe Observatory Quarter, Woodstock Road, Oxford OX2 6GG, United Kingdom. Email: luc.nguyen@maths.ox.ac.uk}}

\date{}

\maketitle
\begin{abstract}
We study the existence, uniqueness and minimality of critical points of the form $m_{\eps,\eta}(x) = (f_{\eps,\eta}(|x|)\frac{x}{|x|}, g_{\eps,\eta}(|x|))$ of the functional
\[
E_{\eps,\eta}[m] = \int_{B^N} \Big[\frac{1}{2} |\nabla m|^2 + \frac{1}{2\eps^2} (1 - |m|^2)^2 + \frac{1}{2\eta^2} m_{N+1}^2\Big]\,dx
\]
for $m=(m_1, \dots, m_N, m_{N+1}) \in H^1(B^N,\RR^{N+1})$ with $m(x) = (x,0)$ on $\partial B^N$. We establish a necessary and sufficient condition on the dimension $N$ and the parameters $\eps$ and $\eta$ for the existence of an escaping vortex solution $(f_{\eps,\eta}, g_{\eps,\eta})$ with $g_{\eps,\eta}> 0$. We also establish its uniqueness and local minimality. In the limiting case $\eta = 0$, we prove the local minimality of the degree-one vortex solution for the Ginzburg--Landau (GL) energy for every $\eps > 0$ and $N \geq 2$. Similarly, when $\eps = 0$, we prove the local minimality of the degree-one escaping vortex solution to an $\Sphere^N$-valued GL model arising in micromagnetics for every $\eta > 0$ and $2 \leq N \leq 6$.

\smallskip


\end{abstract}

\tableofcontents

\section{Introduction}

The minimality of the degree-one vortex solution for the Ginzburg-Landau system in the unit ball $B^N\subset \RR^N$ in dimension $2 \leq N \leq 6$ is an important open question for which a rich literature is available. In dimension $N \geq 7$, this has been proved recently in a joint work of the authors with Slastikov and Zarnescu \cite{INSZ18_CRAS}. In this paper, we address the local minimality of this solution. Motivated by the theory of magnetic materials, we also consider the local minimality of a similar vortex structure taking values into the unit sphere $\Sphere^{N}$. Our strategy is to treat the local minimality of the vortex solution for an extended model of which the previous two models are special limit cases.

We introduce first the Ginzburg--Landau (GL) functional
\[
E^{GL}_\eps [u] = \int_{B^N} \Big[\frac{1}{2} |\nabla u|^2 + \frac{1}{2\eps^2} W(1 - |u|^2)\Big]\,dx, 
\]
where $\eps > 0$, $W(t)=\frac{t^2}2$ and $u$ belongs to the set
\[
\mcA^{GL} = \{u \in H^1(B^N,\RR^N) : u(x) = x \text{ on } \partial B^N\}.
\]
The functional $E^{GL}_\eps$ has a unique radially symmetric critical point of the form (see Definition \ref{Def:RSMap} and Lemma \ref{Lem:ONSym} in Appendix \ref{App})
 \begin{equation}
u_\eps(x) = f_\eps(r) n(x) \in \mcA^{GL}, \quad n(x) = \frac{x}{r}, \quad r = |x|,
	\label{Eq:ueH1}
\end{equation}
where the radial profile $f_\eps$ is the unique solution to the ODE (see e.g. \cite{Hervex2, ODE_INSZ})
\begin{align}
&f_\eps'' + \frac{N-1}{r} f_\eps' - \frac{N-1}{r^2} f_\eps
	= -\frac{1}{\eps^2} W'(1 - f_\eps^2) f_\eps \quad \textrm{in} \quad (0,1),\label{Eq:24VII18-X1}\\
&f_\eps(1) = 1
	.\label{Eq:24VII18-X2}
\end{align}
Note that $f_\eps(0) = 0$ (see Lemma \ref{Lem:ONSym}). Here a map $u_{crit} \in \mcA^{GL}$ is said to be a \emph{bounded} critical point of $E_\eps^{GL}$ if $u_{crit} \in L^\infty(B^N, \RR^N)$ and $\langle DE_{\eps}^{GL}[u_{crit}], \varphi\rangle := \frac{d}{dt}\big|_{t = 0} E_{\eps}^{GL}[u_{crit} + t\varphi] = 0$ for all $\varphi \in C_c^\infty(B^N \setminus \{0\}, \RR^N)$ (which is dense in $H_0^1(B^N,\RR^N)$), and is said to be a \emph{radially symmetric} critical point of $E_\eps^{GL}$ if $u_{crit}$ is radially symmetric\footnote{By Lemma \ref{Lem:RSRep}, radially symmetric maps in $H^1(B^N,\RR^N)$ belong to $L^\infty_{loc}(\bar B^N \setminus \{0\},\RR^N)$.} in the sense of Definition \ref{Def:RSMap} and $\langle DE_{\eps}^{GL}[u_{crit}], \varphi\rangle = 0$ for all $\varphi \in C_c^\infty(B^N \setminus \{0\}, \RR^N)$. By Lemma \ref{lem:f2g2}, radially symmetric critical points of $E_\eps^{GL}$ are bounded.

The map $u_\eps$ in \eqref{Eq:ueH1}, called the \emph{($\RR^N$-valued) Ginzburg--Landau vortex solution} of topological degree one, can be considered as a regularization of the singular harmonic map $n: B^N \rightarrow \Sphere^{N-1}$ given by $n(x) = \frac{x}{|x|}$ for every $x\in B^N$, which is the unique minimizing $\bS^{N-1}$-valued harmonic map for $N \geq 3$ within the boundary condition $n(x)=x$ on $\partial B^N$ (see Brezis, Coron and Lieb \cite{BrezisCoronLieb} and Lin \cite{Lin-CR87}). It is not hard to see that, when $\eps$ is sufficiently large, $E^{GL}_\eps$ is strictly convex and so $u_\eps$ is the unique bounded critical point of $E^{GL}_\eps$ in $\mcA^{GL}$ for every $N\geq 2$ (see e.g. \cite{BBH_book} or \cite[Remark 3.3]{INSZ_AnnENS}). In dimension $N = 2$, Pacard and Rivi\`ere showed in \cite{Pacard_Riviere} that, for small $\eps > 0$, $u_\eps$ is the unique critical point of $E^{GL}_\eps$ in $\mcA^{GL}$; however, whether $u_\eps$ is the unique minimizer of $E_\eps^{GL}$ for all $\eps > 0$ remains an open question. In dimensions $N \geq 7$, it was shown in a recent work of Ignat, Nguyen, Slastikov and Zarnescu \cite{INSZ18_CRAS} that $u_\eps$ is the unique minimizer of $E^{GL}_\eps$ in $\mcA^{GL}$ \emph{for every} $\eps > 0$. It is not known whether $u_\eps$ minimizes $E^{GL}_\eps$ in $\mcA^{GL}$ in dimensions $3 \leq N \leq 6$ when $\eps$ is small.

A different way to regularize the singular harmonic map $n$ is to add an $(N+1)$-st direction in the target space while keeping the constraint of unit length and minimize
\[
E^{MM}_\eta [m] = \int_{B^N} \Big[\frac{1}{2} |\nabla m|^2 + \frac{1}{2\eta^2} \tilde W(m_{N+1}^2)\Big]\,dx
\]
where $\eta > 0$, $\tilde W(t)=t$ and $m$ belongs to
\[
\mcA^{MM} = \{m \in H^1(B^N,\Sphere^N): m(x) = (x,0) \text{ on } \partial B^N\}.
\]
This model comes from micromagnetics, where the order parameter $m$ stands for the magnetization in ferromagnetic materials.\footnote{In fact, in a reduced micromagnetic model in dimension $N=2$ (see e.g. \cite[Section 4.5]{DKMO02-CPAM} or \cite[Section~7]{Ignat09-SEDP}) and after a rotation by $\frac{\pi}{2}$ in the first two components, the condition $\nabla \times (m_1,m_2)=0$ is also imposed in the space of admissible configurations in $\mcA^{MM}$. Note that the vortex solution $m_\eta$ in \eqref{Eq:metaform} satisfies the above curl-free condition and we will prove its local minimality in the larger class of $H^1_0$ perturbations (that are not necessarily curl-free in the in-plane components). See also \cite{Gioia-James} for a different thin-film regime where this curl-free constraint on $(m_1, m_2)$ can be neglected.
}
Considering radially symmetric critical points of $E^{MM}_\eta$ over $\mcA^{MM}$, one is led to (see Appendix \ref{App})
\begin{equation}
m_{\eta}(x) =( \tilde f_\eta(r) n(x), g_\eta(r))\in \mcA^{MM}
	\label{Eq:metaform}
\end{equation}
where the radial profiles $\tilde f_\eta$ and $g_\eta$ satisfy 
\be
\label{f2g2}
\tilde f_\eta^2+g_\eta^2=1 \quad \textrm{in} \quad (0,1),
\ee
and the system of ODEs:
\begin{align}
\tilde f_{\eta}'' + \frac{N-1}{r} \tilde f_{\eta}' - \frac{N-1}{r^2} \tilde f_{\eta}
	&=- \lambda(r) \tilde f_{\eta}  \quad \textrm{in} \quad (0,1),
	\label{Eq:MM-fee}\\
g_{\eta}'' + \frac{N-1}{r} g_{\eta}' 
	&=\frac1{\eta^2}\tilde W'(g_\eta^2) g_\eta- \lambda(r) g_{\eta}  \quad \textrm{in} \quad (0,1),
	\label{Eq:MM-gee}\\
\tilde f_{\eta}(1) &= 1 \text{ and } g_{\eta}(1) = 0,
	\label{Eq:MM-feegeeBC}
\end{align}
where 
\begin{equation}
\lambda(r)=(\tilde f_{\eta}')^2+\frac{N-1}{r^2}\tilde f_{\eta}^2+(g_{\eta}')^2+\frac1{\eta^2}\tilde W'(g_\eta^2)g_{\eta}^2
\label{Eq:lamDef}
\end{equation} is the Lagrange multiplier due to the unit length constraint in $\mcA^{MM}$.

\begin{remark}\label{Rem:R1}
We will see in Lemma \ref{Lem:MMONSym} that solutions to \eqref{Eq:metaform}--\eqref{Eq:MM-feegeeBC} satisfy the dichotomy: either $\tilde f_\eta(0) = 0$ or $\tilde f_\eta(0) = 1$. Furthermore, in the latter case, it holds that $N \geq 3$ and $(\tilde f_\eta=1, g_\eta=0)$ in $(0,1)$, which corresponds to the \emph{equator map}
$$
\nespm(x):=(n(x),0).
$$
In dimension $N \geq 7$, $\nespm$ is the unique minimizing harmonic map from $B^N$ into $\Sphere^N$ in $\mcA^{MM}$ (J\"ager and Kaul \cite{JagerKaul83-JRAM}; see also \cite[Example 1.6]{INSZ_AnnENS}), and so is the unique minimizer of $E^{MM}_\eta$ in $\mcA^{MM}$ \emph{for every} $\eta>0$.
\end{remark}

We will focus in the following on ``escaping" solutions $m_{\eta}(x) =( \tilde f_\eta(r) n(x), \pm g_\eta(r))$ satisfying $g_\eta>0$ in $(0,1)$ which exist \emph{only} in dimension $2\leq N\leq 6$ (see Theorem \ref{Thm:MMExist}). More precisely, 
we will show in these dimensions that, \emph{for every} $\eta>0$, there exists a unique solution $(\tilde f_\eta, g_\eta)$ with $g_\eta>0$ in $(0,1)$ of the system \eqref{f2g2}--\eqref{Eq:MM-feegeeBC} and we call the two configurations $m_{\eta}=( \tilde f_\eta(r) n(x), \pm g_\eta(r))\in \mcA^{MM}$ the \emph{escaping ($\Sphere^N$-valued) Ginzburg--Landau vortex solutions}, or simply the \emph{micromagnetic  vortex solutions}. In addition, the micromagnetic vortex solutions $m_\eta$ have lower energy than the equator map; in particular, the equator map is \emph{no longer} a minimizer of $E^{MM}_\eta$ in $\mcA^{MM}$ (see Proposition \ref{Prop:MMUniqueness}). It is not known whether the micromagnetic vortex solutions $m_\eta$ minimize $E^{MM}_\eta$ in $\mcA^{MM}$ in dimension $2 \leq N \leq 6$.

The goal of this paper is to study the local minimality of the vortex solutions $u_\eps$ and $m_\eta$ with respect to $E^{GL}_\eps$ over the set $\mcA^{GL}$ and $E^{MM}_\eta$ over the set $\mcA^{MM}$ respectively. We will in fact consider $C^2$ potentials $W: (-\infty,1] \rightarrow [0,\infty)$ and $\tilde W: [0,\infty) \rightarrow [0,\infty)$ more general than the ones described above. We make the following assumptions: 
\begin{align}
&W(0) =  0, W(t)\geq 0, W''(t) \geq 0 \text{ in } (-\infty,1]\setminus\{0\},
	\label{Eq:WCond'}\\
& \tilde W(0) = 0, \tilde W(t)\geq 0, \tilde W''(t) \geq 0 \text{ in } (0,\infty).
	\label{Eq:TWCond'}
\end{align}
We point out that \eqref{Eq:WCond'} implies that $W'(0) = 0$ and $tW'(t) \geq 0$ in $(-\infty,1] \setminus \{0\}$. Likewise, \eqref{Eq:TWCond'} implies that $\tilde W'(0) \geq 0$ and $\tilde W'(t) \geq 0$ in $(0,\infty)$. However, we allow the possibility that $W$ or $\tilde W$ are zero in a neighborhood of the origin. This leads to new difficulties as well as new behaviors of solutions; see for example Proposition \ref{Prop:B.1}(ii). 

Under assumptions \eqref{Eq:WCond'} and \eqref{Eq:TWCond'} for $W$ and $\tilde W$, we will prove the existence and uniqueness of the radial profiles $f_\eps$ and $(\tilde f_\eta, g_\eta)$  with $g_\eta > 0$ solving \eqref{Eq:ueH1}--\eqref{Eq:24VII18-X2} and \eqref{Eq:metaform}--\eqref{Eq:MM-feegeeBC}, respectively. See Theorems \ref{Thm:GLExist} and \ref{Thm:MMExist} where the global minimality of these solutions in the class of radial symmetric maps is also established. For these unique radial profiles, we will continue to refer to the maps $u_\eps(x) = f_\eps(|x|) n(x)$ and $m_{\eta}(x) =( \tilde f_\eta(r) n(x), g_\eta(r))$ as the $\RR^N$-valued and $\Sphere^N$-valued Ginzburg--Landau vortex solutions. Our main results concern the local minimizing property of these vortex solutions, in particular the positive definiteness of the second variation at those solutions (see 
Section~\ref{Sec:Stab} for the definition).

\begin{theorem}\label{Thm:GLStab}
Suppose $W \in C^2((-\infty,1])$ satisfies \eqref{Eq:WCond'}. For $N \geq 2$ and every $\eps > 0$,  the $\RR^N$-valued Ginzburg--Landau vortex solution $u_\eps(x) = f_\eps(r) n(x)$ is a local minimizer of $E^{GL}_\eps$ in $\mcA^{GL}$ with a positive definite second variation.
\end{theorem}

\begin{theorem}\label{Thm:MStab}
Suppose $\tilde W \in C^2([0,\infty))$ satisfies \eqref{Eq:TWCond'}. For $2\leq N \leq 6$ and every $\eta > 0$, the escaping $\Sphere^N$-valued Ginzburg--Landau vortex solution $m_{\eta}(x) =( \tilde f_\eta(r) n(x), g_\eta(r))$ with $g_\eta > 0$ is a local minimizer $m_\eta$ of $E^{MM}_\eta$ in $\mcA^{MM}$ with a positive definite second variation. For $3 \leq N \leq 6$ and every $\eta > 0$, the equator map $\nespm=(n(x),0)$ is an unstable critical point of $E^{MM}_\eta$ in $\mcA^{MM}$ and $E^{MM}_\eta(m_\eta)<E^{MM}_\eta(\nespm)$.
\end{theorem}

\begin{remark}
\it{(a)} In Theorem \ref{Thm:MStab}, we can replace \eqref{Eq:TWCond'} by $\tilde W \in C^2([0,1])$ satisfying
\[
 \tilde W(0) = 0, \tilde W(t)\geq 0, \tilde W''(t) \geq 0 \text{ in } [0,1],
\]
since any such function $\tilde W$ can be extended to a function satisfying \eqref{Eq:TWCond'}.

\it{(b)} In dimension $N=2$, the equator map $\nespm\notin H^1(B^N, \Sphere^N)$, so $\nespm\notin \mcA^{MM}$. However, the second variation of $E^{MM}_\eta$ at $\nespm$ can still be defined and it is negative in a certain direction compactly supported in $B^N\setminus \{0\}$, leading to the instability of $\nespm$ also for $N=2$ (see \eqref{new-pp}). 
\end{remark}

In the $\RR^N$-valued Ginzburg--Landau case, when $N = 2$, Theorem \ref{Thm:GLStab} was proved by Mironescu \cite{Mironescu-radial} for $W(t) = \frac{t^2}{2}$. Also when $N = 2$, the non-negativity of the second variation was proved by Lieb and Loss \cite{LiebLoss95-JEDP} for potentials $W$ which are strictly increasing and convex\footnote{See Remark \ref{Rem:WWLocalCond} for a related comment for $E_{\eps,\eta}$.} in $[0,1]$. In dimension $N \geq 7$, the global minimality of the vortex solution was proved by Ignat, Nguyen, Slastikov and Zarnescu \cite{INSZ18_CRAS, INSZ_AnnENS}. When the domain is $\RR^N$ (instead of $B^N$), the local minimality of the entire vortex solution (in the sense of De Giorgi) was obtained in Mironescu \cite{Mironescu_symmetry} for $N = 2$, Millot and Pisante \cite{Mil-Pis} for $N = 3$, and Pisante \cite{Pisante11-JFA} for $N \geq 4$. For the stability of the entire vortex solution, see Ovchinnikov and Sigal \cite{OvchinSigal95}, del Pino, Felmer and Kowalczyk \cite{Pino-Felmer-Kow} for $N = 2$, and  Gustafson \cite{Gustafson} for  $N \geq 3$. 

In the micromagnetic case, in dimension $N = 2$ and for $\tilde W(t) = t$,  Theorem \ref{Thm:MStab} was proved by Hang and Lin \cite{HangLin01-ActaSin}. For dimension $N \geq 7$, see Remark \ref{Rem:R1}. See also Li and Melcher \cite{LiMelcher18-JFA} for related stability analysis in the study of micromagnetics skyrmions.

More generally, we consider a family of extended energy functionals $E_{\eps,\eta}$ depending on two positive parameters $\eps, \eta$ of which $E_\eps^{GL}$ and $E_\eta^{MM}$ are limiting cases: 
\[
E_{\eps,\eta}[m]= \int_{B^N} \Big[\frac{1}{2}|\nabla m|^2 + \frac{1}{2\eps^2} W(1 - |m|^2)  + \frac{1}{2\eta^2} \tilde W(m_{N+1}^2)\Big]\,dx, \quad \eps,\eta > 0,
\]
where $W$ and $\tilde W$ satisfy \eqref{Eq:WCond'}--\eqref{Eq:TWCond'} and $m$ belongs to
\[
\mcA = \{m \in H^1(B^N,\RR^{N+1}): m(x) = (x,0) \text{ on } \partial B^N\}.
\]
Under suitable conditions on $\tilde W$ (e.g. $\tilde W(t) > 0$ for $t > 0$), it can be shown that for a fixed $\eps > 0$, minimizers of $E_{\eps,\eta}$ in $\mcA$ converge in $H^1$ to minimizers of $E_{\eps}^{GL}$ in $\mcA^{GL}$ as $\eta \rightarrow 0$. Likewise under suitable conditions on $W$, for a fixed $\eta > 0$, minimizers of $E_{\eps,\eta}$ in $\mcA$ converge in $H^1$ to minimizers of $E_{\eta}^{MM}$ in $\mcA^{MM}$ as $\eps \rightarrow 0$. We hope that having a good understanding on critical points of $E_{\eps,\eta}$ would lead to new insights on the open problem concerning of the minimality of the vortex solutions $u_\eps$ and $m_\eta$.

We define a map $m_{crit} \in \mcA$ to be a \emph{bounded} critical point of $E_{\eps,\eta}$ if $m_{crit} \in L^\infty(B^N, \RR^{N+1})$ and $\langle DE_{\eps,\eta}[m_{crit}], \varphi\rangle := \frac{d}{dt}\big|_{t = 0} E_{\eps,\eta}[m_{crit} + t\varphi] = 0$ for all $\varphi \in C_c^\infty(B^N \setminus \{0\}, \RR^{N+1})$, and to be a \emph{radially symmetric} critical point of $E_{\eps,\eta}$ if $m_{crit}$ is radially symmetric in the sense of Definition \ref{Def:RSMap} and $\langle DE_{\eps,\eta}[m_{crit}], \varphi\rangle = 0$ for all $\varphi \in C_c^\infty(B^N \setminus \{0\}, \RR^{N+1})$. By Lemma \ref{lem:f2g2}, radially symmetric critical points of $E_{\eps,\eta}$ are bounded. Radially symmetric critical points of $E_{\eps,\eta}$ in $\mcA$ take the form 
\begin{equation}
(f_{\eps,\eta}(r)n(x),g_{\eps,\eta}(r)) \in \mcA
	\label{Eq:feegeeH1}
\end{equation}
where $(f_{\eps,\eta},g_{\eps,\eta})$ satisfies the system of  ODEs
\begin{align}
f_{\eps,\eta}'' + \frac{N-1}{r} f_{\eps,\eta}' - \frac{N-1}{r^2} f_{\eps,\eta}
	&= -\frac{1}{\eps^2} W'(1 -  f_{\eps,\eta}^2 - g_{\eps,\eta}^2) f_{\eps,\eta}
	,\label{Eq:20III21-fee}\\
g_{\eps,\eta}'' + \frac{N-1}{r} g_{\eps,\eta}' 
	&= -\frac{1}{\eps^2} W'(1 -  f_{\eps,\eta}^2 - g_{\eps,\eta}^2) g_{\eps,\eta} + \frac{1}{\eta^2} \tilde W'(g_{\eps,\eta}^2)g_{\eps,\eta}
	,\label{Eq:20III21-gee}\\
 f_{\eps,\eta}(1) &= 1  \text{ and } g_{\eps,\eta}(1) = 0
	.\label{Eq:20III21-feegeeBC}
\end{align}
Note that the above implies $f_{\eps,\eta}(0) = 0$ and $g_{\eps,\eta}'(0) = 0$ (see Lemma \ref{Lem:FullONSym}).

Of special interest to our discussion will be solutions to \eqref{Eq:feegeeH1}--\eqref{Eq:20III21-feegeeBC} satisfying the sign constraint $g_{\eps,\eta} \geq 0$ in $(0,1)$. It is easy to see by the strong maximum principle that either $g_{\eps,\eta} \equiv 0$ or $g_{\eps,\eta} > 0$ in $(0,1)$. When $g_{\eps,\eta} \equiv 0$, we obtain an $\eta$-independent solution given by $(\barfeps, 0)$ where $\barfeps$ is the unique radial profile in \eqref{Eq:ueH1}--\eqref{Eq:24VII18-X2}. We will sometimes refer to $(\barfeps, 0)$ as the \emph{non-escaping solution} to \eqref{Eq:feegeeH1}--\eqref{Eq:20III21-feegeeBC} and 
\[
\bar m_\eps(x) = (\barfeps(r)n(x),0)
\]
as the \emph{non-escaping (radially symmetric) critical point} of the extended energy functional $E_{\eps,\eta}$ in $\mcA$. In contrast, we will refer to solutions $(f_{\eps,\eta}, g_{\eps,\eta})$ of \eqref{Eq:feegeeH1}--\eqref{Eq:20III21-feegeeBC} satisfying $g_{\eps,\eta} > 0$ as \emph{escaping solutions} and the corresponding maps
\[
m_{\eps,\eta}(x) = (f_{\eps,\eta}(r)n(x),\pm g_{\eps,\eta}(r)) 
\]
as\footnote{In the following, when discussing escaping and non-escaping critical points, we will drop the term ``radially symmetric" as we only study here radially symmetric critical points.} \emph{escaping (radially symmetric) critical points} of the extended energy functional $E_{\eps,\eta}$ in $\mcA$. The escaping phenomenon refers to the positivity of $g_{\eps,\eta}$. We will prove that such escaping solutions satisfy $f_{\eps,\eta}>0$ in $(0,1)$, see Proposition \ref{Prop:fPos}.

There exists a sufficiently large $\eps_*$ such that $E_{\eps,\eta}$ is strictly convex for all $\eps > \eps_*$ and $\eta > 0$  and so $\bar m_\eps$ is the unique critical point and hence the unique global minimizer of $E_{\eps,\eta}$ in $\mcA$ if $N\geq 2$. In dimensions $N \geq 7$, it follows from \cite[Theorem 2]{INSZ18_CRAS}\footnote{In the cited paper, beside the convexity of $W$, it is assumed that $W$ is strictly positive away from $0$; but it can be seen from the proof there that non-negativity $W\geq 0$ is sufficient as in \eqref{Eq:WCond'}.} (compare \cite[Theorem~1.7]{INSZ_AnnENS}) that $\bar m_\eps(x)$ is the unique global minimizer of $E_{\eps,\eta}$ in $\mcA$ for every $\eps>0$. In dimension $2\leq N\leq 6$ and for small $\eps>0$, it is not known if a solution to \eqref{Eq:feegeeH1}--\eqref{Eq:20III21-feegeeBC} satisfying $g_{\eps,\eta} \geq 0$ gives a global minimizer of  $E_{\eps,\eta}$ in $\mcA$. Our next theorem concerns the existence, uniqueness and local minimality of these solutions. See Figure \ref{Fig1}.

\begin{figure}[h]
\caption{Radial critical points of the extended functional $E_{\eps, \eta}$ when $W'(1) > 0$ and $\tilde W'(0) > 0$. In the escaping region, there is a co-existence of non-escaping and escaping critical points. In the non-escaping region, only the non-escaping critical point exists.}\label{Fig1}
\begin{center}
\begin{tikzpicture}
\draw (-.3,0)--(4,0);
\draw (4,-.2) node {$\eps$};
\draw (0,-.3)--(0,4);
\draw (-.3,4) node {$\eta$};
\draw[dashed] (3,-.3)--(3,3.75);
\draw (3.3,-.3) node {$\eps_*$};
\draw[dashed] (2,-.3)--(2,3.75);
\draw (2.3,-.3) node {$\eps_0$};

\begin{scope}
\clip (0,0) rectangle (3,3.75);
\draw [domain=0:2, samples = 25, variable=\t, dashed, pattern = north west lines] plot ({\t},{\t/(2.3-\t)}) -- (0,4);
\end{scope}
\draw[pattern = north west lines] (5,3) circle (.3);
\draw (7.2,3) node {Escaping region};
\draw (5,1.5) circle (.3);
\draw (7.2,1.5) node {Non-escaping region};
\end{tikzpicture}
\end{center}
\end{figure}
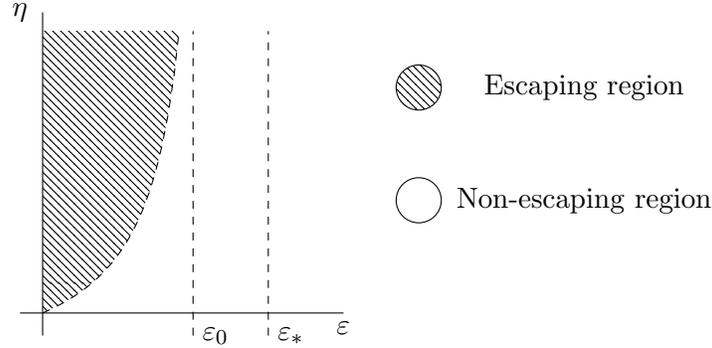

\begin{theorem}
\label{Thm:ExtendedMain}
Let $N \geq 2$, $W \in C^2((-\infty,1])$ and $\tilde W \in C^2([0,\infty))$ satisfy \eqref{Eq:WCond'} and \eqref{Eq:TWCond'}.
\begin{enumerate}[(a)]
\item There is at most one escaping critical point $m_{\eps,\eta}(x) = (f_{\eps,\eta}(r)n(x), g_{\eps,\eta}(r))$ of $E_{\eps,\eta}$ in $\mcA$ with $g_{\eps,\eta} > 0$. Moreover, if such escaping critical point exists, then it is a local minimizer of $E_{\eps,\eta}$ in $\mcA$ with a positive definite second variation, and the non-escaping critical point $\bar m_\eps(x) = (f_\eps(r) n(x),0)$ is unstable for $E_{\eps,\eta}$.

\item An escaping critical point $m_{\eps,\eta}(x) = (f_{\eps,\eta}(r)n(x), g_{\eps,\eta}(r))$ with $g_{\eps,\eta} > 0$ exists if and only if $2 \leq N \leq 6$, $W'(1) > 0$, $0 < \eps < \eps_0$ and $\eta > \eta_0(\eps)$ for some $\eps_0 \in (0,\infty)$ and a continuous non-decreasing function\footnote{For further information about the constant $\eps_0$ and the function $\eta_0$, see Lemma \ref{Lem:eigen-gl}(c) and Remark \ref{Rem:eta0}.} $\eta_0: [0,\eps_0) \rightarrow [0,\infty)$ with $\eta_0(0) = 0$.

\item In the absence of an escaping critical point $m_{\eps,\eta}(x) = (f_{\eps,\eta}(r)n(x), g_{\eps,\eta}(r))$ with $g_{\eps,\eta} > 0$ for $E_{\eps,\eta}$, the non-escaping critical point $\bar m_\eps(x) = (f_\eps(r) n(x),0)$  is a local minimizer of $E_{\eps,\eta}$ in $\mcA$ with a positive definite second variation unless $2 \leq N \leq 6$, $W'(1) > 0$, $\tilde W'(0) > 0$, $0 < \eps < \eps_0$ and $\eta = \eta_0(\eps)$. Moreover, in the latter case, the second variation of $E_{\eps,\eta}$ at $\bar m_\eps$ is non-negative semi-definite with a one-dimensional kernel generated by $(0,q_{\eps})\in C^2(\bar B^N, \R^{N+1})$ for some positive  smooth function $q_{\eps}>0$ in $B^N$ with $q_{\eps} = 0$ on $\partial B^N$.
\end{enumerate}
\end{theorem}

A main part of our paper concerns the local minimality of vortex solutions. Let us explain our strategy for the Ginzburg--Landau model. We establish
$$E_\eps^{GL}[u_\eps + v] \geq E_\eps^{GL}[u_\eps]+ c\|v\|^2_{H^1} \textrm{ for } u_\eps + v \in \mcA^{GL}, \|v\|_{H^1}<\delta,$$
 for some small $c>0$ and $\delta>0$. This draws on a careful study of the second variation of $E_\eps^{GL}$ at $u_\eps$ based on a separation of variables and a Hardy decomposition technique \cite{INSZ3}. To separate variables, we first decompose $v = s n + w$ where $w \cdot n =0$, and then, for each $0 < r < 1$, we use the Helmholtz decomposition to write $w = \mathring{w} + \slashed{D} \psi$ on $\partial B_r$ where $\mathring{w}$ is a divergence-free vector field on $\partial B_r$ and $\slashed{D}$ is the gradient operator. In the context of Ginzburg--Landau theory, our use of the Helmholtz decomposition appears new in dimension $N \geq 3$. The contribution of $\mathring{w}$ to the second variation is treated at once using the sharp Poincar\'e inequality in Appendix \ref{App:C} and the Hardy decomposition technique. Finally, we decompose $s$ and $\psi$ into spherical harmonics and treat them using again the Hardy decomposition technique with special choices of factoring functions.

An important point in proving our results resides in the analysis of the radial profiles $f_\eps$, $(\tilde f_\eta, g_\eta)$ and $(f_{\eps, \eta}, g_{\eps, \eta})$ for general potentials $W$ and $\tilde W$ that goes beyond the existing (very rich) literature. For example, the choice of factoring functions in our use of the Hardy decomposition technique is based on the positivity and monotonicity of (a-priori, nodal solutions) $f_\eps$, $\tilde f_\eta$ and $f_{\eps, \eta}$. The proof of these uses the moving plane method for cooperative systems \cite{Sirakov07, G-N-N-1979, Se}. A novel part of our argument is in the fact that cooperativity is obtained alongside the application of the moving plane method. Another issue is the uniqueness of the radial profiles, which is established again using the Hardy decomposition technique that handles the nonlinear part in the ODE. This analysis enables us to prove the dichotomy of escaping vs. non-escaping critical points in the extended model introduced here for the first time.

The rest of the paper is organized as follows. In Section \ref{Sec:EU}, we establish the existence and uniqueness of vortex radial profiles and discuss their minimality within radially symmetric configurations. In Section \ref{Sec:Stab} we analyze their stability and give the proof of the main theorems. We include also four appendices on some miscellaneous results.

\medskip
\noindent{\bf Acknowledgment.} R.I. thanks the Mathematical Institute and St Edmund Hall, University of Oxford and L.N. thanks the Institut de Math\'ematiques de Toulouse where part of this work was done.

\section{Existence and uniqueness of vortex radial profiles}\label{Sec:EU}

We study existence and uniqueness properties of radially symmetric critical points of $E_\eps^{GL}$, $E_\eta^{MM}$ and $E_{\eps,\eta}$. We define the following reduced energy functionals relevant in the discussion of radially symmetric critical points in $\mcA^{GL}$, $\mcA^{MM}$ and $\mcA$ (see Appendix \ref{App}).
\begin{itemize}
\item The reduced $\RR^N$-valued Ginzburg--Landau functional
\begin{align*}
I_\eps^{GL}[f] 
	&= \frac{1}{|\Sphere^{N-1}|} E_{\eps}^{GL}[f(|x|)n(x)] 
	=\frac{1}{2}\int_0^1 \Big[(f')^2 + \frac{N-1}{r^2} f^2 + \frac{1}{\eps^2} W(1 - f^2)\Big]r^{N-1}\,dr
\end{align*}
where $f$ belongs to
\[
\mcB^{GL} = \Big\{ f: r^{\frac{N-1}{2}}f', r^{\frac{N-3}{2}}f \in L^2(0,1), f(1) = 1\Big\}.
\]

\item The reduced $\Sphere^N$-valued Ginzburg--Landau functional:
\begin{align*}
I_\eta^{MM}[f,g]
	&= \frac{1}{|\Sphere^{N-1}|} E_{\eta}^{MM}[(f(r) n(x), g(r))]\\
	&=\frac12\int_0^1 \Big[(f')^2 + (g')^2 + \frac{N-1}{r^2} f^2 + \frac{1}{\eta^2} \tilde W(g^2)\Big]\,r^{N-1}\,dr,
\end{align*}
where $(f,g)$ belongs to
\begin{align*}
\mcB^{MM}= \Big\{(f,g):~ 
		&r^{\frac{N-1}{2}}f', r^{\frac{N-3}{2}} f, r^{\frac{N-1}{2}}g', r^{\frac{N-1}{2}}g \in L^2(0,1),\\
		& f^2 + g^2 = 1, f(1) = 1, g(1) = 0\Big\}.
\end{align*}

\item The reduced extended functional
\begin{align*}
I_{\eps,\eta}[f,g] 
	&= \frac{1}{|\Sphere^{N-1}|} E_{\eps,\eta}[(f(r) n(x), g(r))]\\
	&= \frac{1}{2}\int_0^1 \Big[(f')^2 + (g')^2 + \frac{N-1}{r^2} f^2 + \frac{1}{\eps^2} W(1 - f^2 - g^2) + \frac{1}{\eta^2} \tilde W(g^2)\Big]\,r^{N-1}\,dr
\end{align*}
where $(f,g)$ belongs to
\[
\mcB = \Big\{(f,g): r^{\frac{N-1}{2}}f', r^{\frac{N-3}{2}} f, r^{\frac{N-1}{2}}g', r^{\frac{N-1}{2}}g \in L^2(0,1), f(1) = 1, g(1) = 0\Big\}.
\]

\end{itemize}

Note that $(f,g) \in \mcB$ is equivalent to  $m(x)=(f(r) n(x), g(r))\in H^1(B^N, \RR^{N+1})$ with $m(x) =(x,0)$ on $\partial B^N$, and
\[
\int_{B^N} |\nabla m|^2\,dx = |\Sphere^{N-1}|\int_0^1\Big[(f')^2 + (g')^2 + \frac{N-1}{r^2} f^2\Big]\, r^{N-1}\,dr.
\]
It is straightforward to check that bounded critical points of $I_\eps^{GL}$, $I_\eta^{MM}$ and $I_{\eps,\eta}$ correspond to bounded radially symmetric critical points of $E_\eps^{GL}$, $E_\eta^{MM}$ and $E_{\eps,\eta}$, respectively.\footnote{In this radially symmetric setting, when $W$ and $\tilde W$ satisfy \eqref{Eq:WCond'} and \eqref{Eq:TWCond'}, the boundedness assumption on critical points can be dropped, in view of Lemma \ref{lem:f2g2}.}

\medskip

\noindent\underline{The $\RR^N$-valued Ginzburg--Landau model}

\begin{theorem}\label{Thm:GLExist}
Let $N \geq 2$ and suppose that $W \in C^2((-\infty,1])$ satisfies $W(0) = 0$ and $W \geq 0$. Then, for every $\eps > 0$, \eqref{Eq:24VII18-X1}--\eqref{Eq:24VII18-X2} has a solution $f_\eps$ such that $\frac{f_\eps}{r} \in C^2([0,1])$, $0 < f_\eps < 1$ in $(0,1)$, and $f_\eps(0)=0$. If, in addition, $W$ satisfies \eqref{Eq:WCond'}, then $f'_\eps>0$ in $(0,1]$ and $f_\eps$ is the unique solution to \eqref{Eq:ueH1}--\eqref{Eq:24VII18-X2}; in particular, $f_\eps$ is the unique minimizer of $I_\eps^{GL}$ in $\mcB^{GL}$.
\end{theorem}

\begin{remark}
The existence and uniqueness of the vortex radial profile for the $\RR^N$-valued Ginzburg--Landau model has been studied by many authors. Closely related to our result above is a result in \cite{INSZ18_CRAS} which gives the uniqueness in dimensions $N \geq 7$. Earlier results in \cite{aguareles-baldoma, tang, farina-guedda, Hervex2, ODE_INSZ} are for all dimensions $N \geq 2$ but assume the inequality $W''(0) > 0$, while Theorem \ref{Thm:GLExist} above allows the case $W''(0) = 0$.
\end{remark}

Let $f_\eps$ be the radial profile in Theorem \ref{Thm:GLExist}. Note that $(f_\eps,0)$ is the non-escaping critical point for the extended functional $I_{\eps,\eta}$ for any $\eta > 0$. For the existence of escaping solutions in the extended model, we give an estimate on the first eigenvalue $\ell(\eps)$ of
\begin{equation}
L^{GL}_\eps = -\Delta - \frac{1}{\eps^2}W'(1 - f_\eps^2)
	\label{Eq:LeGLDef}
\end{equation}
in $B^N$ with respect to the zero Dirichlet boundary condition. Note that since the potential $\frac{1}{\eps^2}W'(1 - f_\eps^2)$ is radially symmetric, any first eigenfunction of $L_{\eps}^{GL}$ is also radially symmetric. It is clear that, under \eqref{Eq:WCond'}, we have $\ell(\eps) >  -W'(1) \eps^{-2}$ for every $\eps>0$.

\begin{lemma}\label{Lem:eigen-gl}
Suppose $W \in C^2((-\infty,1])$ satisfies \eqref{Eq:WCond'}. Then $\ell$ is a continuous function of $\eps$  satisfying
\begin{equation}
\eps^2 \ell(\eps) > \tilde \eps^2 \ell(\tilde \eps) \text{ for all } 0 < \tilde \eps < \eps < \infty,
	\label{Eq:ellMon}
\end{equation}
and the following estimates hold.

\begin{enumerate}[(a)]

\item If $W'(1) = 0$, then $W = 0$ in $(0,1)$, $L_\eps^{GL} = -\Delta$ and
\[
\ell(\eps) = \lambda_1(-\Delta) >  0 \text{ for all } \eps > 0,
\]
where $\lambda_1(-\Delta)$ is the first eigenvalue of the Laplacian on $B^N$ with respect to the zero Dirichlet boundary value.

\item If $N\geq 7$, 
\[
\ell(\eps)\geq \frac{(N-2)^2}4-(N-1)>0 \text{ for all } \eps > 0.
\]

\item If $2 \leq N \leq 6$ and $W'(1) > 0$, then there exists $\eps_0 \in (0,\infty)$ such that $\ell(\eps) < 0$ and increasing in $(0,\eps_0)$, $\ell(\eps_0) = 0$ and $\ell(\eps) > 0$ in $(\eps_0,\infty)$. Furthermore, for some $\eps_1 \in (0,\eps_0)$ and $c_1 \in (0,W'(1))$, 
\[
-\frac{W'(1)}{\eps^2} <  \ell(\eps) \leq -\frac{c_1}{\eps^2} \text{ for } \eps \in (0,\eps_1).
\]

\end{enumerate}
\end{lemma}

\noindent\underline{The extended model}

We are now in position to give a necessary and sufficient condition for the existence of an escaping solution of \eqref{Eq:feegeeH1}--\eqref{Eq:20III21-feegeeBC}. For an illustration see Figure \ref{Fig1}.

\begin{theorem}\label{Thm:ExtendedExist}
Suppose $W \in C^2((-\infty,1])$ and $\tilde W \in C^2([0,\infty))$ satisfy \eqref{Eq:WCond'} and \eqref{Eq:TWCond'}.
\begin{enumerate}[(a)]
\item If $N \geq 7$ or $W'(1) = 0$, then for every $\eps, \eta>0$, \eqref{Eq:feegeeH1}--\eqref{Eq:20III21-feegeeBC} has no solution $(f_{\eps,\eta},g_{\eps,\eta} )$ which satisfies $g_{\eps,\eta} > 0$ in $(0,1)$. Moreover, the non-escaping solution $(f_\eps,0)$ is the unique minimizer of $I_{\eps,\eta}$ in $\mcB$.

\item Suppose $2 \leq N \leq 6$, $W'(1) > 0$. Let $\eps_0 \in (0,\infty)$ be as in Lemma \ref{Lem:eigen-gl} and define
\begin{align*}
\eta_0(\eps) 
	&= \sqrt{\frac{\tilde W'(0)}{|\ell(\eps)|}} \in [0,\infty) \text{ for } \eps \in (0,\eps_0).
\end{align*}

\begin{enumerate}[({\theenumi}1)]

\item The system \eqref{Eq:feegeeH1}--\eqref{Eq:20III21-feegeeBC} has an escaping solution $(f_{\eps,\eta},g_{\eps,\eta} )$ which satisfies $g_{\eps,\eta} > 0$ in $(0,1)$ if and only if $0 < \eps < \eps_0$ and $\eta > \eta_0(\eps)$. In this case, it is the unique escaping solution of \eqref{Eq:feegeeH1}--\eqref{Eq:20III21-feegeeBC}, $\frac{f_{\eps,\eta}}{r}, g_{\eps,\eta} \in C^2([0,1])$, $f_{\eps,\eta}^2 + g_{\eps,\eta}^2 < 1$, $f_{\eps,\eta} > 0$, $f_{\eps,\eta}' > 0$, $g_{\eps,\eta}' < 0$ in $(0,1)$, and there are exactly two minimizers of $I_{\eps,\eta}$ in $\mcB$ given by $(f_{\eps,\eta},\pm g_{\eps,\eta} )$. 

\item If $\eps \geq \eps_0$ or $0 < \eta \leq \eta_0(\eps)$, the non-escaping solution $(f_\eps,0)$ of \eqref{Eq:feegeeH1}--\eqref{Eq:20III21-feegeeBC} is the unique minimizer of $I_{\eps,\eta}$ in $\mcB$. Otherwise (i.e. $0 < \eps < \eps_0$ and $\eta > \eta_0(\eps)$), the non-escaping solution $(f_\eps,0)$ of \eqref{Eq:feegeeH1}--\eqref{Eq:20III21-feegeeBC} is an unstable critical point of $I_{\eps,\eta}$ in $\mcB$.
\end{enumerate}

\end{enumerate}
\end{theorem}

We note that if $2 \leq N \leq 6$, $W'(1) > 0$ and $\tilde W'(0) = 0$, then $\eta_0(\eps) = 0$ for all $\eps \in (0,\eps_0)$.  In this case, the theorem asserts for all $\eta>0$, an escaping solution of  \eqref{Eq:feegeeH1}--\eqref{Eq:20III21-feegeeBC} exists if and only if $\eps \in (0,\eps_0)$.

\begin{remark}\label{Rem:eta0}
By Lemma \ref{Lem:eigen-gl}, when $2 \leq N \leq 6$, $W'(1) > 0$ and $\tilde W'(0) > 0$, the function $\eta_0$ defined in Theorem \ref{Thm:ExtendedExist}(b) belongs to $C([0,\eps_0))$, $\frac{\eta_0(\eps)}{\eps}$ is increasing with respect to $\eps$, 
\[
\lim_{\eps \rightarrow \eps_0} \eta_0(\eps) = \infty, \quad \lim_{\eps \rightarrow 0} \eta_0(0) = 0,
\]
and, for some $C > 1$ and $\eps_1 \in (0,\eps_0)$, $\frac{\sqrt{\tilde W'(0)} \eps}{C} \leq \eta_0(\eps) \leq C\sqrt{\tilde W'(0)} \eps$ for every $\eps \in (0,\eps_1)$.
\end{remark}

Theorem \ref{Thm:ExtendedExist} can be viewed as an extension of the results in \cite{INSZ_AnnENS} but within radial symmetry, relating the escaping phenomenon with the stability property of critical points.
 
\bigskip

\noindent\underline{The $\Sphere^N$-valued Ginzburg--Landau model}

\begin{theorem}\label{Thm:MMExist}
Suppose that $\tilde W \in C^2([0, \infty))$ satisfies \eqref{Eq:TWCond'}. 
\begin{enumerate}[(a)]
\item If $N \geq 7$, then for every $\eta>0$, the system \eqref{Eq:metaform}--\eqref{Eq:MM-feegeeBC} has no escaping solution $(\tilde f_\eta, g_\eta)$ with $g_\eta > 0$ in $(0,1)$. 
\item If $2 \leq N \leq 6$, then for every $\eta > 0$ the system \eqref{Eq:metaform}--\eqref{Eq:MM-feegeeBC} has a unique escaping solution $(\tilde f_\eta, g_\eta)$ with $g_\eta > 0$.  Furthermore, $(\tilde f_\eta, \pm g_\eta)$ are the only two minimizers of the functional $I_{\eta}^{MM}$ in $\mcB^{MM}$, $\frac{\tilde f_\eta}{r}, g_\eta \in C^2([0,1])$, $\tilde f_\eta > 0$, $\tilde f_\eta' > 0$ and $g_\eta' < 0$ in $(0,1)$. In addition, for $3 \leq N \leq 6$, the non-escaping solution $(1,0)$ is an unstable critical point of $I_\eta^{MM}$ in $\mcB^{MM}$.
\end{enumerate}
\end{theorem}

Recall that, when $N \geq 7$, the non-escaping solution $(1,0)$ is the unique minimizer of $I_{\eta}^{MM}$ in $\mcB^{MM}$ for every $\eta>0$ (see Remark \ref{Rem:R1}). Note that when $N=2$, the non-escaping solution $(1,0)\notin \mcB^{MM}$; however, the second variation of $I_{\eta}^{MM}$ at $(1,0)$ can still be defined and it is negative in a certain direction with compact support in the interval $(0,1)$, leading to the instability of the non-escaping solution $(1,0)$ also for $N=2$ (see \eqref{new-pp}).

The rest of the section is organized as follows. In Subsection \ref{SSec:ExtUniq}, for the extended model, we prove the monotonicity (see Proposition \ref{Prop:ODEMonotonicity}) and uniqueness (see 
Proposition~\ref{pro:unique_eps_eta}) of escaping solutions \eqref{Eq:feegeeH1}--\eqref{Eq:20III21-feegeeBC}, if exist, together with the positivity of $f_{\eps, \eta}$ in Proposition~\ref{Prop:fPos}; we also prove the boundedness of arbitrary solutions to \eqref{Eq:feegeeH1}--\eqref{Eq:20III21-feegeeBC}, see Lemma~\ref{lem:f2g2}. In Subsection \ref{SSec:GLEU}, for the $\RR^N$-valued GL model, we give the proof of Theorem \ref{Thm:GLExist} and Lemma \ref{Lem:eigen-gl}. In Subsection \ref{SSec:ExtE}, we give the proof of Theorem \ref{Thm:ExtendedExist} for the extended. Finally, Theorem \ref{Thm:MMExist} for the $\Sphere^N$-valued GL model is proved in Subsection \ref{SSec:MMEU}.

\subsection{The extended model: Monotonicity and uniqueness}\label{SSec:ExtUniq}

In this subsection we establish the monotonicity and the uniqueness of escaping radially symmetric critical points of the extended functional $E_{\eps,\eta}$, which correspond to escaping solutions $(f_{\eps,\eta} , g_{\eps,\eta} )$ with $g_{\eps,\eta} > 0$ of the ODE system \eqref{Eq:feegeeH1}--\eqref{Eq:20III21-feegeeBC}. Furthermore, we show that $f_{\eps,\eta} > 0$ and prove the minimality of this escaping solution with respect to radially symmetric competitors.

\bigskip

The following lemma shows that every solution to \eqref{Eq:feegeeH1}--\eqref{Eq:20III21-feegeeBC} is bounded in $(0,1)$ under conditions \eqref{Eq:WCond'}--\eqref{Eq:TWCond'}. To dispel confusion, in this result, we do not assume a priori the boundedness nor the non-negativity of $f_{\eps,\eta}$ and $g_{\eps,\eta}$.

\begin{lemma}\label{lem:f2g2}
Let $N \geq 2$, $\eps > 0$ and $\eta > 0$. If $W \in C^2((-\infty,1])$ and $\tilde W \in C^2([0,\infty))$ satisfy \eqref{Eq:WCond'}--\eqref{Eq:TWCond'} and $(f_{\eps,\eta},g_{\eps,\eta})$ satisfies \eqref{Eq:feegeeH1}--\eqref{Eq:20III21-feegeeBC}, then $f_{\eps,\eta}^2 + g_{\eps,\eta}^2 < 1$ in $(0,1)$ and the map $x \mapsto m_{\eps,\eta}(x) = (f_{\eps,\eta}(r)n(x), g_{\eps,\eta}(r))$ belongs to $C^2(\bar B^N)$. In particular, $f_{\eps,\eta}(0) = 0$ and $g_{\eps,\eta}'(0) = 0$.
\end{lemma}

\begin{proof}
Note that $m_{\eps,\eta} \in H^1(B^N)$ (since $(f_{\eps,\eta},g_{\eps,\eta}) \in \mcB$) and that \eqref{Eq:20III21-fee}--\eqref{Eq:20III21-feegeeBC} gives
\begin{align}
\Delta m_{\eps,\eta} \label{equ_m} 
	&= -\frac{1}{\eps^2} W'(1 - |m_{\eps,\eta}|^2) m_{\eps,\eta} + \frac{1}{\eta^2} \tilde W'(g_{\eps,\eta}^2) g_{\eps,\eta} e_{N+1} \text{ in } B^N \setminus \{0\},\\
m_{\eps,\eta}(x)
	\nonumber&= (n(x),0) \text{ on } \partial B^N.
\end{align}
Let $M = f_{\eps,\eta}^2 + g_{\eps,\eta}^2$. Note that $M(1)=1$ and
\begin{align*}
\frac{1}{2}(M'' + \frac{N-1}{r}M')
	&= (f_{\eps,\eta}')^2 + (g_{\eps,\eta}')^2 + \frac{N-1}{r^2} f_{\eps,\eta}^2 
		- \frac{1}{\eps^2} W'(1 - M)M+ \frac{1}{\eta^2} \tilde W'(g_{\eps,\eta}^2) g_{\eps,\eta}^2\nonumber\\
	&\geq - \frac{1}{\eps^2} W'(1 - M) M.
\end{align*}
In particular, the function $X = 1 - M$ satisfies
\begin{equation}
- X'' - \frac{N-1}r X' + 2 a(r) X \geq 0
	\label{Eq:05V21-M1}
\end{equation}
where $a: (0,1] \rightarrow [0,\infty)$ is given by
\begin{equation}
a(r) = \left\{\begin{array}{ll}
	 \frac{1}{\eps^2} \frac{W'(1 - M(r))}{1 - M(r)} M(r) & \text{ if } M(r) \neq 1,\\
	\frac{1}{\eps^2} W''(0) &\text{ if } M(r) = 1.
\end{array}\right.
	\label{Eq:M1Aux-a1}
\end{equation}
Note that \eqref{Eq:WCond'} and the continuity of $M$ in  $(0,1]$ imply $a\geq 0$ and $a$ is continuous on $(0,1]$. 
Now, define
\[
r_0 = \inf\Big\{r \in (0,1]: M \leq 1 \text{ in } [r,1]\Big\}.
\]
The aim is to show that $r_0=0$.

\medskip

\noindent
\underline{Step 1:} {\it We show that if $r_0 > 0$, then $M > 1$ in $(0,r_0)$.}
Assume by contradiction that $M(r_1) \leq 1$ for some $r_1 \in (0,r_0)$. Multiplying \eqref{Eq:05V21-M1} by $r^{N-1} X^-$ (where $X^\pm = \max\{0,\pm X\}$), noting that $X^-(1) = X^-(r_1) = 0$, and integrating over $[r_1,1]$ give
\[
\int_{r_1}^1 r^{N-1}\big[((X^-)')^2 + 2a(r) (X^-)^2\big]\,dr \leq 0.
\]
This shows that $X^- = 0$ in $[r_1,1]$, i.e. $X \geq 0$ and $M \leq 1$ in $[r_1,1]$. By definition of $r_0$, this implies that $r_0 \leq r_1$, which contradicts the fact that $r_1 \in (0,r_0)$. Step 1 is established.

\medskip
\noindent\underline{Step 2:} {\it We show that $f_{\eps,\eta}^2+g_{\eps,\eta}^2\leq 1$ in $(0,1)$.}
Indeed,  if $r_0 = 0$, this step is clear. Suppose that $r_0 > 0$. By Step 1, we have $M > 1$ and so $W'(1 - M) \leq 0$ in $(0,r_0)$. Returning to \eqref{Eq:20III21-fee}--\eqref{Eq:20III21-gee}, as \eqref{Eq:TWCond'} implies $\tilde W'(t)\geq \tilde W'(0)\geq 0$ for $t\geq 0$, we have that the functions $f_{\eps,\eta}$ and $g_{\eps,\eta}$, considered as functions on the ball $B(0,{r_0})$ in $\RR^N$, satisfy
\begin{align*}
\Delta f_{\eps,\eta} 
	&= c_1 f_{\eps,\eta} \text{ in } B(0,{r_0})\setminus \{0\},\\
\Delta g_{\eps,\eta} 
	&= c_2 g_{\eps,\eta} \text{ in } B(0,{r_0})\setminus \{0\},
\end{align*}
where $c_1 = \frac{N-1}{r^2} - \frac{1}{\eps^2} W'(1 - M) \geq 0$ and $c_2 = - \frac{1}{\eps^2} W'(1 - M) + \frac{1}{\eta^2} \tilde W'(g_{\eps,\eta}^2) \geq 0$ in $(0,{r_0})$. By Kato's inequality (see \cite{Kato72-IJM} or \cite[Lemma A.1]{Brezis84-AMO}), this implies
\begin{align*}
\Delta f_{\eps,\eta}^\pm 
	&\geq 0 \text{ in } B(0,{r_0})\setminus \{0\},\\
\Delta g_{\eps,\eta}^\pm
	&\geq 0 \text{ in } B(0,{r_0})\setminus \{0\}.
\end{align*}
Since $f_{\eps,\eta}, g_{\eps,\eta} \in H^1(B(0,{r_0}))$, these hold in $B(0,{r_0})$. By the maximum principle,
\[
f_{\eps,\eta}^\pm \leq f_{\eps,\eta}^\pm(r_0) \text{ and } 
g_{\eps,\eta}^\pm \leq g_{\eps,\eta}^\pm(r_0)  \text{ in } B(0,{r_0}).
\]
We deduce that $f^2_{\eps,\eta}+ g^2_{\eps,\eta}\leq M(r_0)\leq 1$ in $(0,{r_0})$. As $M=f^2_{\eps,\eta}+ g^2_{\eps,\eta}\leq 1$ in $[r_0, 1]$, the conclusion of Step 2 follows.

\medskip
\noindent\underline{Step 3:} {\it Conclusion.} 
By Step 2 and the fact that $m_\eta\in H^1(B^N)$, we deduce that \eqref{equ_m} holds in the whole $B^N$;  then standard elliptic regularity theory yields $m_{\eps,\eta}$ and so $X$ are $C^2$ in $\bar B^N$. In particular, $f_{\eps,\eta}(0) = 0$  (as $f_{\eps,\eta}(r) n(x)\in C^2(B^N)$) and $g_{\eps,\eta}'(0) = 0$ (since $g_{\eps,\eta}$ extends to an even $C^2$ function on $(-1,1)$). By Step 2, we know that $M \leq 1$ in $(0,1)$. Moreover, since $f_{\eps, \eta}(1) = 1$, we deduce that the inequality in \eqref{Eq:05V21-M1} is strict near $r = 1$, in particular, $X$ cannot be identically $0$. Thus, the strong maximum principle applied to \eqref{Eq:05V21-M1} yields $X>0$ in $(0,1)$ i.e. $M < 1$ in $(0,1)$. The conclusion follows.
\end{proof}

By restricting attention to solutions with $g_{\eps,\eta} \equiv 0$ (for any $\tilde W$ satisfying \eqref{Eq:TWCond'} e.g. $\tilde W(t) = t$), we immediately obtain:

\begin{corollary}\label{Cor:fFinite}
Let $N \geq 2$ and $\eps > 0$. If $W \in C^2((-\infty,1])$ satisfies \eqref{Eq:WCond'} and $f_{\eps}$ satisfies \eqref{Eq:ueH1}--\eqref{Eq:24VII18-X2}, then $|f_{\eps}| < 1$ in $(0,1)$ and the map $x \mapsto u_{\eps}(x) = f_{\eps}(r)n(x)$ belongs to $C^2(\bar B^N)$. In particular, $f_{\eps}(0) = 0$.
\end{corollary}

We next consider the monotonicity of solutions of \eqref{Eq:feegeeH1}--\eqref{Eq:20III21-feegeeBC} satisfying $g_{\eps,\eta} \geq 0$. We first prove the monotonicity under an additional assumption that $f_{\eps,\eta} \geq 0$ in 
Proposition~\ref{Prop:ODEMonotonicity}. We then show that this additional non-negativity assumption on $f_{\eps,\eta}$ can be removed in Proposition \ref{Prop:fPos}.

\begin{proposition}\label{Prop:ODEMonotonicity}
Suppose $W \in C^2((-\infty,1])$ and $\tilde W \in C^2([0,\infty))$ satisfy \eqref{Eq:WCond'} and \eqref{Eq:TWCond'}, and $(f_{\eps,\eta},g_{\eps,\eta})$ satisfies \eqref{Eq:feegeeH1}--\eqref{Eq:20III21-feegeeBC} with $f_{\eps,\eta} \geq 0, g_{\eps,\eta}\geq 0$ in $(0,1)$. Then $f'_{\eps,\eta}>0$, $\big(\frac{f_{\eps,\eta}}{r}\big)' \leq 0$ and either $g'_{\eps,\eta} < 0$ or $g_{\eps,\eta} =0$ in $(0,1]$. 
\end{proposition}

\begin{proof}[Proof of Proposition \ref{Prop:ODEMonotonicity}.] To simplify notation, we drop off the indices $\eps$ and $\eta$, so that in the following we denote  $f$ and $g$ the solution considered in \eqref{Eq:feegeeH1}--\eqref{Eq:20III21-feegeeBC}. 
First, by Lemma~\ref{lem:f2g2}, we know that $f^2+ g^2<1$ in $(0,1)$ and $f(0)=0$ and $g'(0)=0$. By the strong maximum principle applied to \eqref{Eq:20III21-fee} for $f\geq 0$ in $(0,1)$, we get $f>0$ in $(0,1)$ (as $f=0$ in $(0,1)$ would contradict the boundary condition 
$f(1)=1$ in \eqref{Eq:20III21-feegeeBC}). By the strong maximum principle applied to \eqref{Eq:20III21-gee} (as a PDE in $B^N$ for $g\geq 0$) we get $g>0$ in $[0,1)$ or $g=0$ in $(0,1)$.

\medskip

\nd {\it Case 1: $g>0$ in $[0,1)$}. For $a,b\in [0,1]$, let
\begin{equation}
A(a,b)=-\frac1{\eps^2} W'(1-a^2-b^2)a, \quad B(a,b)=-\frac1{\eps^2} W'(1-a^2-b^2)b+
\frac1{\eta^2} \tilde W'(b^2)b.
	\label{Eq:abDef}
\end{equation}
Then \eqref{Eq:20III21-fee} and \eqref{Eq:20III21-gee} rewrite as
\begin{align}
\label{eq:44}
f'' + \frac{N-1}{r} f' - \frac{N-1}{r^2} f&=A(f, g) \quad \textrm{in } (0,1),\\
\label{eq:55}
g'' + \frac{N-1}{r} g' &=B(f, g) \quad \textrm{in } (0,1).
\end{align}
The convexity assumption on $W$ in \eqref{Eq:WCond'} yields $$\partial_b A(a,b) = \partial_a B(a,b) \geq 0 \text{ for all } a, b\in [0,1].$$
These inequalities give the system \eqref{eq:44}--\eqref{eq:55} a cooperative structure, see e.g. \cite{Sirakov07, G-N-N-1979, Se}. In order to prove the monotonicity of $f$ and $g$, we follow the ideas based on a moving plane argument in the proof of \cite[Theorem 1.6]{INSZ_CVPDE}. See also \cite{AFN21-CPDE} for a similar argument in the context of phase segregation in Bose--Einstein condensates.
For $0 < s < 1$, define
\[
f_s(r) = f(2s - r) \text{ and } g_s(r) = g(2s - r) \text{ for } \max(0,2s - 1) < r < s.
\]
By \eqref{Eq:20III21-feegeeBC} and \eqref{Eq:WCond'} (in particular, $W'(0)=0$), we have $A(f(1),g(1)) = B(f(1),g(1)) = 0$ and recall that $0 < f < 1=f(1)$ and $g >0=g(1)$ in $(0,1)$. Combined with the monotonicity of $A(a,b)$ in $b$, we deduce that the function $\hat f = f - f(1)$ satisfies
\begin{align*}
\hat f'' + \frac{N-1}{r} \hat f' - \frac{N-1}{r^2}\,\hat f &= \frac{N-1}{r^2} f(1) + A(f,g) - A(f(1),g(1))\\ 
& \geq A(f,g) - A(f(1),g)= c(r)\,\hat f
\end{align*}
for some continuous function $c \in C[0,1]$. As $\hat f(1) = 0$ and $\hat f < 0$ in $(0,1)$, we deduce from the Hopf lemma (see e.g. \cite[Lemma 3.4]{GT}) that $f'(1) > 0$. Likewise, we can show that $g'(1) < 0$. Consequently, there is some small $\delta > 0$ such that $f_s > f$ and $g_s < g$ in $\max(0,2s - 1) < r < s$ for any $s\in (1 - \delta, 1)$. 
We define
\[
\underline{s} = \inf\Big\{0 < s < 1 \, : \, f_{t} > f \text{ and }g_{t} < g \text{ in }\max(0,2{t} - 1) < r < t \text{ for all } t \in (s,1)\Big\}.
\]
It follows that $\underline{s} \in [0,1-\delta]$. 

\medskip

\noindent\underline{Claim:} $\underline{s}=0$, $f' > 0$ and $g' < 0$ in $(0,1]$. 

\nd \underline{Proof of Claim}: Assume by contradiction that $\underline{s}>0$. By the definition of $\underline s$, we deduce
\begin{enumerate}[(a)]
\item $f' \geq 0$ and $g' \leq 0$ in $(\underline{s},1)$,
\item and $f_{\underline{s}} \geq f > 0$ and $g_{\underline{s}} \leq g $ in $\max(0,2{\underline{s}} - 1) < r < \underline{s}$.
\end{enumerate}
Combined with the monotonicity of $A(a, \cdot)$ and $B(\cdot, b)$, it follows for 
every $s\in [\underline{s},1)$ and every $r\in (\max(0,2s - 1),  s)$:
\begin{align}
\nonumber f_{s}''(r) + &\frac{N-1}{r}f_{s}'(r) - \frac{N-1}{r^2}\,f_{s}(r)
	=f''(2s-r) - \frac{N-1}{r}f'(2s-r) - \frac{N-1}{r^2}\, 
	f(2s-r)\\
	\label{Eq:789} &\leq A(f(2s-r),g(2s-r))=A(f_{s}(r),g_{s}(r)) \leq A(f_{s}(r),g(r)),\\
& \quad \quad \quad g_{s}''(r) + \frac{N-1}{r}g_{s}'(r)
	\geq B(f_{s}(r),g_{s}(r)) \geq B(f(r), g_{s}(r))\label{Eq:567}
\end{align}
 and equality in all the inequalities \eqref{Eq:789} (resp. in \eqref{Eq:567}) for some $s\in [\underline{s},1)$ implies 
 \be
 \label{new_23}
 \textrm{$f'(2s-r)=0$ (resp. $g'(2s-r)=0$) for every $r\in (\max(0,2s - 1),  s)$.} 
 \ee
 Combining \eqref{Eq:789} and \eqref{Eq:567} with \eqref{eq:44} and \eqref{eq:55}, we obtain
for every $s\in [\underline{s},1)$:
\begin{align*}
(f_{s}-f)'' + \frac{N-1}{r}(f_{s}-f)' - \frac{N-1}{r^2}\,(f_{s}-f)
	&\leq A(f_{s},g) - A(f,g)=(f_{s}-f)c_1(r),\\
(g_{s}-g)'' + \frac{N-1}{r}(g_{s}-g)'
	&\geq B(f,g_{s}) - B(f,g)=(g_{s}-g)c_2(r),
\end{align*}
with $c_1, c_2$ being two continuous functions on $[\max(0,2s - 1), s]$ and equality in the above inequalities implies again \eqref{new_23}.

Recall that, by the definition of $\underline{s}$, $f_s>f$ and $g_s < g$ in $(\max(0,2s - 1), s)$ for $s \in (\underline{s},1)$. By the Hopf lemma, applied to the above differential inequalities, we have $f_s'(s) < f'(s)$ and $g_s'(s) > g'(s)$, i.e. $f'(s) > 0$ and $g'(s) < 0$ for $s \in (\underline{s},1)$. We now show that these assertions continue to hold with $s = \underline{s}$, i.e.

\smallskip

\nd \underline{Fact 1:} $f_{\underline{s}} > f$ and $g_{\underline{s}} < g$ in $\max(0,2{\underline{s}} - 1) < r < \underline{s}$.
\smallskip

\nd \underline{Fact 2:} $f'>0$ and $g'<0$ in $[\underline{s},1)$.

\smallskip

\noindent Indeed, since $f' > 0$ and $g' < 0$ in $(\underline{s},1)$, \eqref{new_23} does not hold and so the above differential inequalities for $f_{\underline{s}} - f$ and $g_{\underline{s}}-g$ are strict in $(\max(0,2\underline{s}-1),\underline{s})$. Since $f_{\underline{s}} - f \geq 0$ and $g_{\underline{s}}-g \leq 0$ in $(\max(0,2\underline{s}-1),\underline{s})$, the strong maximum principle applied to those differential inequalities gives Fact 1. By the Hopf lemma, we then have $f_{\underline{s}}'(\underline{s}) < f'(\underline{s})$ and $g_{\underline{s}}'(\underline{s}) > g'(\underline{s})$, i.e. $f'(\underline{s}) > 0$ and $g'(\underline{s}) < 0$, and Fact 2 follows.

\smallskip

\nd \underline{Conclusion:} We now show that Facts 1 and 2 contradict the minimality of $\underline{s}$. 
Indeed, observe first that $(f_{\underline{s}}-f)\big(\max(0,2{\underline{s}} - 1)\big)>0$ since
\begin{align*}
&f_{\underline{s}}(\max(0,2\underline{s}-1)) = 1 > f(\max(0,2\underline{s}-1)) \text{ when } \frac{1}{2} \leq \underline{s} < 1,\\
&f_{\underline{s}}(\max(0,2\underline{s}-1)) > 0 = f(\max(0,2\underline{s}-1)) \text{ when } \underline{s} < \frac{1}{2}.
\end{align*}
Likewise, we have $(g_{\underline{s}}-g)\big(\max(0,2{\underline{s}} - 1)\big)<0$ since
\begin{align*}
&g_{\underline{s}}(\max(0,2\underline{s}-1)) = 0 < g(\max(0,2\underline{s}-1)) \text{ when } \frac{1}{2} \leq \underline{s} < 1,\\
&g_{\underline{s}}'(\max(0,2\underline{s}-1)) = -g'(2\underline{s}) > 0 = g'(0) = g'(\max(0,2\underline{s}-1)) \text{ when } \underline{s} < \frac{1}{2}
\end{align*}
(in the latter case, this is combined with $g_{\underline{s}}<g$ on $(0,{\underline{s}})$ by Fact 1). Thus, thanks to Facts 1 and 2, we deduce by continuity the existence of a small $\tilde \delta>0$ such that, for every $s\in (\underline{s}-\tilde \delta, \underline{s}]$, $f_s > f$ and $g_s < g$ in $\max(0,2s - 1) < r < s$, contradicting the minimality of $\underline{s}$. Thus, $\underline{s}=0$. Also, by Fact 2, $f'>0$ and $g'<0$ in $(0,1]$. The Claim is proved.

\medskip

\nd {\it Case 2: $g=0$ in $(0,1)$}. The above argument applies to solutions $f\geq 0$ of \eqref{Eq:ueH1}--\eqref{Eq:24VII18-X2}, where equation \eqref{Eq:789} is replaced by
\[
f_{s}''(r) + \frac{N-1}{r}f_{s}'(r) - \frac{N-1}{r^2}\,f_{s}(r)
	 \leq A(f_s(r),0),
\]
yielding $f'>0$. (Note that the assumption $W''\geq 0$ is no longer needed in this case, though the condition $W'(0) = 0$ is used.)

\medskip

\nd \underline{Proof of $\big(\frac{f}{r}\big)' \leq 0$ in $(0,1)$.} Indeed, by Lemma \ref{Lem:FullONSym}, we know that $v:=\frac{f}{r}\in C^2([0,1])$. To prove that $v$ is decreasing, we follow the argument in \cite[Proposition 2.2]{ODE_INSZ}:
by \eqref{Eq:WCond'} we have $W'\geq 0$ in $(0,1)$ so that
\[
(r^{N+1} v'(r))' = - \frac{r^{N+1}}{\eps^2}W'(1 - f^2 -g^2) v(r)\leq 0, \quad r\in (0,1).
\]
This implies that $r^{N+1} v'(r)$ is a nonincreasing $C^1$ function in $[0,1]$. Since $\lim_{r \rightarrow 0} r^{N+1} v'(r) = 0$ (as $v\in C^1([0,1])$), we deduce that $v'(r) \leq 0$ in $[0,1]$. 
\end{proof}

Next, we prove the positivity of $f_{\eps,\eta}$ when $g_{\eps,\eta}\geq 0$. When $g_{\eps,\eta} \equiv 0$, the result was obtained in \cite{Hervex2, INSZ3} under some slightly different condition on $W$.

\begin{proposition}\label{Prop:fPos}
Suppose $W \in C^2((-\infty,1])$ and $\tilde W \in C^2([0,\infty))$ satisfy \eqref{Eq:WCond'} and \eqref{Eq:TWCond'}, and $(f_{\eps,\eta},g_{\eps,\eta})$ satisfies \eqref{Eq:feegeeH1}--\eqref{Eq:20III21-feegeeBC} with $g_{\eps,\eta}\geq 0$ in $(0,1)$. Then $f_{\eps,\eta}>0$ in $(0,1)$.
\end{proposition}

\begin{proof} As in the proof of the previous proposition, we drop off the indices $\eps$ and $\eta$, so that in the following we denote  $f$ and $g$ the solution considered in \eqref{Eq:feegeeH1}--\eqref{Eq:20III21-feegeeBC}. Suppose by contradiction that $f$ changes sign in $(0,1)$. Let $r_1 \in (0,1)$ be such that $f(r_1) = 0$ and $f > 0$ in $(r_1,1]$. Applying the Hopf lemma to \eqref{Eq:20III21-fee} in $(r_1,1)$, we have $f'(r_1) > 0$. In particular, $f < 0$ in some small interval $(r_1 - \delta, r_1)$. Observe that $(|f|, g)$ satisfies in the sense of distribution
\begin{align*}
|f|'' + \frac{N-1}{r} |f|' - \frac{N-1}{r^2} |f|&= A(|f|, g) \quad \textrm{in } (r_1,1),\\
|f|'' + \frac{N-1}{r} |f|' - \frac{N-1}{r^2} |f|&\geq A(|f|, g) \quad \textrm{in } (0,1),\\
g'' + \frac{N-1}{r} g' &=B(|f|, g) \quad \textrm{in } (0,1),
\end{align*}
where $A$ and $B$ are defined in \eqref{Eq:abDef}. Consequently, we can apply the proof of 
Proposition~\ref{Prop:ODEMonotonicity} to the pair $(|f|, g)$ to obtain
\[
(|f|)_s \geq |f| \text{ and } g_s \leq g \text{ in } \max(0, 2s - 1) < r < s \text{ for all } r_1 \leq s < 1,
\]
where $(|f|)_s(r) = |f|(2s - r)$ and $g_s(r) = g(2s - r)$. Observe also that, by definition, both $|f|$ and $(|f|)_{r_1}$ have the same first left-derivative at $r_1$; thus, we deduce by the Hopf lemma that $(|f|)_{r_1} \equiv  |f|$ and $f'(2r_1 - r) = 0$  in $\max(0, 2r_1 - 1) < r < r_1$ (see \eqref{new_23}). The latter identity is impossible, since $f'(r_1) > 0$.
We conclude that $f \geq 0$ in $(0,1)$. The positivity of $f$ follows by the strong maximum principle applied to \eqref{Eq:20III21-fee} (as $f(1)=1$).
\end{proof}

Applying Propositions \ref{Prop:ODEMonotonicity} and \ref{Prop:fPos} to the non-escaping solution $(f_\eps,0)$, we  obtain:

\begin{corollary}\label{Prop:GLMonotonicity}
Suppose $W \in C^2((-\infty,1])$ satisfies \eqref{Eq:WCond'}, and $f_\eps$ satisfies \eqref{Eq:ueH1}--\eqref{Eq:24VII18-X2}. Then $f_\eps > 0$, $f'_{\eps}>0$ and $\big(\frac{f_{\eps}}{r}\big)' \leq 0$ in $(0,1]$.
\end{corollary}

Finally, we prove the uniqueness of escaping solutions of \eqref{Eq:feegeeH1}--\eqref{Eq:20III21-feegeeBC}.

\begin{proposition}
Let $N\geq 2$ and suppose that $W \in C^2((-\infty,1])$ and $\tilde W \in C^2([0,\infty))$ satisfy \eqref{Eq:WCond'} and \eqref{Eq:TWCond'}. Then, for every $\eps > 0$ and $\eta > 0$, the system \eqref{Eq:feegeeH1}--\eqref{Eq:20III21-feegeeBC} has at most one escaping solution $(f_{\eps,\eta} ,g_{\eps,\eta} )$ with $g_{\eps,\eta} > 0$ in $(0,1)$. Furthermore, when it exists, $(f_{\eps,\eta}, \pm g_{\eps,\eta})$ are the only two minimizers of $I_{\eps,\eta}$ over the set $\mcB$; in particular, $I_{\eps,\eta}[f_\eps, 0]>I_{\eps,\eta}[f_{\eps,\eta}, g_{\eps,\eta}]$ where $f_\eps$ is the radial profile satisfying \eqref{Eq:24VII18-X1}--\eqref{Eq:24VII18-X2}.
\label{pro:unique_eps_eta}
\end{proposition}

\begin{proof}
We use ideas from our previous papers \cite{INSZ18_CRAS, INSZ_AnnENS}. Suppose that $(f_{\eps,\eta},g_{\eps,\eta})$ solves \eqref{Eq:feegeeH1}--\eqref{Eq:20III21-feegeeBC} and $g_{\eps,\eta} > 0$ in $(0,1)$. By Proposition \ref{Prop:fPos}, $f_{\eps,\eta} > 0$ in $(0,1)$. For $(f,g) \in \mcB$, we write $(f,g) = (f_{\eps,\eta},g_{\eps,\eta}) + (s,q)$ and
\[
V(x)=(s(r)n(x), q(r))\in H^1_0(B^N, \RR^{N+1}).
\]
Using first the convexity of $W$ and $\tilde W$ and then equations \eqref{Eq:20III21-fee}--\eqref{Eq:20III21-gee}, we compute
\begin{align*}
&I_{\eps,\eta}[f,g] - I_{\eps,\eta}[f_{\eps,\eta},g_{\eps,\eta}]
\geq \frac{1}{2} \int_0^1 \Big\{ 2f_{\eps,\eta}' s' + (s')^2 + 2g_{\eps,\eta}' q' + (q')^2 + \frac{N-1}{r^2}(2f_{\eps,\eta}s + s^2)\\
		&\qquad - \frac{1}{\eps^2} W'(1 - f_{\eps,\eta}^2 - g_{\eps,\eta}^2)[ 2(f_{\eps,\eta}s + g_{\eps,\eta}q) + s^2 + q^2]
		 + \frac{1}{\eta^2} \tilde W'(g_{\eps,\eta}^2)(  2g_{\eps,\eta}q + q^2)  \Big\} r^{N-1} dr
\end{align*}
	\begin{align*}	 
	&= \frac{1}{2} \int_0^1 \Big\{ (s')^2  + (q')^2 + \frac{N-1}{r^2} s^2
		- \frac{1}{\eps^2} W'(1 - f_{\eps,\eta}^2 - g_{\eps,\eta}^2)( s^2 + q^2) + \frac{1}{\eta^2} \tilde W'(g_{\eps,\eta}^2) q^2  \Big\} r^{N-1} dr\\
	&=\frac1{2|\bS^{N-1}|} \int_{B^N}\Big\{ |\nabla V|^2- \frac{1}{\eps^2} W'(1 - f_{\eps,\eta}^2 - g_{\eps,\eta}^2)|V|^2 + \frac{1}{\eta^2} \tilde W'(g_{\eps,\eta}^2) V_{N+1}^2  \Big\}\,dx
	=: \frac{F_{\eps,\eta}[V]}{2|\bS^{N-1}|}.
\end{align*}

\noindent\underline{Claim 1:} For every $V(x)=(s(r)n(x), q(r))\in H^1_0(B^N, \RR^{N+1})$, it holds
\[
F_{\eps,\eta}[V] \geq \int_{B^N} \Big\{f_{\eps,\eta}^2(|x|) \Big| \Big(\frac{s}{f_{\eps,\eta}}\Big)'(|x|)\Big|^2 + g_{\eps,\eta}^2(|x|) \Big|\Big(\frac{q}{g_{\eps,\eta}}\Big)'(|x|)\Big|^2 \Big\}\,dx,
\]
and as a consequence, $(f_{\eps,\eta},g_{\eps,\eta})$ minimizes $I_{\eps,\eta}$ in $\mcB$.

\nd \underline{Proof of Claim 1:}
Since $F_{\eps,\eta}$ is continuous in $H_0^1(B^N, \RR^{N+1})$ (because $W'(1 - f_{\eps,\eta}^2 - g_{\eps,\eta}^2)$, $\tilde W'(g_{\eps,\eta}^2)\in L^\infty(B^N)$ by Lemma \ref{lem:f2g2}), by standard density results and Fatou's lemma, it suffices to show the claim for $V=(s(r)n, q(r)) \in C_c^\infty(B^N\setminus \{0\}, \RR^{N+1})$. For that, we will apply \cite[Lemma~A.1]{ODE_INSZ} for the operators 
\be
\label{opLT}
\begin{cases}
&L:=-\Delta-\frac{1}{\eps^2} W'(1 - f_{\eps,\eta}^2 - g_{\eps,\eta}^2),\\ 
&T:=-\Delta-\frac{1}{\eps^2} W'(1 - f_{\eps,\eta}^2 - g_{\eps,\eta}^2)+ \frac{1}{\eta^2} \tilde W'(g_{\eps,\eta}^2). 
\end{cases}
\ee
Indeed, writing $V=(s(r)n, q(r))=(V_1, \dots, V_N, V_{N+1})\in C_c^\infty(B^N\setminus \{0\}, \RR^{N+1})$ and decomposing $V_j=f_{\eps,\eta}\hat V_j$ with $\hat V_j= \frac{V_j}{f_{\eps,\eta}}$ for $j=1,\dots, N$ and $V_{N+1}=g_{\eps,\eta} \hat V_{N+1}$ with $\hat V_{N+1}=\frac{q}{g_{\eps,\eta}}$,
\begin{align}
F_{\eps,\eta}[V] 
	&= \sum_{j=1}^N \int_{B^N} \, L V_j\cdot V_j \, dx+\int_{B^N} \, T V_{N+1}\cdot  V_{N+1} \, dx 
	\nonumber\\
	&=\sum_{j=1}^N \int_{B^N} \, \Big\{f_{\eps,\eta}^2 |\nabla \hat V_j|^2 + \hat V_j^2 Lf_{\eps,\eta}\cdot f_{\eps,\eta}
	+g_{\eps,\eta}^2 |\nabla \hat V_{N+1}|^2 + \hat V_{N+1}^2 Tg_{\eps,\eta}\cdot g_{\eps,\eta}\Big\}\, dx
		\nonumber \\
	&\nonumber=\int_{B^N} \,\Big\{ f_{\eps,\eta}^2(|x|) \Big|\nabla \big(\frac{s(r)}{f_{\eps,\eta}(r)} n(x)\big)\Big|^2 -\frac{N-1}{r^2}s^2+g_{\eps,\eta}^2(|x|) \Big|\Big(\frac{q}{g_{\eps,\eta}}\Big)'(|x|)\Big|^2 \Big\}\, dx\\
	&=\int_{B^N} \Big\{f_{\eps,\eta}^2(|x|) \Big| \Big(\frac{s}{f_{\eps,\eta}}\Big)'(|x|)\Big|^2 + g_{\eps,\eta}^2(|x|) \Big|\Big(\frac{q}{g_{\eps,\eta}}\Big)'(|x|)\Big|^2 \Big\}\,dx,
	\label{Eq:10IV21-E1}
\end{align}
because $Lf_{\eps,\eta}=-\frac{N-1}{r^2} f_{\eps,\eta}$, $Tg_{\eps,\eta}=0$ (by \eqref{Eq:20III21-fee}--\eqref{Eq:20III21-gee}) and $(\hat V_1, \dots, \hat V_N)=\frac{s(r)}{f_{\eps,\eta}(r)} n(x)$ with $|\nabla n|^2=\frac{N-1}{r^2}$. Hence, the claim is proved.

\medskip
\noindent\underline{Step 1:} {\it We prove that $\{(f_{\eps,\eta},\pm g_{\eps,\eta})\}$ is the set of minimizers of $I_{\eps,\eta}$ in $\mcB$.} Indeed, we have seen that $(f_{\eps,\eta},\pm g_{\eps,\eta})$ minimizes $I_{\eps,\eta}$ in $\mcB$. Suppose $(\tilde f_{\eps,\eta}, \tilde g_{\eps,\eta})$ also minimizes $I_{\eps,\eta}$ in $\mcB$, in particular, $I_{\eps,\eta}[f_{\eps,\eta},g_{\eps,\eta}]=I_{\eps,\eta}[\tilde f_{\eps,\eta}, \tilde g_{\eps,\eta}]$ so that, for $V=\Big((\tilde f_{\eps,\eta}  -  f_{\eps,\eta})n(x), \tilde g_{\eps,\eta}  -  g_{\eps,\eta}\Big)$, one has $F[V]=0$ leading to:
\[
\frac{\tilde f_{\eps,\eta}  -  f_{\eps,\eta} }{ f_{\eps,\eta} } 
	\text{ and }
	\frac{\tilde g_{\eps,\eta}  -  g_{\eps,\eta} }{ g_{\eps,\eta} }
	\text{ are constant in } (0,1).
\]
This together with $\tilde f_{\eps,\eta}(1) - f_{\eps,\eta}(1) = 0$ gives $\tilde f_{\eps,\eta}  \equiv f_{\eps,\eta}$ and 
$\tilde g_{\eps,\eta}  \equiv a g_{\eps,\eta}$ in $(0,1)$ for some constant $a \in \RR$. Since $g_{\eps,\eta} > 0$, this implies that $\tilde g_{\eps,\eta}$ has a fixed sign. Furthermore, either $a=0$ (so $\tilde g_{\eps,\eta}\equiv 0$), or $|\tilde g_{\eps,\eta}|>0$ in $(0,1)$ in which case, we can interchange $g_{\eps,\eta}$ and $\pm \tilde g_{\eps,\eta}$  if necessary (note that $(\tilde f_{\eps,\eta}, -\tilde g_{\eps,\eta})$ also minimizes $I_{\eps,\eta}$ in $\mcB$), so that we may always assume that $0 \leq a \leq 1$.

To finish the proof, we prove that $a=1$, i.e.,  $\tilde g_{\eps,\eta}  \equiv  g_{\eps,\eta}$ in $(0,1)$. Assume by contradiction that $0 \leq a<1$. We will show that
\begin{equation}
W'(1 - f_{\eps,\eta}^2 - g_{\eps,\eta}^2) \equiv 0 \text{ in }(0,1).
	\label{Eq:W'Red0}
\end{equation}
Once this is done, we deduce from \eqref{Eq:20III21-gee} that $-\Delta g_{\eps,\eta}+\frac1{\eta^2} \tilde W'(g^2_{\eps,\eta}) g_{\eps,\eta}=0$ in $B^N$. Since $\tilde W'\geq \tilde W'(0)\geq 0$ in $[0, \infty)$ (by \eqref{Eq:TWCond'}) and $g_{\eps,\eta}=0$ on $\partial B^N$, we deduce that $g_{\eps,\eta}=0$ in $B^N$ which gives a contradiction to the assumption $g_{\eps,\eta}>0$ in $B^N$, and completes the proof.

Let us now prove \eqref{Eq:W'Red0}. Returning to \eqref{Eq:20III21-fee}, we see that
\begin{equation}
W'(1 - f_{\eps,\eta}^2 - g_{\eps,\eta}^2) \equiv W'(1 - f_{\eps,\eta}^2 - a^2 g_{\eps,\eta}^2)  \quad \text{ in } [0,1].
	\label{Eq:Step1}
\end{equation}
Therefore, to prove \eqref{Eq:W'Red0}, it suffices to show that $W'(t)=0$ for every $0\leq t\leq \max_{[0,1]} (1-f_{\eps,\eta}^2 - a^2 g_{\eps,\eta}^2) =: \tau$. For that, we have $f_{\eps,\eta}^2 + a^2 g_{\eps,\eta}^2<f_{\eps,\eta}^2 + g_{\eps,\eta}^2<1$ in $(0,1)$ by Lemma \ref{lem:f2g2}, and hence $\tau >0$. Note that the range of $1-f_{\eps,\eta}^2 - a^2 g_{\eps,\eta}^2$ over $[0,1]$ is $[0,\tau]$ because of \eqref{Eq:20III21-feegeeBC}. Set $t_0=\inf \{t>0\, :\, W'(s)=W'(\tau) \textrm{ for all } s\in [t,\tau]\}$. We show that $t_0=0$. For that, let $r_0\in [0,1]$ such that 
$1-f_{\eps,\eta}^2(r_0) - a^2 g_{\eps,\eta}^2(r_0)=t_0$. By the continuity of $W'$ and \eqref{Eq:Step1}, we deduce for $t_1:=1-f_{\eps,\eta}^2(r_0) - g_{\eps,\eta}^2(r_0)\leq t_0$ that 
$W'(t_1)=W'(t_0)=W'(\tau)$. As $W'$ is nondecreasing (because $W$ is convex), we deduce that $W'(s)=W'(\tau)$ for every $s\in [t_1,\tau]$. By the minimality of $t_0$, it means that $t_1=t_0$, i.e., $g_{\eps,\eta}^2(r_0)=0$. Since $g_{\eps,\eta}>0$ in $[0,1)$ (which is a consequence of the strong maximum principle applied to \eqref{Eq:20III21-gee}, considered as a PDE on $B^N$), this yields $r_0=1$, i.e., $t_0=0$. It follows that $W'\equiv W'(0)=0$ on $[0, \tau]$ as desired (where we use that $0$ is a minimum point of $W$ by the assumption \eqref{Eq:WCond'}). 

\medskip
\noindent\underline{Step 2:} {\it We prove the uniqueness of escaping solutions of \eqref{Eq:feegeeH1}--\eqref{Eq:20III21-feegeeBC}.} Indeed, assume that $(\check f_{\eps,\eta},\check g_{\eps,\eta})$ is also a solution to \eqref{Eq:feegeeH1}--\eqref{Eq:20III21-feegeeBC} with $\check g_{\eps,\eta} > 0$ in $(0,1)$. Then Claim 1 yields that both $(f_{\eps,\eta},g_{\eps,\eta})$ and $(\check f_{\eps,\eta},\check g_{\eps,\eta})$ minimize $I_{\eps,\eta}$ in $\mcB$. By Step 1, we have $f_{\eps,\eta} \equiv \check f_{\eps,\eta}$ and $g_{\eps,\eta} \equiv \check g_{\eps,\eta}$ as desired. The proof is complete.
\end{proof}

\subsection{The $\RR^N$-valued GL model: Existence and uniqueness}\label{SSec:GLEU}

We prove existence and uniqueness of the radial profile and its minimality for $I_\eps^{GL}$ as stated in in Theorem \ref{Thm:GLExist}. Then we prove Lemma \ref{Lem:eigen-gl}.

\begin{proof}[Proof of Theorem \ref{Thm:GLExist}]
Let $f_\eps$ be a minimizer of the reduced energy functional $I_\eps^{GL}$ in $\mcB^{GL}$. (It is easy to see that such minimizer exists.) Since $I_\eps^{GL}[f] \geq I_\eps^{GL}[\min\{|f|,1\}]$, we may also assume that $0 \leq f_\eps \leq 1$. In addition, we have that $f_\eps$ satisfies \eqref{Eq:24VII18-X1}, $f_\eps(1) = 1$ and $f_\eps \in C^2((0,1])$. Noting also that the constant functions $0$ and $1$ are a solution and a super-solution to \eqref{Eq:24VII18-X1} respectively (since $W'(0) = 0$), the strong maximum principle implies that $0 < f_\eps < 1$ in $(0,1)$. By Lemma \ref{Lem:ONSym}, $f_\eps/r \in C^2([0,1])$, in particular, $f_\eps(0) = 0$.

If \eqref{Eq:WCond'} holds, then by Corollary \ref{Prop:GLMonotonicity} we have $f'_\eps> 0$ in $(0,1]$. Also, the same argument as in the proof of Proposition \ref{pro:unique_eps_eta} applies giving also the uniqueness of $f_\eps$ as solution of \eqref{Eq:24VII18-X1}-\eqref{Eq:24VII18-X2}, in particular, as unique minimizer of $I_\eps^{GL}$ over $\mcB^{GL}$. We omit the details.
\end{proof}

We next prove estimates for $\ell(\eps)$.

\begin{proof}[Proof of Lemma \ref{Lem:eigen-gl}.]
Note that by the definition of the first eigenvalue for $L_\eps^{GL}$ and standard elliptic regularity, $\ell$ depends continuously on $\eps$. Let us prove \eqref{Eq:ellMon} for $0<\tilde \eps<\eps<\infty$. We have
\[
\int_{B^N} \Big[|\nabla \varphi|^2 - \frac{1}{\tilde\eps^2} W'(1- f_{\tilde\eps}^2)\varphi^2\Big]\,dx \geq \ell(\tilde\eps) \int_{B^N} \varphi^2\,dx \text{ for all } \varphi \in H_0^1(B^N).
\]
By rescaling, we deduce:
\[
\int_{B(0,1/{\tilde\eps})} \Big[|\nabla \psi|^2 -  W'(1- f_{\tilde\eps}^2(\tilde\eps |x|))\psi^2\Big] \geq \tilde\eps^2 \ell(\tilde\eps)\int_{B(0,1/{\tilde\eps})} \psi^2\,dx  \text{ for all } \psi \in H_0^1(B(0,{1/{\tilde\eps}})).
\]
As $B(0,1/\eps) \subset B(0,1/{\tilde \eps})$, by the strict monotonicity of the first eigenvalue with respect to domains (due to the positivity of the first eigenfunctions), we have
\[
\int_{B(0,1/\eps)} \Big[|\nabla \psi|^2 -  W'(1- f_{\tilde\eps}^2(\tilde\eps |x|))\psi^2\Big] > \tilde\eps^2 \ell(\tilde\eps)\int_{B(0,1/\eps)} \psi^2\,dx  \text{ for all } 0 \not\equiv \psi \in H_0^1(B(0,{1/{\eps}})).
\]
Now using the inequality $1 \geq f_{\eps}(\eps |x|) \geq f_{\tilde\eps}(\tilde\eps |x|) \geq 0$ (see Proposition \ref{Prop:B.1}(a)) for $|x| < 1/\eps$ and the monotonicity of $W'$, we deduce that
\[
\int_{B(0,1/\eps)} \Big[|\nabla \psi|^2 -  W'(1- f_{\eps}^2(\eps |x|))\psi^2\Big] > \tilde\eps^2 \ell(\tilde\eps)\int_{B(0,1/\eps)} \psi^2\,dx  \text{ for all } 0 \not\equiv \psi \in H_0^1(B(0,{1/{\eps}})).
\]
Rescaling once again we get that
\[
\int_{B^N} \Big[|\nabla \varphi|^2 - \frac{1}{\eps^2} W'(1- f_\eps^2)\varphi^2\Big] > \frac{\tilde\eps^2 \ell(\tilde\eps)}{\eps^2} \int_{B^N} \varphi^2\,dx  \text{ for all } 0 \not\equiv \varphi \in H_0^1(B^N),
\]
which is equivalent to \eqref{Eq:ellMon}.

Assertion (a) is clear because if $W'(1) = 0$, then \eqref{Eq:WCond'} implies that $W = 0$ in $(0,1)$. 

Assertion (b) for $N\geq 7$ is a consequence of the inequality
\[
\int_{B^N} L_\eps^{GL} v \cdot v\,dx \geq \Big(\frac{(N-2)^2}4-(N-1)\Big) \int_{B^N} \frac{v^2}{r^2}\,dx \text{ for all } v \in H_0^1(B^N),
\]
which was proved in Step 4 of the proof of \cite[Theorem 2]{INSZ18_CRAS}.

We next prove assertion (c) for $2\leq N\leq 6$ and $W'(1)>0$. We have seen that $\ell(\eps) > -W'(1)\eps^{-2}$. We prove the rest in 2 steps.

\underline{Step 1:} {\it We show that there exist $\eps_1 > 0$ and $c_1 > 0$ such that $\ell(\eps) \leq -\frac{c_1}{\eps^2}$ for $\eps \in (0,\eps_1)$, by exhibiting a non-zero function $q = q_\eps(r) \in Lip_c((0,1)) $ satisfying }
\[ 
\int_{B^N} L_\eps^{GL} q \cdot q\,dx \leq -\frac{c_1}{\eps^2}\int_{B^N} q^2\,dx.
\]
(Note that by the lower bound of $\ell(\eps)$, it is clear that $c_1 < W'(1)$.)

Note that, by \cite[Lemma~A.1]{INSZ3}, for every positive function $\varphi \in C^{1,1}_{loc}((0,1))$, we have the following identity for every $q = f_\eps \varphi \tilde q \in Lip_c(B^N \setminus \{0\})$
\begin{align}
\int_{B^N} L_\eps^{GL} q \cdot q\,dx 
	&= \int_{B^N}\varphi^2 \Big\{ \barfeps^2 |\nabla \tilde q|^2 + \frac{L_{\eps}^{GL}(\f \barfeps) \barfeps}{\varphi}   \tilde q^2\Big\} \,dx.
	\label{Eq:2VIPNew}
\end{align}
We choose\footnote{See \cite[inequality (6)]{INSZ18_CRAS} for an explanation of this choice of $\varphi$.} $\varphi = r^{-\frac{N-2}{2}}\in C^\infty((0,1))$, and note that, by \eqref{Eq:24VII18-X1}, 
\[
L_\eps^{GL}(\f \barfeps) \barfeps=\frac{(N^2 - 8N + 8) f^2_\eps \f}{4r^2} 
	 -2 \barfeps  f'_\eps \f' \quad \textrm{in}\quad (0,1).
\]
The idea now is to exploit the negativity of $N^2 - 8N + 8$ for $2 \leq N \leq 6$ to reach the desired conclusion. Let $t_0 = \sup \{0 \leq t < 1: W(t) = 0\}$.
By Proposition \ref{Prop:B.1}(b), for every small $\delta > 0$, there exists $C_\delta >0$ such that for every $a > C_\delta$ we can find $\eps_1 = \eps_1(\delta,a)$ for which
\be
\label{ineg_feps}
1 - t_0  - \delta \leq f^2_\eps  \leq 1 - t_0   \text{ in } [C_\delta\eps, a \eps] \text{ for all } \eps \in (0,\eps_1).
\ee
The contribution of the term $-2 \barfeps  f'_\eps \f'$ in the above expression of $L_\eps^{GL}(\varphi f_\eps)f_\eps$ to the right hand side of \eqref{Eq:2VIPNew} is handled as follows. (Note that if $N=2$, then $\f'=0$ so that term vanishes and the reader can proceed directly to estimate \eqref{Eq:2VIPX7} below.) We impose that $\tilde q = \tilde q(r)$ is supported in $[C_\delta\eps, a\eps]$, then integration by parts combined with \eqref{ineg_feps} and $\big(r^{N-1}(\f^2)'\big)'=0$ for $r\in (0,1)$ yields by Cauchy-Schwarz:
\begin{align*}
&- 2 \int_0^1 r^{N-1}  \tilde q^2 \barfeps   f'_\eps \f \f'\, dr
	= \frac12 \int_0^1 r^{N-1}  \tilde q^2 (1 - t_0- f^2_\eps)' (\f^2)'\, dr\\
	&\qquad = -\int_0^1 r^{N-1}  \tilde q \tilde q' (1 - t_0 - f^2_\eps) (\f^2)'\, dr
	\leq \delta \int_0^1 (\tilde q')^2 r\, dr+ \frac{(N-2)^2}4\delta  \int_0^1 \frac{\tilde q^2}{r}\, dr.
\end{align*}
Since $2\leq N\leq 6$ implies $N^2 - 8N + 8 <0$,  using \eqref{ineg_feps}, we deduce
\begin{align}
\int_{B^N} \Big[L_\eps^{GL} q \cdot q &+ \frac{c_1 q^2}{\eps^2}\Big]\,dx 
	\leq |\Sphere^{N-1}|\int_0^1 r \Big\{ (1-t_0 + \delta) (\tilde q')^2\nonumber\\
		&\quad + \frac{1}{r^2}\Big[\frac{(N^2 - 8N + 8) (1 - t_0 -\delta)+ (N-2)^2\delta}{4} + \frac{c_1 r^2}{\eps^2} \Big] \tilde q^2		  \Big\}\,dr.
	\label{Eq:2VIPX7}
\end{align}
We now choose a non-negative $\tilde q\in Lip_c((0,1))$ given by
\[
\tilde q(r) = \tilde q_{a,\eps}(r) := 
\left\{\begin{array}{ll}
\sin \Big( \frac{\pi}{\ln \frac{a}{C_\delta} } \ln \frac{r}{C_\delta\eps}\Big) &\text{ for } r \in (C_\delta\eps, a\eps),\\
0 &\text{ elsewhere}.
\end{array}\right.
\]
Note that $(N^2 - 8N + 8)(1 - t_0 - \delta)+ (N-2)^2\delta= (N^2 - 8N + 8)(1-t_0) + c\delta$ for $c=4N-4>0$. Inserting into \eqref{Eq:2VIPX7}, we get
\begin{align}
\int_{B^N} &\Big[L_\eps^{GL} q \cdot q + \frac{c_1q^2}{\eps^2}\Big]\,dx 
	\leq |\Sphere^{N-1}|\int_{C_\delta \eps}^{a\eps}  \Big\{ \Big(\frac{\pi}{\ln \frac{a}{C_\delta} }\Big)^2 \cos^2 \Big( \frac{\pi}{\ln \frac{a}{C_\delta} } \ln \frac{r}{C_\delta\eps}\Big) (1 - t_0 +\delta)
	\nonumber\\
		&\qquad \qquad \qquad \quad + \Big(\frac{(N^2 - 8N + 8)(1-t_0) + c\delta}{4} + c_1a^2   \Big) \sin^2 \Big( \frac{\pi}{\ln \frac{a}{C_\delta\eps} } \ln \frac{r}{C_\delta\eps}\Big)
		  \Big\}\,\frac{1}{r}dr\nonumber
	\end{align}
	\begin{align}	  
	& =  \frac{|\Sphere^{N-1}|\ln \frac{a}{C_\delta} }{2} \Big(\Big(\frac{\pi}{\ln \frac{a}{C_\delta} }\Big)^2 (1 - t_0 +\delta)+\frac{(N^2 - 8N + 8)(1-t_0) + c\delta}{4} 
		+ c_1 a^2
		 \Big).  
	\label{Eq:2VIPX8}
\end{align}
Recalling $N^2 - 8N + 8 < 0$ for $2 \leq N \leq 6$, we can choose $\delta > 0$ sufficiently small, $a = a_\delta > 0$ sufficiently large and then $c_1 = c_1(\delta) > 0$ sufficiently small such that the right hand side of \eqref{Eq:2VIPX8} is negative for $\eps < \eps_1(\delta)$, yielding Step 1.

\medskip
\noindent\underline{Step 2:} {\it We prove that there exists $\eps_0 > 0$ such that $\ell(\eps) < 0$ and increasing in $(0,\eps_0)$, $\ell(\eps_0) = 0$ and $\ell(\eps) > 0$ for $\eps > \eps_0$.} Let $I = \{\eps \in (0,\infty): \ell(\eps) < 0\}$. It is clear that $\ell(\eps) > 0$ for large $\eps$ and so $I$ is bounded. By Step 1, $I$ contains $(0,\eps_1)$. Let 
\[
\eps_0 = \sup\{\tilde \eps: \ell(\eps) < 0 \text{ for } \eps \in (0,\tilde\eps)\} \in (\eps_1, \infty).
\]
By the continuity of $\ell$, we must have $\ell(\eps_0) = 0$. Then \eqref{Eq:ellMon} yields the monotonicity of $\ell$ in $(0, \eps_0)$ and also, $\ell(\eps) > 0$ for $\eps > \eps_0$. Step 2 is proved. 
\end{proof}

\subsection{The extended model: Existence.}\label{SSec:ExtE}

The aim is to prove Theorem \ref{Thm:ExtendedExist} for the extended model.

\begin{proof}[Proof of Theorem \ref{Thm:ExtendedExist}.]

\medskip
\noindent\underline{Proof of (a) when $N \geq 7$.} By \cite[Theorem 2]{INSZ18_CRAS}, when $N \geq 7$, $\bar m_\eps(x) = (f_\eps(|x|)n(x),0)$ \footnote{\cite[Theorem 1]{INSZ18_CRAS} assumes \eqref{Eq:WCond'}, but is clear from the proof there that \eqref{Eq:WCond'} is sufficient.} is the unique minimizer for the functional $E_{\eps,\infty}: \mcA \subset H^1(B^N,\RR^{N+1}) \rightarrow [0,\infty]$, i.e.
\[
E_{\eps,\infty}[m] = 
	\int_{B^N} \Big[\frac{1}{2}|\nabla m|^2 + \frac{1}{2\eps^2} W(1 - | m|^2)\Big]\,dx, \quad \eps>0.
\]
Recalling the fact that $\tilde W \geq 0$, it follows that for every $\eps, \eta>0$, $\bar m_\eps$ is the unique minimizer of $E_{\eps,\eta}$ in $\mcA$ and so  $(f_\eps,0)$ is the unique minimizer of $I_{\eps,\eta}$ in $\mcB$. This together with Proposition \ref{pro:unique_eps_eta} implies that \eqref{Eq:feegeeH1}--\eqref{Eq:20III21-feegeeBC} has no escaping solution.

\medskip
\noindent\underline{Proof of (a) when $W'(1) = 0$.} When $W'(1) = 0$, we have by \eqref{Eq:WCond'} that $W = 0$ in $[0,1]$. In particular, $E_{\eps,\infty}$ is exactly the Dirichlet energy (and hence convex) when restricting to the set $\{m \in \mcA : |m| \leq 1 \text{ a.e.}\}$. This together with the fact that for $m\in \mcA$,
\[
E_{\eps,\infty}[m] \geq E_{\eps,\infty}[m^\sharp] \text{ where } m^\sharp(x) = \left\{\begin{array}{ll} m(x) & \text{ if } |m| \leq 1,\\\frac{m(x)}{|m(x)|} &\text{ if } |m(x)| > 1,\end{array}\right.
\]
implies that the unique minimizer of $E_{\eps,\infty}$ is the map $Y(x) = (x,0)$ (i.e.  the unique $H^1(B^N, \RR^{N+1})$ harmonic map with boundary value $(x,0)$). Also, note that for $W\equiv 0$ in $[0,1]$, then $f_\eps(r) = r$ solves \eqref{Eq:24VII18-X1}--\eqref{Eq:24VII18-X2}, so by Theorem \ref{Thm:GLExist}, $f_\eps$ is the unique solution of \eqref{Eq:24VII18-X1}--\eqref{Eq:24VII18-X2}. Thus, $\bar m_\eps = (f_\eps n(x), 0)=Y$. We thus have that $\bar m_\eps$ is the unique minimizer of $E_{\eps,\infty}$ and hence of $E_{\eps,\eta}$ (since $\tilde W \geq \tilde W(0)$) in $\mcA$; in particular, $(f_\eps,0)$ is the unique minimizer of $I_{\eps, \eta}$ over $\mcB$. By appealing again to Proposition \ref{pro:unique_eps_eta}, we conclude that \eqref{Eq:feegeeH1}--\eqref{Eq:20III21-feegeeBC} has no escaping solution.

\medskip
\noindent\underline{Proof of (b)} First, we focus on the existence of escaping solutions of \eqref{Eq:feegeeH1}--\eqref{Eq:20III21-feegeeBC} when $2 \leq N \leq 6$ and $W'(1) > 0$. It is easy to see that $I_{\eps,\eta}$ admits a minimizer $(f_{\eps,\eta},g_{\eps,\eta}) \in \mcB$. Since $(f_{\eps,\eta},g_{\eps,\eta}) \in \mcB$, $(f_{\eps,\eta},g_{\eps,\eta}) \in C((0,1])$. It follows that $(f_{\eps,\eta},g_{\eps,\eta})$ satisfies \eqref{Eq:20III21-fee}--\eqref{Eq:20III21-feegeeBC} in the weak sense, and so $(f_{\eps,\eta},g_{\eps,\eta}) \in C^2((0,1])$.

Since $(|f_{\eps,\eta}|,|g_{\eps,\eta}|)$ is also a minimizer of $I_{\eps,\eta}$ in $\mcB$, the above argument also shows that $(|f_{\eps,\eta}|,|g_{\eps,\eta}|) \in C^2((0,1])$ satisfies \eqref{Eq:20III21-fee}--\eqref{Eq:20III21-feegeeBC}. Since $|f_{\eps,\eta}|,|g_{\eps,\eta}| \geq 0$ and $f_{\eps,\eta}(1) = 1$, by the strong maximum principle, we have that $|f_{\eps,\eta}| > 0$ in $(0,1)$, and either $|g_{\eps,\eta}| > 0$ in $(0,1)$ or $g_{\eps,\eta} \equiv 0$ in $(0,1)$. It follows that $f_{\eps,\eta} > 0$ in $(0,1)$, and either $g_{\eps,\eta} > 0$ in $(0,1)$ or $g_{\eps,\eta} < 0$ in $(0,1)$ or $g_{\eps,\eta} \equiv 0$ in $(0,1)$. Clearly, when $g_{\eps,\eta} \equiv 0$, $f_{\eps,\eta}$ is equal to the radial profile $f_\eps$ obtained in Theorem \ref{Thm:GLExist}. By considering $(f_{\eps,\eta},-g_{\eps,\eta})$ instead of $(f_{\eps,\eta},g_{\eps,\eta})$ if necessary, we assume in the sequel that $g_{\eps,\eta} \geq 0$.

\medskip
\noindent\underline{Claim:} $g_{\eps,\eta} > 0$ in $(0,1)$ if and only if $(\eps,\eta) \in A := \{(\eps,\eta): 0 < \eps < \eps_0, \eta > \eta_0(\eps)\}$.

\medskip
\noindent\underline{Proof of Claim:} Define 
\begin{align*}
&Q_{\eps,\eta}[\alpha,\beta]\\
	&=  \int_{B^N} \Big[L_\eps^{GL} \alpha \cdot \alpha + L_\eps^{GL} \beta \cdot \beta
		+\frac{N-1}{r^2} \alpha^2  + \frac{2}{\eps^2} W''(1 - f_\eps^2) f_\eps^2 \alpha^2 + \frac{1}{\eta^2} \tilde W'(0)\beta^2\Big]\,dx,
\end{align*}
for $(\alpha, \beta)$ belonging to the Hilbert space 
$$\mcH = \{(\alpha,\beta): (f_\eps + \alpha, \beta) \in \mcB\} \, \textrm{ with the norm} \quad \|(\alpha,\beta)\|_{\mcH} := \|(\alpha n, \beta)\|_{H^1(B^N,\RR^{N+1})}.$$ 
This can be considered as the second variation of $I_{\eps,\eta}$ at $(f_{\eps},0)$; see equation \eqref{Eq:09IV19-2VarX1} in Subsection \ref{SSec:OD}. Note that the $C^2$ regularity of $W$ together with \eqref{Eq:WCond'}, $\tilde W'(0)\geq 0$ and the boundedness of $f_{\eps}$ yield a constant $c_1 > 0$ (independent of $\eps$ and $\eta$) such that 
\be
\label{inerq}
Q_{\eps,\eta}[\alpha, \beta] \geq \|(\alpha, \beta)\|_{\mcH}^2 - \frac{c_1}{\eps^2} \|(\alpha, \beta)\|_{L^2(B^N)}^2 \text{ for all } (\alpha, \beta) \in \mcH.
\ee

\medskip
\noindent\underline{($\Leftarrow$)} 
If $(\eps,\eta) \in A$, then $\frac{\tilde W'(0)}{\eta^2} < -\ell(\eps)$. Taking $\beta \in H_0^1(B^N)$ to be any first eigenfunction of $L_{\eps}^{GL}$, which is radially symmetric, we have $r^{\frac{N-1}{2}}\beta', r^{\frac{N-1}{2}}\beta \in L^2(0,1)$, $\beta(1) = 0$ and $Q_{\eps,\eta}[0,\beta] < 0$. This implies that $(f_\eps,0)$ is not minimizing $I_{\eps,\eta}$ in $\mcB$, and thus $g_{\eps,\eta} > 0$.

\medskip
\noindent\underline{($\Rightarrow$)} For the converse, we suppose by contradiction that there exists $(\eps,\eta) \in B =  (0,\infty)^2 \setminus A$ with $g_{\eps,\eta} > 0$. By \eqref{Eq:2VIPNew} with the choice $\varphi = 1$ and by \eqref{Eq:24VII18-X1}, we have
\[
\int_{B^N} L_\eps^{GL} \alpha \cdot \alpha \, dx= \int_{B^N} \Big\{f_\eps^2 \Big|\nabla\Big(\frac{\alpha}{f_\eps}\Big)\Big|^2 - \frac{N-1}{r^2} \alpha^2 \Big\}\,dx \quad \text{ for every } \alpha \in C_c^\infty(0,1).
\]
By a density argument in $H^1_0(B^N)$ using Fatou's lemma, we deduce by \eqref{Eq:WCond'} that
\[
Q_{\eps,\eta}[\alpha,\beta] \geq \int_{B^N} \Big\{f_\eps^2 \Big|\nabla\Big(\frac{\alpha}{f_\eps}\Big)\Big|^2 + \big(\ell(\eps) +  \frac{\tilde W'(0)}{\eta^2} \big) \beta^2\Big\}\,dx \quad \text{ for every }(\alpha,\beta) \in \mcH.
\]
In view of Lemma \ref{Lem:eigen-gl}, we thus have that $Q_{\eps,\eta}$ is positive definite over $\mcH$ for $(\eps,\eta) \in \mathring{B} = (0,\infty)^2 \setminus \bar A $ where $\ell(\eps) +  \frac{\tilde W'(0)}{\eta^2}>0$. More precisely, there exists a constant $c>0$ (depending on $\eps$ and $\eta$) such that $Q_{\eps,\eta}[\alpha,\beta]\geq c \|(\alpha, \beta)\|^2_{L^2(B^N)}$ for every $(\alpha,\beta) \in \mcH$. This follows by the above inequality for $Q_{\eps,\eta}[\alpha,\beta]$ combined with the following estimate based on the Hardy inequality in $\RR^{N+2}$ using $r\leq f_{\eps}(r)\leq 1$ for every $r\in (0,1)$ (see Corollary \ref{Prop:GLMonotonicity}):
\be
\label{ineg55}
\int_0^1  r^{N-1} f_{\eps}^2 (h')^2\,dr
	\geq \int_0^1  r^{N+1}  (h')^2\,dr \geq \frac{N^2}{4}\int_0^1  r^{N-1} h^2\,dr \geq \frac{N^2}{4} \int_0^1  r^{N-1}f_{\eps}^2 h^2\,dr,
\ee
where $h$ plays the role of $\frac{\alpha}{f_\eps}$. Thus, by \eqref{inerq}, for $(\eps,\eta) \in \mathring{B}$, there exists a constant $\tilde c>0$ (depending on $\eps$ and $\eta$) such that
\be
\label{eusitu}
Q_{\eps,\eta}[\alpha,\beta] \geq \tilde c \|(\alpha, \beta)\|_{\mcH}^2 \quad \textrm{ for all } (\alpha, \beta) \in \mcH.
\ee

\medskip

\nd \underline{Fact}: {\it $(f_\eps,0)$ is a local minimizer of $I_{\eps,\eta}$ if $(\eps,\eta) \in \mathring{B}$.}
Indeed, by \eqref{Eq:24VII18-X1}, for  $(\alpha,\beta) \in \mcH$, 
\begin{align*}
&|\Sphere^{N-1}|\big(I_{\eps,\eta}[f_\eps + \alpha, \beta] - I_{\eps,\eta}[f_\eps , 0]\big) - \frac{1}{2} Q_{\eps,\eta}[\alpha,\beta] = 
\int_{B^N} h(x,\alpha(|x|)n(x),\beta(|x|))\,dx,
\\
h(x,V)
		&=    \frac{1}{2\eps^2} \Big\{W(1 - |f_\eps(r)n(x) + V_{\parallel} |^2 - V_{N+1}^2) - W(1 - f_\eps(r)^2)\\
			&\,\, \, + W'(1 - f_\eps(r)^2)(2 f_\eps(r)n(x)\cdot V_\parallel + |V|^2) - 2 W''(1 - f_\eps(r)^2) f_\eps(r)^2 (n(x)\cdot V_\parallel)^2\Big\} \\
		&\,\, \,   + \frac{1}{2\eta^2} \Big\{ \tilde W(V_{N+1}^2) - \tilde W(0) 
			 - \tilde W'(0) V_{N+1}^2\Big\}, \qquad r = |x|, V = (V_{\parallel}, V_{N+1})\in \R^{N+1}.
\end{align*}
We have $h\in C^0(\bar B^N, C^2(\RR^{N+1}))$ (since $W, \tilde W\in C^2$ and $f_\eps n\in C^2(\bar B^N)$ by Lemma \ref{Lem:ONSym}), $h(x,0)=0$, $\nabla_V h(x, 0)=0$, $\nabla^2_V h(x,0)=0$ (thus, \eqref{rrr} holds true in Lemma~\ref{Lem:RA}) and $h$ satisfies the growth assumption \eqref{add-now} in Lemma \ref{Lem:RA} for $p=2$ (due to the convexity of $W$ and $\tilde W$); therefore, Lemma \ref{Lem:RA} applies and 
yields some small radius $\tilde r>0$ such that  
$$\int_{B^N} h(x,\alpha(|x|)n(x),\beta(|x|))\,dx\geq -\frac{\tilde c }{4} \|(\alpha, \beta)\|^2_{\mcH} \quad \text{ for } for \|(\alpha, \beta)\|_{\mcH}<\tilde r.$$
Combined with \eqref{eusitu}, the local minimality of $(f_\eps,0)$ follows.

\smallskip

\nd \underline{End of proof of Claim}: Recalling our assumption that the constructed minimizer $(f_{\eps, \eta}, 
g_{\eps, \eta})$ of $I_{\eps, \eta}$ satisfies $g_{\eps, \eta}>0$,  the above Fact combined with Lemma \ref{Lem:NENM} below yield $(\eps,\eta) \in B\setminus \mathring{B}$ and, for all $(\tilde \eps,\tilde \eta) \in \mathring{B}$, $(f_{\tilde \eps},0)$ is the unique minimizer for $I_{\tilde \eps,\tilde \eta}$ in $\mcB$. Thanks to the latter, by considering a sequence $\{(\tilde \eps_j,\tilde \eta_j)\} \subset \mathring{B}$ which converges to $(\eps,\eta)$, since $f_{\tilde \eps_j}$ converges to $f_{\eps}$ in $H^1(B^N)$, Fatou's lemma implies that $(f_{\eps},0)$ is a minimizer for $I_{\eps,\eta}$ in $\mcB$, which contradicts the fact that $(f_{\eps,\eta},\pm g_{\eps,\eta})$ are the only two minimizers of $I_{\eps,\eta}$ in $\mcB$ (see Proposition \ref{pro:unique_eps_eta}). This proves the claim.

\medskip
\noindent\underline{Proof of (b1)} By the Claim, an escaping solution of \eqref{Eq:feegeeH1}--\eqref{Eq:20III21-feegeeBC} exists if and only if $0 < \eps < \eps_0$ and $\eta > \eta_0(\eps)$. In this case, the uniqueness of an escaping solution and the classification of minimizers of $I_{\eps,\eta}$ are obtained in Proposition \ref{pro:unique_eps_eta}, Lemma \ref{lem:f2g2} yields $f_{\eps,\eta}^2 + g_{\eps,\eta}^2 < 1$, the regularity of 
$(f_{\eps, \eta}, g_{\eps, \eta})$ follows from Lemma \ref{Lem:FullONSym}, while the positivity of $f_{\eps, \eta}$ and monotonicity of $f_{\eps, \eta}$ and $g_{\eps, \eta}$ are given by Propositions~\ref{Prop:fPos} and \ref{Prop:ODEMonotonicity}. 

\medskip
\noindent\underline{Proof of (b2)} The fact that the non-escaping solution $(f_\eps,0)$ is an unstable critical point (and hence not minimizer) of $I_{\eps,\eta}$ in $\mcB$ when $0 < \eps < \eps_0$ and $\eta > \eta_0(\eps)$ was obtained in the proof of the ($\Leftarrow$) part of the claim. The fact that the non-escaping solution $(f_\eps,0)$ is the unique minimizer of $I_{\eps,\eta}$ in $\mcB$ when $\eps \geq \eps_0$ or $0 < \eta \leq \eta_0(\eps)$ follows from the claim.
\end{proof}

It remains to prove the following lemma used above:

\begin{lemma}\label{Lem:NENM}
Let $N\geq 2$, $\eps, \eta > 0$, and suppose that $W \in C^2((-\infty,1])$ and $\tilde W \in C^2([0,\infty))$ satisfy \eqref{Eq:WCond'} and \eqref{Eq:TWCond'}. If $I_{\eps,\eta}$ admits an escaping critical point $(f_{\eps,\eta}, g_{\eps,\eta})$ in $\mcB$ with $g_{\eps,\eta} > 0$ in $(0,1)$, then the non-escaping critical point $(f_\eps,0)$ is not a local minimizer of $I_{\eps,\eta}$. As a consequence, if the non-escaping critical point $(f_\eps,0)$ is a local minimizer of $I_{\eps,\eta}$, then $(f_\eps,0)$ is the unique {\bf global} minimizer of $I_{\eps,\eta}$ in $\mcB$ and $I_{\eps,\eta}$ does not admit any escaping critical point $(f_{\eps,\eta}, g_{\eps,\eta})$ in $\mcB$ with $g_{\eps,\eta} > 0$ in $(0,1)$.
\end{lemma} 

\begin{proof}
By Proposition \ref{pro:unique_eps_eta}, $(f_{\eps,\eta},\pm g_{\eps,\eta})$ are the only two minimizers of $I_{\eps,\eta}$ in $\mcB$. In particular, $I_{\eps,\eta}[f_{\eps,\eta},g_{\eps,\eta}] < I_{\eps,\eta}[f_\eps,0]$. Suppose by contradiction that $(f_\eps,0)$ is a local minimizer of $I_{\eps,\eta}$. We use some ideas from \cite{ABG-uniqueness, INSZ_CVPDE}: we show, by mean of a mountain-pass theorem, the existence of a second escaping critical point 
$(\hat f, \hat g)$ of $I_{\eps,\eta}$ with $\hat g>0$ which would lead to a contradiction with Proposition \ref{pro:unique_eps_eta}. Along the way, care is given due to the fact that $I_{\eps,\eta}$ is not always finite in $\mcB$. To avoid this problem, let $V, \tilde V  \in C^2(\RR)$ be bounded non-negative functions such that $V|_{[0,1]} = W|_{[0,1]}$, $\tilde V|_{[0,1]} = \tilde W|_{[0,1]}$ and define $J: \mcH \rightarrow \RR$ by
\[
J[\alpha,\beta]
	=  \frac{1}{2}\int_0^1 \Big[((f_\eps + \alpha)')^2 + (\beta')^2 + \frac{N-1}{r^2} (f_\eps + \alpha)^2 + \frac{1}{\eps^2} V(1 - (f_\eps + \alpha)^2 - \beta^2) + \frac{1}{\eta^2} \tilde V(\beta^2)\Big]\,r^{N-1}\,dr.
\]
Let $\mcM := \{(\alpha,\beta) \in \mcH: f_\eps + \alpha \geq 0, \beta  \geq 0 \text{ and } (f_\eps + \alpha)^2 + \beta^2 \leq 1 \text{ in } (0,1)\}$. Then $J \in C^1(\mcH)$, $\mcM$ is a closed convex subset of $\mcH$, $J[\alpha,\beta] = I_{\eps,\eta}[f_\eps + \alpha, \beta]$ for $(\alpha,\beta) \in \mcM$, and $(0,0)$ and $(f_{\eps,\eta} - f_\eps,g_{\eps,\eta})$ are two relative minima of $J$ in $\mcM$ with $J (f_{\eps,\eta} - f_\eps,g_{\eps,\eta}) < J(0,0)$.

We proceed to check that $J$ satisfies the Palais-Smale condition on $\mcM$ (see e.g. \cite[Theorem II.12.8]{Struwe}): if $\{(\alpha_j, \beta_j)\} \subset \mcM$ is such that $\{J[\alpha_j,\beta_j]\}$ is bounded and 
\begin{equation}
G[\alpha_j,\beta_j] := \sup_{(\alpha_{j} - \varphi, \beta_j - \psi) \in \mcM: \|(\varphi,\psi)\|_{\mcH} \leq 1} \langle DJ[\alpha_j, \beta_j], (\varphi, \psi)\rangle \rightarrow 0,
	\label{Eq:PS-1}
\end{equation}
then $\{(\alpha_j, \beta_j)\}$ is relatively compact in $\mcH$. Indeed, since $\{J[\alpha_j,\beta_j]\}$ is bounded, $\{(\alpha_j,\beta_j)\}$ is bounded in $\mcH$. Thus, we assume that $(\alpha_j, \beta_j)$ converges weakly in $\mcH$, strongly in $L^2(B^N)$, and almost everywhere in $(0,1)$ to some $(\alpha_*, \beta_*) \in \mcM$.

Let us note that we may use $(\varphi ,\psi) = t(\alpha_j - \alpha_*, \beta_j - \beta_*) = t((f_\eps + \alpha_j) - (f_\eps + \alpha_*), \beta_j - \beta_*)$ for some small $t > 0$ (which is independent of $j$) in \eqref{Eq:PS-1}, since
$(\alpha_j - \varphi, \beta_j - \psi)$ is a convex combination of $(\alpha_j, \beta_j), (\alpha_*, \beta_*) \in \mcM$ 
and $\mcM$ is convex.
This gives 
\begin{align*}
0
	&\geq \limsup_{j\rightarrow \infty} \langle DJ[\alpha_j, \beta_j], (\alpha_j - \alpha_*, \beta_j - \beta_*)\rangle\\
	&= \limsup_{j\rightarrow \infty} \int_0^1 \Big[
		(f_\eps + \alpha_j)' (\alpha_j - \alpha_*)'
		+ \beta_j' (\beta_j - \beta_*)'
		+ \frac{N-1}{r^2} (f_\eps + \alpha_j) (\alpha_j - \alpha_*)\\
		&\quad - \frac{1}{\eps^2} W'(1 - (f_\eps + \alpha_j)^2 - \beta_j^2)[(f_\eps + \alpha_j) (\alpha_j - \alpha_*) + \beta_j(\beta_j - \beta_*)]\\
		&\quad + \frac{1}{\eta^2} \tilde W'(\beta_j^2) \beta_j (\beta_j - \beta_*)
		\Big]\,r^{N-1}\,dr.
\end{align*}
Using the strong convergence of $(\alpha_j, \beta_j)$ to $(\alpha_*, \beta_*)$ in $L^2(B^N)$ and the boundedness of $(\alpha_j, \beta_j)$ in $L^\infty(B^N)$, the last two lines above converge to $0$ as $j\to \infty$. Then
 writing $\alpha_j - \alpha_* = (f_\eps + \alpha_j) - (f_\eps +\alpha_*)$, by the weak convergence of $(\alpha_j, \beta_j)$ in $\mcH$, we get
\begin{align*}
0
	&\geq \limsup_{j\rightarrow \infty} \int_0^1 \Big[
		((f_\eps + \alpha_j)')^2 + (\beta_j')^2 + \frac{N-1}{r^2} (f_\eps + \alpha_j) ^2\Big]\,r^{N-1}\,dr\\
		&\qquad - \int_0^1 \Big[
		((f_\eps + \alpha_*)')^2 + (\beta_*')^2 + \frac{N-1}{r^2} (f_\eps + \alpha_*) ^2\Big]\,r^{N-1}\,dr.
\end{align*}
This implies that $\|((f_\eps + \alpha_j)n,\beta_j)\|_{H^1(B^N,\RR^{N+1})}$ converges to $\|((f_\eps + \alpha_*)n,\beta_*)\|_{H^1(B^N,\RR^{N+1})}$ and so $((f_\eps + \alpha_j)n,\beta_j)$ converges strongly in $H^1(B^N,\RR^{N+1})$ to $((f_\eps + \alpha_*)n,\beta_*)$. This means also that $(\alpha_j, \beta_j)$ converges strongly in $\mcH$ to $(\alpha_*,\beta_*)$, giving the desired Palais--Smale property for $J$.

Applying the mountain pass theorem (see e.g. \cite[Theorem II.12.8]{Struwe}), we see that $J$ has a mountain-pass critical point $(\hat \alpha_{\eps,\eta},\hat \beta_{\eps,\eta}) \in \mcM$ relative to $\mcM$, i.e.
\begin{equation}
\sup_{(\hat \alpha_{\eps,\eta} - \varphi, \hat\beta_{\eps,\eta} - \psi) \in \mcM: \|(\varphi,\psi)\|_{\mcH} \leq 1} \langle DJ[\hat \alpha_{\eps,\eta},\hat \beta_{\eps,\eta}], (\varphi, \psi)\rangle = 0.
	\label{Eq:MPCrit-1}
\end{equation}
In addition, $(\hat \alpha_{\eps,\eta},\hat \beta_{\eps,\eta})$ is not a local minimizer of $J$ relative to $\mcM$. For ease of exposition, we write $\hat f = f_\eps + \hat \alpha_{\eps,\eta}$ and $\hat g=\hat \beta_{\eps,\eta}$. Then \eqref{Eq:MPCrit-1} means
\begin{align}
&0 = \sup\Big\{
	\int_0^1 r^{N-1}\Big[ \hat f' \varphi' + \hat g' \psi' + \frac{N-1}{r^2} \hat f \varphi - \frac{1}{\eps^2} W'(1 - \hat f^2 - \hat g^2)(\hat f \varphi + \hat g \psi)\label{Eq:MPCrit-1X} \\
		& + \frac{1}{\eta^2}\tilde W'(\hat g^2) \hat g \psi\Big]\,dr : \|(\varphi,\psi)\|_{\mcH} \leq 1,  \hat f - \varphi \geq 0, \hat g - \psi \geq 0, (\hat f - \varphi)^2 + (\hat g - \psi)^2 \leq 1\Big\}.
	\nonumber
\end{align}

To proceed, we show that $\hat f^2 + \hat g^2 < 1$ in $(0,1)$, $\hat f > 0$ in $(0,1)$, and either $\hat g \equiv 0$ in $(0,1)$ or $\hat g > 0$ in $(0,1)$, so that we have in fact that $(\hat f, \hat g)$ is either a non-escaping solution $(f_\eps,0)$ or an escaping solution of \eqref{Eq:feegeeH1}--\eqref{Eq:20III21-feegeeBC}. Once this is proved, by Theorem \ref{Thm:GLExist} and Proposition \ref{pro:unique_eps_eta}, we then have that $(\hat f, \hat g)$ must be identical to either $(f_\eps,0)$ or $(f_{\eps,\eta},g_{\eps,\eta})$, which contradicts the fact that $(\hat \alpha_{\eps,\eta}, \hat \beta_{\eps,\eta})$ is not a local minimizer of $J$ relative to $\mcM$. 

Indeed, using $(\varphi, \psi) = t \zeta (\hat f, \hat g)$ in \eqref{Eq:MPCrit-1X} where $\zeta \in C_c^\infty(0,1)$ is non-negative and $t \geq 0$ is sufficiently small so that $0 \leq 1 - t\zeta \leq 1$ in $(0,1)$, we obtain
\begin{align*}
&- \frac{1}{2}[(\hat f^2)'' + \frac{N-1}{r} (\hat f^2)']  - \frac{1}{2} [(\hat g^2)'' + \frac{N-1}{r} (\hat g^2)'] + (\hat f')^2 + (\hat g')^2  + \frac{N-1}{r^2} \hat f^2 \\
	&\qquad
		- \frac{1}{\eps^2} W'(1 - \hat f^2 - \hat g^2) (\hat f^2 + \hat g^2)  + \frac{1}{\eta^2} \tilde W'(\hat g^2) \hat g^2
		\leq 0 \text{ in } (0,1)
\end{align*}
in the sense of distribution. It follows that the function $\hat X = 1 - \hat f^2 - \hat g^2$, considered as a radially symmetric function in $B^N$, satisfies
\[
- \hat X'' - \frac{N-1}{r} \hat X' + 2 a(r) \hat X \geq \frac{2(N-1)}{r^2} \hat f^2 \geq 0 \text{ in } (0,1)
\]
where the continuous function $a: (0,1] \rightarrow [0,\infty)$ is given in \eqref{Eq:M1Aux-a1}. Since $\hat X \geq 0$, we deduce from the strong max principle that either $\hat X \equiv 0$ in $(0,1)$ or $\hat X > 0$ in $(0,1)$. The case $\hat X \equiv 0$ is impossible since it would imply, in view of the above differential inequality, that $\hat f \equiv 0$, contradicting that $\hat f(1) = 1$. We thus have $\hat X > 0$ and $\hat f^2 + \hat g^2 < 1$ in $(0,1)$.

As $\hat f^2 + \hat g^2 < 1$ in $(0,1)$, we may use $(\varphi, \psi) = (-t \zeta,0)$ in \eqref{Eq:MPCrit-1X} where $\zeta \in C_c^\infty(0,1)$ is non-negative and $t \geq 0$ is sufficiently small so that $(\hat f + t\zeta)^2 + \hat g^2 < 1$ in $(0,1)$ to get
\[
\hat f'' + \frac{N-1}{r} \hat f'  - b(r) \hat f \leq 0 \text{ in } (0,1), \quad b(r) := \frac{N-1}{r^2} - \frac{1}{\eps^2} W'(1 -  \hat f^2 - \hat g^2)\in L^\infty_{loc}((0,1]).
\]
Since $\hat f \geq 0$ and $\hat f(1) = 1$, we have by the strong maximum principle that $\hat f > 0$ in $(0,1)$.

Likewise, we use $(\varphi, \psi) = (0, -t \zeta)$ in \eqref{Eq:MPCrit-1X} where $\zeta \in C_c^\infty(0,1)$ is non-negative and $t \geq 0$ is sufficiently small so that $\hat f^2 + (\hat g + t\zeta)^2 < 1$ in $(0,1)$ to get
\[
\hat g'' + \frac{N-1}{r} \hat g' - c(r) \hat g \leq 0 \text{ in } (0,1), \quad c(r) := -\frac{1}{\eps^2} W'(1 -  \hat f^2 - \hat g^2) + \frac{1}{\eta^2} \hat W'(\hat g^2).
\]
Since $\hat g \geq 0$, we have by the strong maximum principle that either $\hat g \equiv 0$ in $(0,1)$ or $\hat g > 0$ in $(0,1)$. As explained earlier, this together with the previous shown fact that $\hat f^2 + \hat g^2 < 1$ and $\hat f > 0$ in $(0,1)$ shows that the statement that $(\eps,\eta) \in \mathring{B}$ amounts to a contradiction.

Finally, we explain the stated consequence: by the proof of Theorem \ref{Thm:ExtendedExist} b), any minimizer 
$(f_{\eps, \eta},
g_{\eps, \eta})$ of $I_{\eps, \eta}$ in $\mcB$ satisfies $|g_{\eps, \eta}|>0$ or $g_{\eps, \eta}\equiv 0$. As we have just proved that escaping critical points of $I_{\eps, \eta}$  cannot exist whenever $(f_\eps, 0)$ is a local minimizer of  $I_{\eps, \eta}$, we conclude that every minimizer satisfies $g_{\eps, \eta}\equiv 0$, i.e., it is given by $(f_\eps, 0)$.
\end{proof}

\subsection{The $\Sphere^N$-valued GL model: Existence, monotonicity and uniqueness}\label{SSec:MMEU}

We start with positivity of $\tilde f_\eta$ and the monotonicity for an escaping solution $(\tilde f_{\eta}, g_{\eta})$ of \eqref{Eq:metaform}--\eqref{Eq:MM-feegeeBC} with $g_\eta>0$. Next we prove Theorem
\ref{Thm:MMExist}.

\begin{proposition}\label{Prop:MM-Monotonicity}
Suppose $\tilde W \in C^2([0,\infty))$ satisfies \eqref{Eq:TWCond'}, and $(\tilde f_{\eta}, g_{\eta})$ satisfies \eqref{Eq:metaform}--\eqref{Eq:MM-feegeeBC} with $g_{\eta}> 0$ in $(0,1)$. Then $\tilde f_\eta > 0$, $\tilde f'_{\eta}>0$, $g'_{\eta} < 0$ and $\big(\frac{\tilde f_{\eta}}{r}\big)' \leq 0$ in $(0,1]$.
\end{proposition}

\begin{proof} We adapt the strategy in the proof of Propositions \ref{Prop:ODEMonotonicity} and \ref{Prop:fPos}. By Lemma \ref{Lem:MMONSym}, $(\tilde f_\eta, g_\eta) \in C^2([0,1], \mathbb{S}^1)$ and $f(0) = 0$. Recalling also that $g_\eta > 0$, we may thus write $\tilde f_\eta = \sin \theta$, $g_\eta = \cos\theta$ in $[0,1]$ where the lifting $\theta: [0,1]\rightarrow [-\pi/2,\pi/2]$ is $C^2$, $\theta(0) = 0$ and $\theta(1) = \pi/2$. Then $\theta$ satisfies 
\begin{equation}
\theta'' + \frac{N-1}{r} \theta' = \frac{N-1}{r^2} \sin \theta\,\cos\theta  - \frac{1}{\eta^2} \tilde W'(\cos^2\theta) \sin\theta\cos\theta =: P(r,\theta) \text{ in } (0,1).
	\label{Eq:ThetaODE}
\end{equation}
Since $\theta(1) = \pi/2$, $\theta \leq \pi/2$ in $(0,1)$, and $\pi/2$ is a constant solution of \eqref{Eq:ThetaODE}, the maximum principle  and the Hopf lemma applied to \eqref{Eq:ThetaODE} yield $\theta < \pi/2$ in $(0,1)$ and $\theta'(1) > 0$.

Let $r_1 \in [0,1)$ be such that $\theta(r_1) = 0$ and $\theta > 0$ in $(r_1,1]$. Observe that, if $r_1 > 0$, then by applying the Hopf lemma to \eqref{Eq:ThetaODE} in $(r_1,1)$, we have $\theta'(r_1) > 0$. In particular, $\theta < 0$ in some small interval $(r_1 - \delta, r_1)$ when $r_1 > 0$.

Observe that, since $P(r,\theta)$ is odd in $\theta$, $|\theta|$ satisfies in the sense of distribution
\begin{align*}
|\theta|'' + \frac{N-1}{r} |\theta|' = P(r,|\theta|) \textrm{ in } (r_1,1), \quad \text{ and } \quad 
|\theta|'' + \frac{N-1}{r} |\theta|' \geq P(r,|\theta|) \textrm{ in } (0,1).
\end{align*}
Since $P$ is non-increasing in $r$, we can apply the proof of Proposition \ref{Prop:ODEMonotonicity} to obtain
\[
(|\theta|)_s \geq |\theta| \text{ in } \max(0, 2s - 1) < r < s \text{ for all } r_1 \leq s < 1,
\]
where $(|\theta|)_s(r) = |\theta|(2s - r)$. As in the proof of Proposition \ref{Prop:fPos}, the Hopf lemma then implies that $r_1 = 0$, i.e. $\theta > 0$ in $(0,1)$, and so the above gives
\[
\theta_s \geq \theta \text{ in } \max(0, 2s - 1) < r < s \text{ for all } 0 < s < 1.
\]
In addition, we have that $\theta' > 0$ in $(0,1]$ (see Fact 2 in the proof of Proposition \ref{Prop:ODEMonotonicity}). In particular, $0 = \theta(0) < \theta < \theta(1) = \pi/2$ in $(0,1)$.

Returning to $(\tilde f_\eta, g_\eta)$, we have shown that $\tilde f_\eta > 0$, $\tilde f'_{\eta}>0$ and $g'_{\eta} < 0$ in $(0,1]$. The statement $\big(\frac{\tilde f_{\eta}}{r}\big)' \leq 0$ in $(0,1]$ is obtained in the same way as in the last part of the proof of Proposition \ref{Prop:ODEMonotonicity} using the following equivalent form of \eqref{Eq:MM-fee}
\[
\Big(r^{N+1} \big(\frac{\tilde f_\eta}{r}\big)'(r)\Big)' = - r^{N+1}\lambda(r) \frac{\tilde f_\eta(r)}{r}\leq 0, \quad r\in (0,1).
\]
The proof is complete.
\end{proof}

Next we prove the uniqueness of escaping solutions of \eqref{Eq:metaform}--\eqref{Eq:MM-feegeeBC}.

\begin{proposition}\label{Prop:MMUniqueness}
Let $N\geq 2$ and $\eta > 0$. Suppose that $\tilde W \in C^2([0, \infty))$ satisfies \eqref{Eq:TWCond'}. Then the system \eqref{Eq:metaform}--\eqref{Eq:MM-feegeeBC} has at most one escaping solution $(\tilde f_\eta, g_\eta)$ with $g_\eta > 0$ in $(0,1)$. Furthermore, when it exists, then $(\tilde f_\eta, \pm g_\eta)$ are the only two minimizers of the functional $I_{\eta}^{MM}$ in $\mcB^{MM}$.
\end{proposition}

\begin{proof} By Proposition \ref{Prop:MM-Monotonicity}, we have $\tilde f_\eta > 0$ in $(0,1)$ for any escaping $(\tilde f_\eta, g_\eta)$ with $g_\eta > 0$ in $(0,1)$ of the system \eqref{Eq:metaform}--\eqref{Eq:MM-feegeeBC}. To prove the uniqueness, we follow a similar argument to the proof of Proposition \ref{pro:unique_eps_eta}, adapted to the new target space $\bS^{N}$. Indeed, denoting $m_\eta=(\tilde f_\eta(r) n(x), g_\eta(r))\in H^1(B^N, \bS^{N})$ for a solution $(\tilde f_\eta, g_\eta)$ in $(0,1)$ of the system \eqref{Eq:metaform}--\eqref{Eq:MM-feegeeBC} with $g_\eta>0$ in $(0,1)$, we consider an arbitrary radial configuration $m=(f(r) n(x), g(r))\in H^1(B^N, \bS^{N})$ with $m=(n,0)$ on $\partial B^N$. Setting $V=m-m_\eta=(s(r)n, q(r))\in H_0^1(B^N, \RR^{N+1})$, the constraints $|m|=|m_\eta|=1$ yield $\tilde f_\eta s+g_\eta q=m_\eta\cdot V=-\frac12 |V|^2$ in $B^N$. Together with the convexity of $\tilde W$ and \eqref{Eq:MM-fee}--\eqref{Eq:MM-gee}, we compute
\begin{align*}
&I_{\eta}^{MM}[f,g] - I_{\eta}^{MM}[\tilde f_{\eta},g_{\eta}]\\
&	\geq \frac{1}{2} \int_0^1 r^{N-1}\, \Big\{ 2\tilde f_{\eta}' s' + (s')^2 + 2g_{\eta}' q' + (q')^2 + \frac{N-1}{r^2}(2\tilde f_{\eta}s + s^2) + \frac{1}{\eta^2} \tilde W'(g_{\eta}^2)(  2g_{\eta}q + q^2)  \Big\}\,dr\\
	&= \frac{1}{2} \int_0^1 r^{N-1} \Big\{ (s')^2  + (q')^2 + \frac{N-1}{r^2} s^2
	 + \frac{1}{\eta^2} \tilde W'(g_{\eta}^2) q^2 +2\lambda(r) (\tilde f_\eta s+g_\eta q) \Big\}\,dr\\
	&=\frac1{2|\bS^{N-1}|} \int_{B^N}\Big\{ |\nabla V|^2+ \frac{1}{\eta^2} \tilde W'(g_{\eta}^2) V_{N+1}^2  -\lambda(r) |V|^2\Big\}\,dx
	=: \frac1{2|\bS^{N-1}|} F^{MM}_{\eta}[V].
\end{align*}

\noindent\underline{Claim:} For every $V(x)=(s(r)n(x), q(r))\in H^1_0(B^N, \RR^{N+1})$, it holds
\[
F^{MM}_{\eta}[V] \geq \int_{B^N} \Big\{\tilde f_{\eta}^2(|x|) \Big| \nabla\Big(\frac{s}{\tilde f_{\eta}}\Big)(|x|)\Big|^2 + g_{\eta}^2(|x|) \Big|\nabla\Big(\frac{q}{g_{\eta}}\Big)(|x|)\Big|^2 \Big\}\,dx.
\]
\nd \underline{Proof of Claim:} Since $F^{MM}_{\eta}$ is continuous in $H_0^1(B^N, \RR^{N+1})$ (because $\lambda, \tilde W'(g_{\eta}^2)\in L^\infty(B^N)$ by Lemma \ref{Lem:MMONSym}), by standard density results and Fatou's lemma, it suffices to show the claim for $V=(s(r)n, q(r)) \in C_c^\infty(B^N\setminus \{0\}, \RR^{N+1})$. For that, we apply \cite[Lemma A.1]{INSZ3} to the operators 
$${\tilde L}:=-\Delta-\lambda(r) \quad\textrm{ and } \quad {\tilde T}:=-\Delta+ \frac{1}{\eta^2} \tilde W'(g_{\eps,\eta}^2)-\lambda(r).$$ Writing $V=(s(r)n, q(r))=(V_1, \dots, V_N, V_{N+1})\in C_c^\infty(B^N\setminus \{0\}, \RR^{N+1})$ and decomposing $V_j= \tilde f_{\eta}\hat V_j$ with $\hat V_j= \frac{V_j}{\tilde f_{\eta}}$ for $j=1,\dots, N$ and $V_{N+1}=g_{\eta} \hat V_{N+1}$ with $\hat V_{N+1}=\frac{q}{g_{\eta}}$,
\begin{align*}
\nonumber F^{MM}_{\eta}[V] &= \sum_{j=1}^N \int_{B^N} \, {\tilde L} V_j\cdot V_j \, dx+\int_{B^N} \, {\tilde T} V_{N+1}\cdot 
V_{N+1} \, dx \\
\nonumber&=\sum_{j=1}^N \int_{B^N} \, \Big\{\tilde f_\eta^2 |\nabla \hat V_j|^2 + \hat V_j^2 {\tilde L}\tilde f_\eta\cdot \tilde f_\eta\Big\}\, dx+\int_{B^N} \, \Big\{g_{\eta}^2 |\nabla \hat V_{N+1}|^2 + \hat V_{N+1}^2 {\tilde T}g_{\eta}\cdot g_{\eta}\Big\}\, dx\\
&\nonumber=\int_{B^N} \,\Big\{ \tilde f_\eta^2 \Big|\nabla \big(\frac{s(r)}{\tilde f_\eta(r)} n(x)\big)\Big|^2 -\frac{N-1}{r^2}s^2+g_{\eta}^2(|x|) \Big|\nabla\Big(\frac{q}{g_{\eta}}\Big)(|x|)\Big|^2 \Big\}\, dx\\
&=\int_{B^N} \Big\{\tilde f_\eta^2(|x|) \Big| \nabla\Big(\frac{s}{\tilde f_\eta}\Big)(|x|)\Big|^2 + g_{\eta}^2(|x|) \Big|\nabla\Big(\frac{q}{g_\eta}\Big)(|x|)\Big|^2 \Big\}\,dx,
	\end{align*}
because ${\tilde L}\tilde f_\eta=-\frac{N-1}{r^2} \tilde f_\eta$, ${\tilde T}g_{\eta}=0$ (by \eqref{Eq:MM-fee}--\eqref{Eq:MM-gee}) and $(\hat V_1, \dots, \hat V_N)=\frac{s(r)}{\tilde f_\eta(r)} n(x)$ with $|\nabla n|^2=\frac{N-1}{r^2}$.
Hence, the claim is proved.

As direct consequence of the claim, $(\tilde f_\eta, \pm g_{\eta})$ minimizes $I^{MM}_{\eta}$ in $\mcB^{MM}$. If $(\hat f_{\eta},\hat g_{\eta})$ also minimizes $I^{MM}_{\eta}$ in $\mcB^{MM}$, the argument in Step 1 of the proof of Proposition \ref{pro:unique_eps_eta} gives 
\[
\frac{\hat f_{\eta}  -  \tilde f_{\eta} }{\tilde f_{\eta} } 
	\text{ and }
	\frac{\hat g_{\eta}  - g_{\eta} }{ g_{\eta} }
	\text{ are constant in } (0,1).
\]
This together with $\hat f_{\eta}(1) -\tilde f_{\eta}(1) = 0$ gives $\hat f_{\eta}  \equiv \tilde f_{\eta}$ and $\hat g_{\eta}  \equiv a g_{\eta}$ in $(0,1)$ for some constant $a \in \RR$. Since $\tilde f_\eta^2 + g_\eta^2 = 1 = \hat f_\eta^2 + \hat g_\eta^2$ we deduce that $\hat g_{\eta}  \equiv \pm g_{\eta}$ in $(0,1)$. This proves that  $(\tilde f_\eta, \pm g_{\eta})$ are the only two minimizers of $I^{MM}_{\eta}$ in $\mcB^{MM}$

Lastly, if $(\check f_{\eta},\check g_{\eta})$ is also a solution to \eqref{Eq:metaform}--\eqref{Eq:MM-feegeeBC} with $\check g_{\eta} > 0$ in $(0,1)$, then the claim yields that $(\check f_{\eta}, \check g_{\eta})$ also minimizes $I^{MM}_{\eta}$ in $\mcB^{MM}$, and by the above, $\check f_{\eta}  \equiv \tilde f_{\eta}$ and $\check g_{\eta}  \equiv g_{\eta}$ in $(0,1)$. The proof is complete.
\end{proof}

\begin{proof}[Proof of Theorem \ref{Thm:MMExist}.] Recall that in dimension $N \geq 7$, since $\tilde W \geq 0$, the equator map $\nespm(x) = (n(x),0)$ is the unique minimizer of $E_\eta^{MM}$ in $\mcA$ for every $\eta > 0$ (see Remark \ref{Rem:R1}). Thus, by \eqref{Eq:TWCond'} and Proposition \ref{Prop:MMUniqueness}, escaping solutions of \eqref{Eq:metaform}--\eqref{Eq:MM-feegeeBC} do not exist for any $\eta > 0$.

Suppose in the rest of the proof that $2 \leq N \leq 6$ and fix some $\eta > 0$. The uniqueness of escaping solution $(\tilde f_\eta, g_\eta)$ of \eqref{Eq:metaform}--\eqref{Eq:MM-feegeeBC} with $g_\eta > 0$ together with its minimality, monotonicity and positivity were proved in Propositions \ref{Prop:MM-Monotonicity} and \ref{Prop:MMUniqueness} and its regularity follows from Lemma \ref{Lem:MMONSym} in Appendix \ref{App}. 

It remains to prove the existence\footnote{For the existence of an escaping solution, it suffices to assume $\tilde W \in C^2([0,\infty))$ instead of \eqref{Eq:TWCond'}.} of an escaping solution of \eqref{Eq:metaform}--\eqref{Eq:MM-feegeeBC} for $2 \leq N \leq 6$ and the instability of the non-escaping solution $(1,0)$ for $3 \leq N \leq 6$.

\medskip
\noindent\underline{Proof of the instability of $(1,0)$ when $3 \leq N \leq 6$:} We show the second variation of $I_\eta^{MM}$ in $\mcB^{MM}$ at $(1,0)$ is not non-negative semi-definite, i.e. there exists $q \in Lip_c(0,1)$ such that
\[
Q_\eta^{MM}[0,q] = \frac{d^2}{dt^2}\Big|_{t = 0} I_{\eta}^{MM}\Big(\frac{(1, tq)}{\sqrt{1 + t^2q^2}}\Big) = \int_0^1 \Big[(q')^2 - \frac{N-1}{r^2}q^2 + \frac{\tilde W'(0)}{\eta^2} q^2\Big]\,r^{N-1}\,dr < 0.
\]

To this end, we adapt the computation in Step 1 of the proof of Lemma \ref{Lem:eigen-gl}(c). Writing $q = \varphi \tilde q$ with $\varphi = r^{-\frac{N-2}{2}}$ and applying \cite[Lemma A.1]{ODE_INSZ} (for the Laplace operator), we have
\[
Q_\eta^{MM}[0,q] = \int_0^1 \Big\{(\tilde q')^2 + \frac{1}{r^2} \Big[\frac{N^2 - 8N + 8}{4} + \frac{\tilde W'(0) r^2}{\eta^2} \Big]\tilde q^2\Big\}\,r\,dr.
\]
For $0 < b < a < 1$ to be fixed, let
\[
\tilde q(r) = \left\{\begin{array}{ll}
	\sin \Big(\frac{\pi}{\ln\frac{a}{b}} \ln \frac{r}{b}\Big) &\text{ for } r \in (b,a),\\
	0 &\text{ otherwise.}
\end{array}\right.
\]
We have
\begin{align*}
Q_\eta^{MM}[0,q] 
	&\leq \int_b^a \Big\{\Big(\frac{\pi}{\ln\frac{a}{b}}\Big)^2 \cos^2 \Big(\frac{\pi}{\ln\frac{a}{b}} \ln \frac{r}{b}\Big) 
		 + \Big(\frac{N^2 - 8N + 8}{4} + \frac{\tilde W'(0)a^2}{\eta^2}\Big)\sin^2 \Big(\frac{\pi}{\ln\frac{a}{b}} \ln \frac{r}{b}\Big)\Big\}\,\frac{dr}{r} \\
	&= \frac{1}{2} \ln \frac{a}{b} \Big\{\Big(\frac{\pi}{\ln\frac{a}{b}}\Big)^2 + \frac{N^2 - 8N + 8}{4} + \frac{\tilde W'(0)a^2}{\eta^2}\Big\}.
\end{align*}
Noting that $ \frac{N^2 - 8N + 8}{4} < 0$ for $3 \leq N \leq 6$, we can select $0 \ll b \ll a \ll \eta$ such that the above quantity is negative.

\medskip
\noindent\underline{Proof of the existence of an escaping solution:} Minimizing $I_{\eta}^{MM}$ in $\mcB^{MM}$, we obtain a minimizer $(\tilde f_\eta,g_\eta) \in \mcB^{MM}$. Replacing $(\tilde f_\eta,g_\eta)$ by $(|\tilde f_\eta|,|g_\eta|)$ if necessary, we have $\tilde f_\eta \geq 0$ and $g_\eta \geq 0$. It is readily seen that $(\tilde f_\eta,g_\eta)$ satisfies \eqref{Eq:metaform}--\eqref{Eq:MM-feegeeBC}. By \eqref{Eq:MM-fee}, the fact that $\tilde f_\eta(1) = 1$ and the strong maximum principle, $\tilde f_\eta > 0$ in $(0,1)$. By \eqref{Eq:MM-gee} and the strong maximum principle, either $g_\eta > 0$ or $g_\eta \equiv 0$ in $(0,1)$. The case $g_\eta \equiv 0$ cannot hold since it would imply $\tilde f_\eta \equiv 1$ in $(0,1)$ (since $\tilde f_\eta^2 + g_\eta^2 = 1$, $\tilde f_\eta(1) = 1$ and $\tilde f_\eta \in C((0,1])$) and $N \geq 3$ (since $r^{\frac{N-3}{2}}\tilde f_\eta \in L^2(0,1)$), which contradicts the instability statement established above. \end{proof}

\begin{remark}
In dimension $N = 2$, if we define the second variation of $I_\eta^{MM}$ at $(1,0)$ (in $\mcB^{MM}$) along directions $(0,q)$ compactly supported in $(0,1)$ by
\[
Q_\eta^{MM}[0,q] = \int_0^1 \Big[(q')^2 - \frac{N-1}{r^2}q^2 + \frac{\tilde W'(0)}{\eta^2} q^2\Big]\,r^{N-1}\,dr,
\]
then the same proof above yields a perturbation $q \in Lip_c(0,1)$ such that
\be
Q_\eta^{MM}[0,q] < 0.
\label{new-pp}
\ee
\end{remark}

\begin{remark}\label{Rem:EqInst}
One can also prove Theorem \ref{Thm:MMExist} by considering the limit as $\eps \rightarrow 0$ of the escaping (minimizing) solutions $(f_{\eps,\eta} > 0,g_{\eps,\eta} > 0)$ obtained in Theorem \ref{Thm:ExtendedExist} for a fixed $\eta > 0$ with $W(t) = t^2$. The strong limit $(\tilde f_\eta, g_\eta)$ of  $\{(f_{\eps,\eta}, g_{\eps,\eta})\}_{\eps\to 0}$ in $\mcB$ is indeed escaping because the non-escaping solution $(1,0)$ (which corresponds to the equator map $\nespm(x) = (n(x),0)$) is unstable for $I_\eta^{MM}$. 
\end{remark}

\noindent\underline{Proof of the convergence of $(f_{\eps,\eta},g_{\eps,\eta})$ in $\mcB$ when $W(t) = t^2$.} By the minimality of $(f_{\eps,\eta},g_{\eps,\eta})$ for $I_{\eps,\eta}$, we have
\[
I_{\eps,\eta}[f_{\eps,\eta},g_{\eps,\eta}] \leq I_{\eps,\eta}[f,g] = I_\eta^{MM}[f,g] \text{ for all } (f,g) \in \mcB^{MM}.
\]
Recall the expression of $I_{\eps,\eta}$, we see that the sequence $\{m_{\eps,\eta} = (f_{\eps,\eta}n,g_{\eps,\eta})\}_{\eps  > 0}$ is bounded in $H^1(B^N)$ and $(1 - f_{\eps,\eta}^2 - g_{\eps,\eta}^2)^2 = W(1 - f_{\eps,\eta}^2 - g_{\eps,\eta}^2) \rightarrow 0$ in $L^1(B^N,\RR^{N+1})$. Thus, along a sequence $\eps_j \rightarrow 0$, $m_{\eps_j,\eta}$ converges weakly in $H^1(B^N,\RR^{N+1})$, strongly in $L^2(B^N,\RR^{N+1})$ and uniformly on compact subsets of $\bar B^N \setminus \{0\}$ to some limit $m_* = (f_*n, g_*)\in \mcA^{MM}$ satisfying $f_* \geq 0$, $g_* \geq 0$. Furthermore, 
\[
I_\eta^{MM}[f_*,g_*] \leq \liminf_{j \rightarrow \infty} I_{\eps_j,\eta}[f_{\eps_j,\eta},g_{\eps_j,\eta}] \leq I_\eta^{MM}[f,g] \text{ for all } (f,g) \in \mcB^{MM}.
\]
Hence $(f_*, g_*)$ is a minimizer for $I_\eta^{MM}$ in $\mcB^{MM}$. Also, by taking $(f,g) = (f_*,g_*)$ in the above inequality, we get
\[
I_\eta^{MM}[f_*,g_*] = \lim_{j \rightarrow \infty} I_{\eps_j,\eta}[f_{\eps_j,\eta},g_{\eps_j,\eta}].
\]
By inspecting the chain of equality, we also have $\|\nabla m_{\eps_j,\eta}\|_{L^2(B^N,\RR^{N+1})} \rightarrow \|\nabla m_*\|_{L^2(B^N)}$. This together with the weak convergence of $m_{\eps_j,\eta}$ in $H^1$ implies that $m_{\eps_j,\eta}$ in fact converges strongly in $H^1(B^N,\RR^{N+1})$ to $m_*$.

As explain above, we have $(f_*,g_*)$ is an escaping critical point of $I_{\eta}^{MM}$ and by its uniqueness in Proposition \ref{Prop:MMUniqueness}, it is thus independent of the sequence $(\eps_j)$. We deduce that $m_{\eps,\eta}$ converges strongly in $H^1(B^N,\RR^{N+1})$ as $\eps \rightarrow 0$ to $m_* = m_\eta$, i.e. $(f_{\eps,\eta},g_{\eps,\eta})$ converges strongly $\mcB$ to $(\tilde f_\eta, g_\eta)$.

\section{Stability analysis of vortex solutions}\label{Sec:Stab}

\subsection{An orthogonal decomposition for the second variation in the extended model}\label{SSec:OD}

Assume that $N\geq 2$ and $W \in C^2((-\infty,1]) $ and $\tilde W \in C^2([0,\infty))$. Let $m_{\eps,\eta} = (f_{\eps,\eta} n, g_{\eps,\eta})$ be any (bounded) radially symmetric critical point of $E_{\eps,\eta}$ in $\mcA$, and define the second variation $Q_{\eps,\eta}: H_0^1(B^N,\RR^{N+1}) \rightarrow \RR$ of $E_{\eps,\eta}$ at $m_{\eps,\eta}$ as follows. Under our assumptions on $W$ and $\tilde W$, $E_{\eps,\eta}$ may take on infinite value in any neighborhood of $m_{\eps,\eta}$. To bypass this technical matter, we first define the second variation $Q_{\eps,\eta}[V]$ along a direction $V = (v, q) \in C_c^\infty(B^N \setminus \{0\}, \RR^N) \times C_c^\infty(B^N \setminus \{0\},\RR) \cong C_c^\infty(B^N  \setminus \{0\},\RR^{N+1})$ by
\begin{align}
Q_{\eps,\eta}[V] 
	&= \frac{d^2}{dt^2}\Big|_{t = 0} E_{\eps,\eta}[m_{\eps,\eta} + tV]\nonumber\\
	&= \int_{B^N} \Big[|\nabla v|^2 + |\nabla q|^2
		- \frac{1}{\eps^2}W'(1 - f_{\eps,\eta}^2 - g_{\eps,\eta}^2)(|v|^2 + q^2) + \frac{1}{\eta^2} \tilde W'(g_{\eps,\eta}^2) q^2\nonumber\\
		&\qquad
		+  \frac{2}{\eps^2} W''(1 - f_{\eps,\eta}^2 - g_{\eps,\eta}^2)(f_{\eps,\eta} n \cdot v + g_{\eps,\eta} q)^2 + \frac{2}{\eta^2} \tilde W''(g_{\eps,\eta}^2) g_{\eps,\eta}^2 q^2
		\Big]\,dx,
		\label{Eq:09IV19-2VarX1}
\end{align}
and extend this definition to $V \in H_0^1(B^N,\RR^{N+1})$ by density using the fact that the right hand side of \eqref{Eq:09IV19-2VarX1} is continuous $H_0^1(B^N,\RR^{N+1})$ (because $f_{\eps,\eta}, g_{\eps,\eta} \in L^\infty(B^N)$ and $W$ and $\tilde W$ are twice continuously differentiable). We will see that this definition is appropriate for our proof of the local minimality of the escaping critical points.

In the sequel $A \colon B$ denotes the Frobenius scalar product of matrices. Writing $v = s n + w$ where $w \cdot n = 0$ with $s\in C^\infty_c(B^N \setminus \{0\},\RR)$ and $w\in C^\infty_c(B^N\setminus \{0\},\RR^N)$, we compute
\[
|\nabla v|^2 = |\nabla s|^2 + \frac{N-1}{r^2} s^2 + |\nabla w|^2 +  2\nabla(sn) \colon \nabla w
\]
and
\begin{align*}
\int_{B^N} \nabla(sn) \colon \nabla w\,dx 
	&= 	- \int_{B^N}  \Delta(sn) \cdot  w\,dx
		= 	- 2\int_{B^N}  \nabla s \cdot( (\nabla n)^t  w)\,dx
	= 	- \int_{B^N} \frac{2}{r} (w \cdot \nabla) s \,dx,
\end{align*}
where we used $w \cdot \partial_k n =\frac{w_k}r$ for $1\leq k\leq N$ because $w\cdot n=0$. It follows that
\begin{align*}
Q_{\eps,\eta}[V] 
	&= \int_{B^N} \Big[ |\nabla s|^2 + \frac{N-1}{r^2} s^2 + |\nabla w|^2  - \frac{4}{r} (w \cdot \nabla) s  + |\nabla q|^2 \nonumber\\
		&\qquad
		 -\frac{1}{\eps^2}W'(1 - f_{\eps,\eta}^2 - g_{\eps,\eta}^2)(s^2 + |w|^2 + q^2) 
		 	+ \frac{1}{\eta^2} \tilde W'(g_{\eps,\eta}^2) q^2\nonumber\\
		 &\qquad
		+  \frac{2}{\eps^2} W''(1 - f_{\eps,\eta}^2 - g_{\eps,\eta}^2)(f_{\eps,\eta} s + g_{\eps,\eta} q)^2 
		+ \frac{2}{\eta^2} \tilde W''(g_{\eps,\eta}^2) g_{\eps,\eta}^2 q^2
		\Big]\,dx.
\end{align*}

We identify  $x = (r,\theta)$ where $r = |x| \geq 0$ and $\theta = \frac{x}{|x|} \in \mathbb{S}^{N-1}$. Let $\slashed{D}$ denote the covariant derivative of the standard metric $g_{\rm round}$ on the unit sphere $\mathbb{S}^{N-1}$ and $d\sigma$ denote the surface measure on $\mathbb{S}^{N-1}$. For a tangent vector field $w$ on $\Sphere^{N-1}$ (i.e., $w\cdot n=0$), one computes
\begin{align}
\label{grad_spher}
|\nabla w|^2
=|\partial_r w|^2+\frac{1}{r^2}(|w|^2+ |\slashed{D}w|^2).
\end{align}
We have
\begin{align}
&Q_{\eps,\eta}[V] 
	=  \int_0^1 \int_{\mathbb{S}^{N-1}}  r^{N-1}\Big\{(\partial_r s)^2 + \frac{1}{r^2} |\slashed{D}s|^2 + \frac{N-1}{r^2} s^2 
		 + |\partial_r w|^2 + \frac{1}{r^2} |\slashed{D}w|^2
			+ \frac{1}{r^2} |w|^2  \nonumber\\
			&  - \frac{4}{r^2} (w \cdot \slashed{D}) s +   (\partial_r q)^2 + \frac{1}{r^2} |\slashed{D} q|^2
		-\frac{1}{\eps^2}W'(1 - f_{\eps,\eta}^2 - g_{\eps,\eta}^2)(s^2 + |w|^2 + q^2) 
		 	+ \frac{1}{\eta^2} \tilde W'(g_{\eps,\eta}^2) q^2\nonumber\\
		 & 
		+  \frac{2}{\eps^2} W''(1 - f_{\eps,\eta}^2 - g_{\eps,\eta}^2)(f_{\eps,\eta} s + g_{\eps,\eta} q)^2 
		+ \frac{2}{\eta^2} \tilde W''(g_{\eps,\eta}^2) g_{\eps,\eta}^2 q^2
		  \Big\} d\sigma\,dr
	 \label{Eq:09IV19-2VarX3}.
\end{align}

We start with an orthogonal decomposition for $Q_{\eps,\eta}$. Let $\lambda_0 = 0 < \lambda_1 \leq \lambda_2 \leq \ldots \rightarrow \infty$ be the eigenvalues of the Laplacian $-\slashed{\Delta}$ on $\mathbb{S}^{N-1}$, and let $\zeta_0, \zeta_1, \ldots$ be a corresponding orthonormal eigenbasis of $L^2(\mathbb{S}^{N-1})$. In particular, $\lambda_k=N-1$ for $k=1, \dots, N$, $\lambda_k \geq 2N$ for $k \geq N + 1$, and the first $N+1$ eigenfunctions can be taken as
\[
\zeta_0(\theta) = \frac{1}{\sqrt{|\Sphere^{N-1}|}}, \quad \zeta_k(\theta) = \sqrt{\frac{N}{|\Sphere^{N-1}|}} \theta_k, \qquad 1 \leq k \leq N.
\]
Moreover, $\int_{\mathbb{S}^{N-1}} \zeta_k d\sigma = 0$ for all $k \geq 1$.

\begin{proposition}
\label{Prop:OrthoDecomp}
Assume $N\geq 2$, $W \in C^2((-\infty,1])$, $\tilde W \in C^2([0,\infty))$. Let $m_{\eps,\eta} = (f_{\eps,\eta} n, g_{\eps,\eta})$ be a radially symmetric critical point of $E_{\eps,\eta}$ in $\mcA$ and $Q_{\eps,\eta}$ be the second variation of $E_{\eps,\eta}$ at $m_{\eps,\eta}$ defined by \eqref{Eq:09IV19-2VarX1}. Suppose that $V = (v = sn + w,q) \in C_c^\infty(B^N\setminus\{0\},\RR^{N+1})$ with $w \cdot n = 0$. For $r \in (0,1]$, let 
\begin{itemize}
\item $w(r,\cdot) = \mathring{w}(r,\cdot) + \slashed{D} \psi(r,\cdot)$ be the Helmholtz decomposition of $w(r,\cdot)$ as a tangent vector field on $\mathbb{S}^{N-1}$ so that the tangent vector field $ \mathring{w}\in C^\infty_c(B^N\setminus\{0\}, \RR^N)$ with vanishing covariant divergence $\slashed{D} \cdot \mathring{w}(r,\cdot) = 0$ and $\psi \in C^\infty_c(B^N\setminus\{0\}, \RR)$ with $\int_{\mathbb{S}^{N-1}} \psi(r,\theta) d\sigma = 0$, where we use the convention that $\mathring{w} = 0$ when $N = 2$;
\item the expansions of $s(r,\theta), \psi(r,\theta)$ and $q(r,\theta)$ in the basis $\{\zeta_i\}_{i=0}^\infty$ be 
\be
\label{decompo}
s(r,\theta) = \sum_{i=0}^\infty s_i(r) \zeta_i(\theta), \quad \psi(r,\theta) = \sum_{i=0}^\infty \psi_i(r) \zeta_i(\theta), \quad q(r,\theta) = \sum_{i=0}^\infty q_i(r) \zeta_i(\theta),
\ee
with $s_i, \psi_i, q_i\in C^\infty_c\big((0,1)\big)$ for every $i\geq 0$.\footnote{\label{foot12}Note that $\psi_0=0$ since $\psi(r,\cdot)$ as well as $\zeta_i$ have zero average on $\bS^{N-1}$ for $i\geq 1$.}
\end{itemize}
Then
$
\mathring{V} := (\mathring{w},0), V_i := (s_i \zeta_i n +  \psi_i \slashed{D} \zeta_i, q_i \zeta_i)$ belong to $C_c^\infty(B^N\setminus\{0\},\RR^{N+1}) \textrm{ for } i\geq 0,
$
and 
\begin{equation}
Q_{\eps,\eta}[V] = Q_{\eps,\eta}[\mathring{V}] + \sum_{i=0}^\infty Q_{\eps,\eta}[V_i].
	\label{Eq:QeeODecomp}
\end{equation}
\end{proposition}

For related decomposition see \cite{Pino-Felmer-Kow, Gustafson, Mironescu-radial} (in the context of the Ginzburg--Landau functional), \cite{HangLin01-ActaSin, LiMelcher18-JFA} (in the context of micromagnetics), \cite{INSZ3, INSZ_CVPDE} (in the context of the Landau--de Gennes functional).

\begin{proof} 
Observe that for a tangent vector field $w$ (i.e., $w\cdot n=0$),
\begin{equation}
\int_{\mathbb{S}^{N-1}}(w \cdot \slashed{D}) s\,d\sigma
	=- \int_{\mathbb{S}^{N-1}}  \slashed{D} \cdot w s\,d\sigma.
	\label{Eq:04VIII18-E5}
\end{equation}
Hence, in the coupling term $(w \cdot \slashed{D}) s$ between $s$ and $w$ in the expression for $Q_{\eps,\eta}[V]$ in \eqref{Eq:09IV19-2VarX3}, the divergence-free part of the tangent vector field $w$ does not contribute. If $w = \mathring{w} + \slashed{D} \psi$ is the Helmholtz decomposition of $w$ with $\slashed{D} \cdot \mathring{w} = 0$ and $\int_{\mathbb{S}^{N-1}} \psi d\sigma = 0$, then, by \eqref{Eq:04VIII18-E5},
\begin{align*}
\int_{\mathbb{S}^{N-1}} |w|^2\,d\sigma
	&= \int_{\mathbb{S}^{N-1}} |\mathring{w}|^2\,d\sigma + \int_{\mathbb{S}^{N-1}} |w - \mathring{w}|^2\,d\sigma,\\
\int_{\mathbb{S}^{N-1}} |\partial_r w|^2\,d\sigma
	&= \int_{\mathbb{S}^{N-1}} |\partial_r \mathring{w}|^2\,d\sigma + \int_{\mathbb{S}^{N-1}} |\partial_r(w - \mathring{w})|^2\,d\sigma,\\
\int_{\mathbb{S}^{N-1}} (w \cdot \slashed{D}) s\,d\sigma
	&=  \int_{\mathbb{S}^{N-1}} ((w - \mathring{w}) \cdot \slashed{D}) s\,d\sigma,\\
	\int_{\mathbb{S}^{N-1}} |\slashed{D}w|^2\,d\sigma
	&= \int_{\mathbb{S}^{N-1}} |\slashed{D}\mathring{w}|^2\,d\sigma + \int_{\mathbb{S}^{N-1}} |\slashed{D}(w - \mathring{w})|^2\,d\sigma\\
	&= \int_{\mathbb{S}^{N-1}} |\slashed{D}\mathring{w}|^2\,d\sigma + \int_{\mathbb{S}^{N-1}} [(\slashed{\Delta}\psi)^2 - (N-2)|\slashed{D}\psi|^2]\,d\sigma
	\end{align*}
where we used the Bochner identity on the sphere (see e.g. \cite[Chapter I, Proposition 2.2]{SchoenYan-LectDG})
$$
\int_{\mathbb{S}^{N-1}} |\slashed{D}^2 \psi|^2\,d\sigma
	=\int_{\mathbb{S}^{N-1}} [(\slashed{\Delta}\psi)^2 - (N-2)|\slashed{D}\psi|^2]\,d\sigma,
$$
with $\slashed{D}^2\psi$ and $\slashed{\Delta}\psi$ standing for the covariant Hessian and Laplacian of $\psi$, respectively.
Summing up and using  \eqref{decompo}, the Dirichlet part in $Q_{\eps, \eta}[V]$ in \eqref{Eq:09IV19-2VarX3} becomes:
\begin{align*}
&\textrm{Dir}:=\int_{\mathbb{S}^{N-1}}  r^{N-1}\Big\{(\partial_r s)^2  
		 +   (\partial_r q)^2  + |\partial_r w|^2 + \frac{(N-1)s^2+|\slashed{D}s|^2+|\slashed{D} q|^2+|\slashed{D}w|^2+|w|^2-4(w \cdot \slashed{D}) s}{r^2}  
			 			  \Big\} d\sigma\\
		&= \int_{\mathbb{S}^{N-1}}  r^{N-1}\Big\{ |\partial_r \mathring{w}|^2 + \frac{1}{r^2} |\slashed{D}\mathring{w}|^2 + \frac{1}{r^2} |\mathring{w}|^2
			 +(\partial_r s)^2 + \frac{1}{r^2} |\slashed{D}s|^2 + \frac{N-1}{r^2} s^2 
				 +   (\partial_r q)^2 \nonumber\\
			&\quad 
			+ \frac{1}{r^2} |\slashed{D} q|^2	+ |\partial_r \slashed{D}\psi|^2 + \frac{1}{r^2} (\slashed{\Delta} \psi)^2 - \frac{N-3}{r^2} |\slashed{D}\psi|^2
				 - \frac{4}{r^2} \slashed{D}\psi \cdot \slashed{D} s
			  \Big\} d\sigma\\
		&= \int_{\mathbb{S}^{N-1}}  r^{N-1}\Big\{ |\partial_r \mathring{w}|^2 + \frac{1}{r^2} |\slashed{D}\mathring{w}|^2 + \frac{1}{r^2} |\mathring{w}|^2
			  \Big\} d\sigma\\
			  &
			  	+  \sum_{i=0}^\infty r^{N-1}\Big\{(s_i')^2  + \frac{\lambda_i + N-1}{r^2} s_i^2 
				 +   (q_i')^2 + \frac{\lambda_i}{r^2} q_i^2
				 + \lambda_i (\psi_i')^2 + \frac{\lambda_i(\lambda_i - N + 3)}{r^2} \psi_i^2
				 - \frac{4\lambda_i}{r^2} \psi_i   s_i
			  \Big\}.			 
\end{align*}
Noting that, as $\lambda_i(\lambda_i + N-1)(\lambda_i - N + 3) - 4\lambda_i^2 = \lambda_i(\lambda_i + N - 3)(\lambda_i - N + 1) \geq 0$ for $\lambda_i \geq N - 1$, which holds for $i \geq 1$, we have
\[
\frac{\lambda_i + N-1}{r^2} x^2 
	+ \frac{\lambda_i(\lambda_i - N + 3)}{r^2} y^2
	- \frac{4\lambda_i}{r^2} xy \geq 0 \text{ for all } i \geq 1.
\]
Recall also that $\psi_0 \equiv 0$ and $\lambda_0=0$. Hence, all the summands on the right hand side of the identity above are non-negative. Hence, by Fubini-Tonelli's theorem, we obtain the following formula for $Q_{\eps, \eta}[V]$ in \eqref{Eq:09IV19-2VarX3}:
\begin{align*}
&Q_{\eps, \eta}[V]=\int_0^1 \textrm{Dir} \, dr+ \int_0^1 \int_{\mathbb{S}^{N-1}}  r^{N-1}\Big\{-\frac{1}{\eps^2}W'(1 - f_{\eps,\eta}^2 - g_{\eps,\eta}^2)(s^2 + |w|^2 + q^2) + \frac{1}{\eta^2} \tilde W'(g_{\eps,\eta}^2) q^2
		 	\nonumber\\
		 & 
		+  \frac{2}{\eps^2} W''(1 - f_{\eps,\eta}^2 - g_{\eps,\eta}^2)(f_{\eps,\eta} s + g_{\eps,\eta} q)^2 
		+ \frac{2}{\eta^2} \tilde W''(g_{\eps,\eta}^2) g_{\eps,\eta}^2 q^2
		  \Big\} d\sigma\,dr
		   =Q_{\eps,\eta}[\mathring{V}] + \sum_{i=0}^\infty Q_{\eps,\eta}[V_i],
\end{align*}
because the same computation as for the Dirichlet energy $\textrm{Dir}$ yields %
\begin{align*}
\|\nabla \mathring{V}\|_{L^2(B^N,\RR^{N+1})}^2
	&= \int_0^1\int_{\mathbb{S}^{N-1}}  r^{N-1}\Big\{ |\partial_r \mathring{w}|^2 + \frac{1}{r^2} |\slashed{D}\mathring{w}|^2 + \frac{1}{r^2} |\mathring{w}|^2
			  \Big\} d\sigma\,dr < \infty,\\
\|\nabla V_i\|_{L^2(B^N,\RR^{N+1})}^2
	&= \int_0^1  r^{N-1}\Big\{(s_i')^2  + \frac{\lambda_i + N-1}{r^2} s_i^2 
				 +   (q_i')^2 + \frac{\lambda_i}{r^2} q_i^2\\
			&\qquad
				+ \lambda_i (\psi_i')^2 + \frac{\lambda_i(\lambda_i - N + 3)}{r^2} \psi_i^2
				 - \frac{4\lambda_i}{r^2} \psi_i   s_i
			  \Big\}\,dr < \infty
\end{align*} 
which finally gives the expressions of  $Q_{\eps,\eta}[\mathring{V}]$ and $Q_{\eps,\eta}[V_i]$ used above:
\begin{align}
Q_{\eps,\eta}[\mathring{V}] 
	&=  \int_0^1 \int_{\mathbb{S}^{N-1}}  r^{N-1}\Big\{|\partial_r \mathring{w}|^2 + \frac{ |\slashed{D}\mathring{w}|^2+ |\mathring{w}|^2}{r^2} 
		- \frac{W'(1 - f_{\eps,\eta}^2 - g_{\eps,\eta}^2) |\mathring{w}|^2 }{\eps^2}  \Big\} d\sigma\,dr
	,\nonumber\\
Q_{\eps,\eta}[V_i] 
	&=  \int_0^1  r^{N-1}\Big\{(s_i')^2 + \frac{\lambda_i + N-1}{r^2} s_i^2 + \lambda_i (\psi_i')^2 + \frac{\lambda_i(\lambda_i - N + 3)}{r^2} \psi_i^2- \frac{4\lambda_i\psi_i\, s_i}{r^2}
			  +   (q_i')^2 \nonumber\\
		&\qquad   
			 + \frac{\lambda_i}{r^2} q_i^2 -\frac{1}{\eps^2}W'(1 - f_{\eps,\eta}^2 - g_{\eps,\eta}^2)(s_i^2 + \lambda_i \psi_i^2 + q_i^2) 
		 	+ \frac{1}{\eta^2} \tilde W'(g_{\eps,\eta}^2) q_i^2\nonumber\\
		 &\qquad
		+  \frac{2}{\eps^2} W''(1 - f_{\eps,\eta}^2 - g_{\eps,\eta}^2)(f_{\eps,\eta} s_i + g_{\eps,\eta} q_i)^2 
		+ \frac{2}{\eta^2} \tilde W''(g_{\eps,\eta}^2) g_{\eps,\eta}^2 q_i^2
		  \Big\}\,dr.
	\label{Eq:10IV19-2VarX5}
\end{align}
Thus, \eqref{Eq:QeeODecomp}  holds.
\end{proof}

\nd {\bf Strategy of the proof of the stability/instability}. The aim is to study the positivity of the terms in the decomposition of $Q_{\eps, \eta}[V]$ in \eqref{Eq:QeeODecomp}. For that, we will use the Hardy decomposition \cite[Lemma A.1]{INSZ3} for the two operators $L$ and $T$ defined in \eqref{opLT} (as in the proof of 
Proposition~\ref{pro:unique_eps_eta}). By the equations \eqref{Eq:20III21-fee}--\eqref{Eq:20III21-gee}, one easily computes for $\alpha\in \RR$:
\be
\label{comput11}
\begin{cases}
&L(r^\alpha f_{\eps, \eta})=-2\alpha r^{\alpha-1} f'_{\eps, \eta}-\Big(\alpha(\alpha+N-2)+N-1\Big)r^{\alpha-2} f_{\eps, \eta},\\
&L(f_{\eps,\eta}')=-\frac{2(N-1)}{r^2}f'_{\eps, \eta}+\frac{2(N-1)}{r^3}f_{\eps, \eta}-\frac2{\eps^2} W''(1 - f_{\eps,\eta}^2 - g_{\eps,\eta}^2)(f^2_{\eps,\eta}f'_{\eps, \eta}+f_{\eps, \eta} g_{\eps, \eta}g'_{\eps, \eta}),\\
&Tg_{\eps,\eta}=0,\\
& Tg'_{\eps,\eta}=-\frac{N-1}{r^2}g'_{\eps, \eta}-\frac2{\eps^2} W''(1 - f_{\eps,\eta}^2 - g_{\eps,\eta}^2)(g_{\eps, \eta}f_{\eps,\eta}f'_{\eps, \eta}+g^2_{\eps, \eta} g'_{\eps, \eta})-\frac2{\eta^2} \tilde W''(g_{\eps,\eta}^2)g^2_{\eps, \eta}g'_{\eps, \eta},
\end{cases}
\ee
paying attention to the differences in the cases $g_{\eps,\eta} > 0$ and $g_{\eps,\eta} \equiv 0$.

\bigskip

\nd \underline{Stability in direction $\mathring{V}=(\mathring{w},0)$.}

\begin{lemma}\label{Lem:2VarDivFree}
Suppose $N \geq 3$ and $W \in C^2((-\infty,1]) $ and $\tilde W \in C^2([0,\infty))$ satisfy \eqref{Eq:WCond'} and \eqref{Eq:TWCond'}. Let $m_{\eps,\eta} = (f_{\eps,\eta} n, g_{\eps,\eta})$ be a radially symmetric critical point  of $E_{\eps,\eta}$ in $\mcA$ with $g_{\eps,\eta} \geq 0$ in $(0,1)$, and let $Q_{\eps,\eta}$ be the second variation of $E_{\eps,\eta}$ at $m_{\eps,\eta}$ defined by \eqref{Eq:09IV19-2VarX1}. Then there exists a constant $C>0$ independent of $\eps, \eta>0$ such that for every $\mathring{w} \in C_c^\infty(B^N \setminus \{0\}, \RR^{N})$ with $\mathring{w} \cdot n = 0$ and $\slashed{D}\cdot \mathring{w} = 0$:
\[
Q_{\eps,\eta}[(\mathring{w},0)] \geq C \int_{B^N} 
|\mathring{w}|^2\, dx.
\]
\end{lemma}

To be clear, in the lemma above, $m_{\eps,\eta}$ can be either an escaping solution with $g_{\eps,\eta} > 0$ or a non-escaping solution with $g_{\eps,\eta} \equiv 0$. Also, in dimension $N = 2$, this inequality is obvious since $\mathring{w} = 0$ by definition.

\begin{proof} Note that $f_{\eps,\eta}> 0$ by Proposition \ref{Prop:fPos}. Let $\alpha\in \RR$ to be chosen later (see \eqref{choice_1} at the end of the proof). We factor $\mathring{w} = r^\alpha f_{\eps,\eta} \hat w$ with $\hat w=(\hat w_1, \dots, \hat w_N)\in C_c^\infty(B^N \setminus \{0\}, \RR^{N})$ and we apply \cite[Lemma A.1]{INSZ3} for the operator $L$ in \eqref{opLT}:
\begin{align}
\nonumber Q_{\eps,\eta}[(\mathring{w},0)] &= \int_{B^N}  \sum_{j=1}^N L \mathring{w}_j\cdot 
\mathring{w}_j 
\, dx \\
\nonumber&=\sum_{j=1}^N \int_{B^N} \,\Big\{ r^{2\alpha}f_{\eps,\eta}^2 |\nabla \hat w_j|^2 + \hat w_j^2 L(r^\alpha f_{\eps,\eta})\cdot (r^\alpha f_{\eps,\eta})\Big\}\, dx\\
\nonumber&= \int_0^1 \int_{\mathbb{S}^{N-1}}  r^{2\alpha + N-1}f_{\eps,\eta}^2 \Big\{ |\partial_r \hat w|^2 - \frac{2\alpha f_{\eps, \eta}'}{r f_{\eps, \eta}}  |\hat w|^2- \frac{(\alpha + 1)(\alpha + N-3)}{r^2}   |\hat w|^2\\
		&\qquad 
		 + \frac{1}{r^2} (|\slashed{D}\hat w|^2 - |\hat w|^2)
			  \Big\} d\sigma\,dr,
\label{123}
\end{align}
because of \eqref{grad_spher} for the tangent vector field $\hat w$ and \eqref{comput11}.  
By the Poincar\'e inequality for divergence-free vector field on the sphere (see Lemma \ref{Lem:SolPoincare}), we have
\[
\int_{\mathbb{S}^{N-1}} |\slashed{D}\hat w|^2 \,d\sigma \geq (N-2)\int_{\mathbb{S}^{N-1}} |\hat w|^2 d\sigma.
\]
We then choose $\alpha\in (-(N-2),0)$ yielding
\be
\label{choice_1}
\alpha<0\quad \textrm{ and } \quad (\alpha + 1)(\alpha + N-3) < N-3.
\ee
Since $f_{\eps, \eta}'>0$ (see Proposition \ref{Prop:ODEMonotonicity}) and $\frac1{r^2}>1$ in $(0,1)$, it follows $Q_{\eps,\eta}[(\mathring{w},0)]\geq C\|\mathring{w}\|_{L^2}^2$ for a constant $C>0$ independent of $\eps, \eta>0$. The lemma is proved.
 \end{proof}

\subsection{The extended model: Stability of the escaping vortex solution}

\nd \underline {Stability for the zero-mode $V_0$.}

Recall that $\lambda_0=0$ and $\zeta_0$ is a nonzero constant that satisfies $\|\zeta_0\|_{L^2(\bS^{N-1})}=1$, in particular, $\slashed{D}\zeta_0=0$; thus, the zero-mode in \eqref{Eq:QeeODecomp} is given by $V_0=(s \zeta_0 n,q \zeta_0)$ for two functions $s, q \in C_c^\infty(0,1)$.

\begin{lemma}\label{Lem:2VarZero}
Let $N \geq 2$, $W \in C^2((-\infty,1])$ and $\tilde W \in C^2([0,\infty))$. Let $m_{\eps,\eta} = (f_{\eps,\eta} n, g_{\eps,\eta})$ be a bounded radially symmetric critical point of $E_{\eps,\eta}$ in $\mcA$ and let $Q_{\eps,\eta}$ be the second variation of $E_{\eps,\eta}$ at $m_{\eps,\eta}$ defined by \eqref{Eq:09IV19-2VarX1}. Suppose that $f_{\eps,\eta} > 0$ and $g_{\eps,\eta} > 0$ in $(0,1)$. If $(s, q) \in C_c^\infty(0,1)$,
then
\begin{align*}
Q_{\eps,\eta}[(s \zeta_0 n,q \zeta_0)] 
	&= \int_0^1 r^{N-1} \Big\{f_{\eps,\eta}^2 \Big|\Big(\frac{s}{f_{\eps,\eta}}\Big)'\Big|^2
	+ g_{\eps,\eta}^2 \Big|\Big(\frac{q}{g_{\eps,\eta}}\Big)'\Big|^2 \\
		&\quad +  \frac{2}{\eps^2} W''(1 - f_{\eps,\eta}^2 - g_{\eps,\eta}^2)(f_{\eps,\eta} s + g_{\eps,\eta} q)^2 
		+ \frac{2}{\eta^2} \tilde W''(g_{\eps,\eta}^2) g_{\eps,\eta}^2 q^2
		\Big\}\,dr.
\end{align*}
\end{lemma}

\begin{proof} 
Recalling the operators $L$ and $T$ defined in \eqref{opLT},
by \eqref{Eq:10IV19-2VarX5},
\begin{align*}
Q_{\eps,\eta}[(s \zeta_0 n,q \zeta_0)] 
	&= \frac1{|\bS^{N-1}|}\int_{B^N}  \Big\{Ls\cdot s + \frac{N-1}{|x|^2} s^2  
	 	 +   Tq\cdot q\\
				 &\qquad
		+  \frac{2}{\eps^2} W''(1 - f_{\eps,\eta}^2 - g_{\eps,\eta}^2)(f_{\eps,\eta} s + g_{\eps,\eta} q)^2 
		+ \frac{2}{\eta^2} \tilde W''(g_{\eps,\eta}^2) g_{\eps,\eta}^2 q^2
		  \Big\} \,dx.
\end{align*}
We factor $s = f_{\eps,\eta} \hat s$ and $q = g_{\eps,\eta} \hat q$ and \eqref{comput11} combined with \cite[Lemma A.1]{INSZ3} yields the conclusion. (For details, see \eqref{Eq:10IV21-E1}.)
\end{proof}

\bigskip

\nd \underline{Stability for the modes $V_i$, $i\geq 1$.}

\begin{lemma}\label{Lem:2VarRest}
Assume $N \geq 2$ and $W \in C^2((-\infty,1])$ and $\tilde W \in C^2([0,\infty))$ satisfy \eqref{Eq:WCond'} and \eqref{Eq:TWCond'}. Let $m_{\eps,\eta} = (f_{\eps,\eta} n, g_{\eps,\eta})$ be a radially symmetric critical point of $E_{\eps,\eta}$ in $\mcA$ and let $Q_{\eps,\eta}$ be the second variation of $E_{\eps,\eta}$ at $m_{\eps,\eta}$ defined by \eqref{Eq:09IV19-2VarX1}. Suppose that $g_{\eps,\eta}> 0$ in $(0,1)$. 
If $s, \psi, q \in C_c^\infty(0,1)$ then, for $i \geq 1$ and $V_i = (s \zeta_i n + \psi \slashed{D}\zeta_i,q \zeta_i)$,
\begin{align*}
Q_{\eps,\eta}[V_i] 
	&\geq 
	\int_0^1 r^{N-1} \Big\{(f_{\eps,\eta}')^2 \Big|\Big(\frac{s}{f_{\eps,\eta}'}\Big)'\Big|^2
	+ \frac{\lambda_i}{r^2} f_{\eps,\eta}^2 \Big|\Big(\frac{r \psi}{f_{\eps,\eta}}\Big)'\Big|^2
	+ (g_{\eps,\eta}')^2 \Big|\Big(\frac{q}{g_{\eps,\eta}'}\Big)'\Big|^2 \\
	&\quad + \frac{2}{r^3} f_{\eps,\eta}\,f_{\eps,\eta}' \Big(\frac{\sqrt{N-1}\,s}{f_{\eps,\eta}'} - \frac{\sqrt{\lambda_i}\, r \,\psi}{f_{\eps,\eta}}\Big)^2 
	\Big\} \,dr 
	\geq 0.
\end{align*}
Moreover, there exists a constant $C>0$ independent of $\eps, \eta>0$ such that 
$$Q_{\eps,\eta}[V_i] \geq C\|V_i\|^2_{L^2(B^N)} \quad \textrm{ for every}\quad  i\geq N+1.$$ 
\end{lemma}

\begin{proof} By Proposition \ref{Prop:fPos}, $f_{\eps,\eta} > 0$ in $(0,1)$. By Proposition \ref{Prop:ODEMonotonicity} we have that $f_{\eps,\eta}' > 0$ and $g_{\eps,\eta}' < 0$ in $(0,1)$. We factor
\[
s = f_{\eps,\eta}' \hat s, \qquad \psi = \frac{f_{\eps,\eta}}{r} \hat \psi, \qquad \text{ and } \qquad  q = g_{\eps,\eta}' \hat q.
\]
Recalling the operators $L$ and $T$ defined in \eqref{opLT},
by \eqref{Eq:10IV19-2VarX5}, we have
\begin{align}
Q_{\eps,\eta}[V_i] &=  \frac1{|\bS^{N-1}|}\int_{B^N}
\Big\{Ls\cdot s  +\lambda_i L\psi\cdot \psi + Tq\cdot q+
						 \frac{\lambda_i +N-1}{r^2}  s^2 
			\nonumber\\
		&\qquad  
		+ \frac{\lambda_i(\lambda_i - N +3)}{r^2}  \psi^2 
		- \frac{4\lambda_i}{r^2}  s \psi+\frac{\lambda_i }{r^2}  q^2
		\nonumber\\
			\nonumber	&\qquad
			  +\frac{2}{\eps^2} W''(1 - f_{\eps,\eta}^2 - g_{\eps,\eta}^2)  (f_{\eps,\eta} s+  g_{\eps,\eta} q)^2+\frac{2}{\eta^2} \tilde W''(g_{\eps,\eta}^2)  g^2_{\eps,\eta} q^2\Big\} \,dx\\
	&=  \int_0^1  r^{N-1}\Big\{(f_{\eps,\eta}')^2 (\hat s')^2  
			+ \frac{2(N-1)}{r^3} f_{\eps,\eta}\,f_{\eps,\eta}' \hat s^2 
			+ \frac{\lambda_i - (N-1)}{r^2} (f_{\eps,\eta}')^2 \hat s^2 
			\nonumber\\
		&\qquad  + \frac{\lambda_i }{r^2} f_{\eps,\eta}^2 (\hat \psi')^2 
		+ \frac{2\lambda_i}{r^3} f_{\eps,\eta} f_{\eps,\eta}'  \hat \psi^2 
		+ \frac{\lambda_i(\lambda_i - (N - 1))}{r^4} f_{\eps,\eta}^2 \hat \psi^2 
		- \frac{4\lambda_i}{r^3} f_{\eps,\eta}' f_{\eps,\eta} \hat s \hat \psi
		\nonumber\\
		&\qquad
			 + (g_{\eps,\eta}')^2 (\hat q')^2 
			 + \frac{\lambda_i - (N-1)}{r^2} (g_{\eps,\eta}')^2 \hat q^2
		\nonumber\\
		&\qquad
			  - \frac{2}{\eps^2} W''(1 - f_{\eps,\eta}^2 - g_{\eps,\eta}^2)  f_{\eps,\eta} f_{\eps,\eta}'  g_{\eps,\eta} g_{\eps,\eta}'(\hat s - \hat q)^2\Big\} \,dr,
	\label{Eq:10IV19-2VarX6}
\end{align}
where we used \cite[Lemma A.1]{INSZ3} and \eqref{comput11}. Using $f_{\eps,\eta} > 0$, $f_{\eps,\eta}' > 0$, and $\lambda_i \geq N - 1$ for $i \geq 1$, we have
\[
\frac{\lambda_i - (N-1)}{r^2} (f_{\eps,\eta}')^2 \hat s^2 
			+ \frac{\lambda_i(\lambda_i - (N - 1))}{r^4} f_{\eps,\eta}^2 \hat \psi^2 
			\geq \frac{2 \sqrt{\lambda_i}(\lambda_i - (N-1))}{r^3} f_{\eps,\eta}\,f_{\eps,\eta}' |\hat s \hat \psi|.
\]
Also, for $i \geq 1$,
\[
4 \sqrt{\lambda_i (N-1)} 
	+ 2 \sqrt{\lambda_i}(\lambda_i - (N-1)) 
	- 4\lambda_i 
	= 2\sqrt{\lambda_i}[(\sqrt{\lambda_i} - 1)^2 - (\sqrt{N-1} - 1)^2] \geq 0,
\]
which implies
\[
\frac{2 \sqrt{\lambda_i}(\lambda_i - (N-1))}{r^3}  f_{\eps,\eta}\,f_{\eps,\eta}' |\hat s \hat \psi|
	\geq \frac{4\lambda_i  - 4 \sqrt{\lambda_i (N-1)}}{r^3} f_{\eps,\eta}\,f_{\eps,\eta}' \hat s \hat \psi .
\]
Putting these inequalities in \eqref{Eq:10IV19-2VarX6}, we conclude
\begin{align*}
Q_{\eps,\eta}[V_i] 
	&\geq  \int_0^1  r^{N-1}\Big\{(f_{\eps,\eta}')^2 (\hat s')^2  
		+ \frac{\lambda_i }{r^2} f_{\eps,\eta}^2 
		(\hat \psi')^2  + (g_{\eps,\eta}')^2 (\hat q')^2 \\
		&\quad
		+ \frac{2}{r^3} f_{\eps,\eta}\,f_{\eps,\eta}' (\sqrt{N-1} \hat s - \sqrt{\lambda_i} \hat \psi)^2 
			\Big\}\,dr.
\end{align*}
This proves the first assertion.

Consider the second assertion concerning the case $i\geq N+1$. We can prove a uniform $L^2$ lower bound by a different Hardy decomposition using the fact that $\lambda_i \geq 2N$. Indeed, 
we factor
$$
s = f_{\eps,\eta} \tilde s, \qquad \psi = f_{\eps,\eta} \tilde \psi, \qquad \text{ and } \qquad  q = g_{\eps,\eta} \tilde q$$
and we compute using \cite[Lemma A.1]{INSZ3} and \eqref{comput11}:
\begin{align}
Q_{\eps,\eta}[V_i] &=  \frac1{|\bS^{N-1}|}\int_{B^N}
\Big\{f_{\eps,\eta}^2 |\nabla \tilde s|^2 + \tilde s^2 Lf_{\eps,\eta}\cdot f_{\eps,\eta}  +\lambda_i \Big(f_{\eps,\eta}^2 |\nabla \tilde \psi|^2 + \tilde \psi^2 Lf_{\eps,\eta}\cdot f_{\eps,\eta}\Big) \nonumber\\
&\qquad  
		+ g_{\eps,\eta}^2 |\nabla \tilde q|^2 +
						 \frac{\lambda_i +N-1}{r^2}  s^2 
					+ \frac{\lambda_i(\lambda_i - N +3)}{r^2}  \psi^2 
		- \frac{4\lambda_i}{r^2}  s \psi+\frac{\lambda_i }{r^2}  q^2
		\nonumber\\
			\nonumber	&\qquad
			  +\frac{2}{\eps^2} W''(1 - f_{\eps,\eta}^2 - g_{\eps,\eta}^2)  (f_{\eps,\eta} s+  g_{\eps,\eta} q)^2+\frac{2}{\eta^2} \tilde W''(g_{\eps,\eta}^2)  g^2_{\eps,\eta} q^2\Big\} \,dx\\
	&=  \int_0^1  r^{N-1}\Big\{f_{\eps,\eta}^2 (\tilde s')^2  + \frac{\lambda_i}{r^2}  s^2+\lambda_i  f_{\eps,\eta}^2 (\tilde \psi')^2 
						+ \frac{\lambda_i(\lambda_i - 2N +4)}{r^2} \psi^2 + g_{\eps,\eta}^2 (\tilde q')^2 
			 + \frac{\lambda_i}{r^2} q^2
				\nonumber\\
		&\qquad
		- \frac{4\lambda_i}{r^2} s \psi	 +\frac{2}{\eps^2} W''(1 - f_{\eps,\eta}^2 - g_{\eps,\eta}^2)  (f_{\eps,\eta} s+  g_{\eps,\eta} q)^2+\frac{2}{\eta^2} \tilde W''(g_{\eps,\eta}^2)  g^2_{\eps,\eta} q^2  \Big\} \,dr,\nonumber\\
	&\geq  \int_0^1  r^{N-1}\Big\{f_{\eps,\eta}^2 (\tilde s')^2 +\lambda_i  f_{\eps,\eta}^2 (\tilde \psi')^2 
						+ \frac{\lambda_i}{r^2} (s-2\psi)^2  
			 + \frac{\lambda_i}{r^2} q^2
				 \Big\} \,dr,	
	\label{Eq:new2}
\end{align}
where we used \eqref{Eq:WCond'} and $\lambda_i\geq 2N$ for $i\geq N+1$. Finally, the $L^2$ lower bound (uniform in $\eps, \eta>0$) follows by the Hardy inequality in $\RR^{N+2}$ using $r\leq f_{\eps,\eta}(r)\leq 1$ for every $r\in (0,1)$ (as in \eqref{ineg55}):
\be
\label{ineg45}
\int_0^1  r^{N-1} f_{\eps,\eta}^2 (h')^2\,dr
	\geq \int_0^1  r^{N+1}  (h')^2\,dr \geq \frac{N^2}{4}\int_0^1  r^{N-1} h^2\,dr \geq \frac{N^2}{4} \int_0^1  r^{N-1}f_{\eps,\eta}^2 h^2\,dr,
\ee
where $h$ stands for either $\tilde s$ or $\tilde \psi$.
\end{proof}

We are in position to give:

\begin{proof}[Proof of Theorem \ref{Thm:ExtendedMain}(a) and (b)]
By Theorem \ref{Thm:ExtendedExist}, we only need to prove that, when an escaping critical point $m_{\eps,\eta}(x) = (f_{\eps,\eta}(r)n(x),g_{\eps,\eta}(r))$ with $g_{\eps,\eta} > 0$ exists, the second variation $Q_{\eps,\eta}$ of $E_{\eps,\eta}$ at $m_{\eps,\eta}$ is positive definite, and that $m_{\eps,\eta}$ is a local minimizer of $E_{\eps,\eta}$. 

\medskip
\noindent\underline{Proof of the positive definiteness of $Q_{\eps,\eta}$.}
Fix some $V \in C_c^\infty(B^N \setminus \{0\},\RR^{N+1})$ and define $\mathring{V} = (\mathring{w},0)$, $V_i = (s_i \zeta_i n +  \psi_i \slashed{D} \zeta_i, q_i \zeta_i)$ as in Proposition \ref{Prop:OrthoDecomp}. By the orthogonal decomposition \eqref{Eq:QeeODecomp}, Lemmas \ref{Lem:2VarDivFree}, \ref{Lem:2VarZero} and \ref{Lem:2VarRest}, we have
\begin{align}
Q_{\eps,\eta}[V]
	&\geq C  \Big\|V - \sum_{i=0}^N V_i\Big\|_{L^2(B^N)}^2\nonumber\\
		&\quad 
		+ \int_0^1 r^{N-1} \Big\{f_{\eps,\eta}^2 \Big|\Big(\frac{s_0}{f_{\eps,\eta}}\Big)'\Big|^2
		+ g_{\eps,\eta}^2 \Big|\Big(\frac{q_0}{g_{\eps,\eta}}\Big)'\Big|^2 \nonumber\\
			&\qquad +  \frac{2}{\eps^2} W''(1 - f_{\eps,\eta}^2 - g_{\eps,\eta}^2)(f_{\eps,\eta} s_0 + g_{\eps,\eta} q_0)^2 
			+ \frac{2}{\eta^2} \tilde W''(g_{\eps,\eta}^2) g_{\eps,\eta}^2 q_0^2
			\Big\}\,dr\nonumber\\
		&\quad 
		+ \sum_{i=1}^N \int_0^1 r^{N-1} \Big\{(f_{\eps,\eta}')^2 \Big|\Big(\frac{s_i}{f_{\eps,\eta}'}\Big)'\Big|^2
		+ \frac{N-1}{r^2} f_{\eps,\eta}^2 \Big|\Big(\frac{r \psi_i}{f_{\eps,\eta}}\Big)'\Big|^2\nonumber\\
			&\qquad 
			+ (g_{\eps,\eta}')^2 \Big|\Big(\frac{q_i}{g_{\eps,\eta}'}\Big)'\Big|^2
			+ \frac{2(N-1)}{r^3} f_{\eps,\eta}\,f_{\eps,\eta}' \Big(\frac{s_i}{f_{\eps,\eta}'} - \frac{r \,\psi_i}{f_{\eps,\eta}}\Big)^2 
			\Big\} \,dr .
			\label{Eq:ExtPD1}
\end{align}
By the density of $C_c^\infty(B^N \setminus \{0\},\RR^{N+1})$ in $H_0^1(B^N,\RR^{N+1})$ and Fatou's lemma, the above inequality holds for all $V \in H_0^1(B^N,\RR^{N+1})$, proving that $Q_{\eps,\eta}$ is non-negative semi-definite.

Suppose next that $Q_{\eps,\eta}[V] = 0$ for some non-trivial $V \in H_0^1(B^N,\RR^{N+1})$. The above inequality implies that $V=\sum_{i=0}^N V_i$, $s_0=c_0 f_{\eps,\eta}$, $q_0=\tilde c_0 g_{\eps,\eta}$, $s_i=c_i f'_{\eps,\eta}$, $\psi_i=\frac{\hat c_i}{r} f_{\eps,\eta}$, $q_i=\tilde c_i g'_{\eps,\eta}$ in $(0,1)$ for $1\leq i\leq N$ and some constants $c_i, \tilde c_i, \hat c_i$. As $V_i$ are compactly supported in $B^N$ and $f_{\eps,\eta}(1), f'_{\eps,\eta}(1), g'_{\eps,\eta}(1)\neq 0$, we deduce that 
$V = V_0 = (0, q_0 \zeta_0)$. 

Suppose by contradiction that $\tilde c_0\neq 0$. Then $q_0$ has no zeros inside $(0,1)$, therefore $W''(1 - f_{\eps,\eta}^2 - g_{\eps,\eta}^2) \equiv \tilde W''(g_{\eps,\eta}^2) \equiv 0$ in $(0,1)$. It follows that $W'$ is constant in $[\min (1 - f_{\eps,\eta}^2 - g_{\eps,\eta}^2), \max (1 - f_{\eps,\eta}^2 - g_{\eps,\eta}^2)] =: [0,\tau]$ and hence $W' = 0$ in $[0,\tau]$ since $W'(0) = 0$ (by \eqref{Eq:WCond'}). Recalling \eqref{Eq:20III21-gee}, we thus have that $-\Delta g_{\eps,\eta}+\frac1{\eta^2} \tilde W'(g^2_{\eps,\eta}) g_{\eps,\eta}=0$ in $B^N$. Since $\tilde W'\geq \tilde W'(0)\geq 0$ in $[0, \infty)$ (by \eqref{Eq:TWCond'}) and $g_{\eps,\eta}=0$ on $\partial B^N$, we deduce that $g_{\eps,\eta}=0$ in $B^N$ which gives a contradiction to the assumption $g_{\eps,\eta}>0$ in $B^N$. Thus,  $\tilde c_0=0$, leading to $q_0=0$ and $V=0$. This proves that $Q_{\eps,\eta}$ is positive definite. 

By \eqref{Eq:09IV19-2VarX1}, the convexity of $W$ and $\tilde W$, the fact that $\tilde W'\geq 0$ and the boundedness of $(f_{\eps,\eta},g_{\eps,\eta})$, we have for some constant $C_1 = C_1(\eps) > 0$ that 
\[
Q_{\eps,\eta}[V] \geq \|\nabla V\|_{L^2(B^N)}^2 - C_1 \|V\|_{L^2(B^N)}^2 \text{ for all } V \in H_0^1(B^N,\RR^{N+1}).
\]
This together with the weak lower semi-continuity of $Q_{\eps,\eta}$ in $H_0^1(B^N,\RR^{N+1})$ implies that $\min \{Q_{\eps,\eta}[V]: V \in H_0^1(B^N ,\RR^{N+1}), \|V\|_{L^2(B^N)} = 1\}$ is achieved and positive (as $Q_{\eps,\eta}$ is positive definite), yielding for some constant $C_2 = C_2(\eps,\eta) > 0$ 
\[
Q_{\eps,\eta}[V] \geq \frac{1}{C_2} \|V\|_{L^2(B^N)}^2 \text{ for all } V \in H_0^1(B^N,\RR^{N+1}). 
\]
The above two inequalities imply for $C_3 = C_3(\eps,\eta) = 1 + C_2(C_1 + 1)$ that
\[
Q_{\eps,\eta}[V] \geq \frac{1}{C_3} \|V\|_{H^1(B^N)}^2  \text{ for all } V \in H_0^1(B^N,\RR^{N+1}).
\]

\medskip
\noindent\underline{Proof of the local minimality of $m_{\eps,\eta}$.} We note a subtlety in this step due to the fact that $E_{\eps, \eta}$ may not be finite in a $H^1_0$ neighborhood of $m_{\eps, \eta}$ as we make no growth assumption for $W$ and $\tilde W$. Since $m_{\eps,\eta}$ is a critical point for $E_{\eps,\eta}$ in $\mcA$, we have, for $V = (v,q) \in H_0^1(B^N,\RR^{N+1})$,
\begin{align*}
&E_{\eps,\eta}[m_{\eps,\eta} + V] - E_{\eps,\eta}[m_{\eps,\eta}] -\frac{1}{2} Q_{\eps,\eta}[V]=\int_{B^N} h(x, V(x))\, dx,\\
 h(x, y)	&=    \frac{1}{2\eps^2} \Big\{W(1 - |m_{\eps,\eta}(x) + y|^2) - W(1 - |m_{\eps,\eta}(x)|^2)\\
			&\,\, \, + W'(1 - |m_{\eps,\eta}(x)|^2)(2m_{\eps,\eta}(x) \cdot y + |y|^2) - 2 W''(1 - |m_{\eps,\eta}(x)|^2)(m_{\eps,\eta}(x) \cdot y)^2\Big\} \\
		&\,\, \,   + \frac{1}{2\eta^2} \Big\{ \tilde W((g_{\eps,\eta}(x) + y_{N+1})^2) - \tilde W(g_{\eps,\eta}^2(x)) 
			 - \tilde W'(g_{\eps,\eta}^2(x)) (2g_{\eps,\eta}(x) y_{N+1} + y_{N+1}^2)\\
		&\,\, \, - 2 \tilde W''(g_{\eps,\eta}^2(x)) g_{\eps,\eta}^2(x) y_{N+1}^2\Big\}.
\end{align*}
Note that $h\in C^0(\bar B^N, C^2(\RR^{N+1}))$, $h(x,0)=0$, $\nabla_y h(x, 0)=0$, $\nabla^2_y h(x,0)=0$ (thus, \eqref{rrr} holds true in Lemma~\ref{Lem:RA}) and, due to the convexity of $W$ and $\tilde W$, $h$ satisfies the growth assumptions in Lemma \ref{Lem:RA} for $p=2$, namely
$$h(x,y)\geq -\frac1{\eps^2} W''(1 - |m_{\eps,\eta}(x)|^2)(m_{\eps,\eta}(x) \cdot y)^2-\frac1{\eta^2} \tilde W''(g_{\eps,\eta}^2(x)) g_{\eps,\eta}^2(x) y_{N+1}^2\geq -C(\eps, \eta)|y|^2$$
for every $x\in B^N$ and $y\in \RR^{N+1}$ and a constant $C(\eps, \eta)>0$. Therefore, Lemma \ref{Lem:RA} together with the positive definiteness of $Q_{\eps,\eta}$ yield for some constants $\delta>0$ and $\tilde C>0$ (depending on $\eps$ and $\eta$),
$$E_{\eps,\eta}[m_{\eps,\eta} + V] \geq E_{\eps,\eta}[m_{\eps,\eta}] + \tilde C\|V\|^2_{H^1(B^N)}
\,\,  \textrm{for all } V \in H_0^1(B^N,\RR^{N+1}) \textrm{ with }\|V\|_{H^1(B^N)}<\delta.$$
This proves the local minimality of $m_{\eps,\eta}$.
\end{proof}

\begin{remark}\label{Rem:WWLocalCond}
The above result can be used to obtain the local minimality of any escaping radially symmetric critical point $m_{\eps,\eta} = (f_{\eps,\eta} n, g_{\eps,\eta})$ of $E_{\eps,\eta}$ with $g_{\eps,\eta} > 0$ and $f_{\eps,\eta}^2+g_{\eps,\eta}^2\leq 1$ under a slightly weaker assumption that $W \in C^2([0,1])$, $\tilde W \in C^2([0,1])$ and
\begin{align}
&W(0) =  0, W(t)\geq 0 \text{ in }(-\infty,1], W''(t) \geq 0 \text{ in } [0,1], 
	\label{Eq:WCondX}\\
& \tilde W(0) = 0, \tilde W(t)\geq 0 \text{ in } [0,1], \tilde W(t) \geq \tilde W(1) \text{ in } [1,\infty), \tilde W''(t) \geq 0 \text{ in } [0,1].
	\label{Eq:TWCondX}
\end{align}
In the Ginzburg--Landau context, similar conditions appeared in \cite{LiebLoss95-JEDP}.
\end{remark}

\begin{proof}
For $m \in \mcA$, define the truncation $Tm \in \mcA$ of $m$ by
\[
Tm(x) = \left\{\begin{array}{ll}
	m(x) & \text{ if } |m(x)| \leq 1,\\
	\frac{m(x)}{|m(x)|} & \text{ if } |m(x)| > 1.
\end{array}\right.
\]
Observe that, by \eqref{Eq:WCondX}--\eqref{Eq:TWCondX}, $E_{\eps,\eta}[m] \geq E_{\eps,\eta}[Tm]$ for $m \in \mcA$. On the other hand, by applying Theorem \ref{Thm:ExtendedMain} to a pair of potentials satisfying \eqref{Eq:WCond'}--\eqref{Eq:TWCond'} which agree with $(W, \tilde W)$ in $[0,1]$ (e.g. by using suitable quadratic polynomials outside of $[0,1]$), we obtain that there exist $\delta > 0$ and $C > 0$ such that $E_{\eps,\eta}[Tm] \geq E_{\eps,\eta}[m_{\eps,\eta}] + \frac{1}{C} \|Tm - m_{\eps,\eta}\|_{H^1(B^N,\RR^{N+1})}$ whenever $m \in \mcA$ and $\|Tm - m_{\eps,\eta}\|_{H^1(B^N,\RR^{N+1})} \leq \delta$. Therefore, to prove the local minimality of $m_{\eps,\eta}$, it suffices to show that the truncation map is continuous at $m_{\eps,\eta}$, i.e. if $m_j \rightarrow m_{\eps,\eta}$ in $H^1(B^N,\RR^{N+1})$, then $Tm_j \rightarrow m_{\eps,\eta}$ in $H^1(B^N,\RR^{N+1})$.

Indeed, observe that, for $a, b \in \RR^N$ with $ |a| \geq 1, |b| \leq 1$,
\[
|a - b|^2
	 = \Big(|a| -  \frac{b \cdot a}{|a|}\Big)^2 + \Big|b - \frac{b \cdot a}{|a|^2}a\Big|^2 
	\geq \Big(1 -  \frac{b \cdot a}{|a|}\Big)^2 + \Big|b - \frac{b \cdot a}{|a|^2}a\Big|^2 
		= \Big|\frac{a}{|a|} - b\Big|^2. 
\]
This implies that
\[
\|m_j - m_{\eps,\eta}\|_{L^2(B^N,\RR^{N+1})}^2 \geq \|Tm_j - m_{\eps,\eta}\|_{L^2(B^N,\RR^{N+1})}^2,
\]
and so $Tm_j \rightarrow m_{\eps,\eta}$ in $L^2(B^N,\RR^{N+1})$. Since $\|Tm_j\|_{H^1(B^N,\RR^{N+1})} \leq \|m_j\|_{H^1(B^N,\RR^{N+1})}$, $\{Tm_j\}$ has a $H^1$-weakly convergent subsequence $\{T_{m_{j_k}}\}$, whose weak limit must be $m_{\eps,\eta}$ (in view of the strong $L^2$ convergence of $Tm_j$), and 
\[
\|\nabla m_{\eps,\eta}\|_{L^2(B^N,\RR^{N+1})} \leq \liminf_{k\rightarrow \infty} \|\nabla Tm_{j_k}\|_{L^2(B^N,\RR^{N+1})}.
\]
On the other hand, by construction, 
\[
\|\nabla Tm_j\|_{L^2(B^N,\RR^{N+1})} \leq \|\nabla m_j\|_{L^2(B^N,\RR^{N+1})} \rightarrow \|\nabla m_{\eps,\eta}\|_{L^2(B^N,\RR^{N+1})}.
\]
We thus have that $\|\nabla Tm_{j_k}\|_{L^2(B^N,\RR^{N+1})} \rightarrow \|\nabla m_{\eps,\eta}\|_{L^2(B^N,\RR^{N+1})} $ and so $Tm_{j_k} \rightarrow m_{\eps,\eta}$ in $H^1(B^N,\RR^{N+1})$.  Since the above argument can be applied to any subsequence of $\{Tm_j\}$, we deduce that $Tm_j \rightarrow m_{\eps,\eta}$ in $H^1(B^N,\RR^{N+1})$ as desired.
\end{proof}

\subsection{The $\RR^N$-valued GL model: Stability of the vortex solution}

Assume that $N\geq 2$ and $W \in C^2((-\infty,1])$ satisfies \eqref{Eq:WCond'}. Let $u_{\eps} = f_{\eps} n$ be the radially symmetric critical point of the Ginzburg--Landau energy $E_{\eps}^{GL}$ in $\mcA^{GL}$ obtained in Theorem \ref{Thm:GLExist}, and let $\Qgl$ be the second variation of $E_{\eps}^{GL}$ at $u_\eps = \barfeps n$,
\begin{align}
\Qgl[v]&:=\int_{B^N} \Big[|\nabla v|^2 
		- \frac{1}{\eps^2}W'(1 - \barfeps^2)|v|^2 +  \frac{2}{\eps^2} W''(1 - \barfeps^2)\barfeps^2( n \cdot v)^2 \Big]\,dx,
		\label{second_gl}
\end{align}
where $v \in H_0^1(B^N,\RR^N)$.

\begin{proof}[Proof of Theorem \ref{Thm:GLStab}.]
We will only prove the positive definiteness of $Q_\eps^{GL}$ in $C^\infty_c(B^N\setminus\{0\}, \RR^N)$. As in the proof of Theorem \ref{Thm:ExtendedMain}(a), the estimate we obtain (see \eqref{Eq:QeGLS} below) implies that $Q_\eps^{GL}[v] \geq C\|v\|_{H^1(B^N)}^2$ for $v \in H_0^1(B^N,\RR^N)$ and that $u_\eps$ is a local minimizer of $E^{GL}_\eps$ in $\mcA^{GL}$, more precisely, 
for some constants $\delta>0$ and $\tilde C>0$ (depending on $\eps$),
$$E^{GL}_{\eps}[u_{\eps} + v] \geq E^{GL}_{\eps}[u_{\eps}] + \tilde C\|v\|^2_{H^1(B^N)}
\,\,  \textrm{for all } v \in H_0^1(B^N,\RR^{N}) \textrm{ with }\|v\|_{H^1(B^N)}<\delta.$$

Take an arbitrary $v\in C^\infty_c(B^N\setminus\{0\}, \RR^N)$. We use the decomposition in Proposition \ref{Prop:OrthoDecomp} in the orthonormal basis $(\zeta_i)_{i\geq 0}$ of $L^2(\bS^{N-1})$. We write $v = sn + \mathring{w}+\slashed{D} \psi$ with $s\in C_c^\infty(B^N \setminus \{0\})$, $\mathring{w}\in C^\infty_c(B^N\setminus\{0\}, \RR^N)$ being a tangent vector field (i.e.,  $\mathring{w} \cdot n = 0$) having vanishing covariant divergence $\slashed{D} \cdot \mathring{w}(r,\cdot) = 0$ on $\bS^{N-1}$ and $\psi \in C^\infty_c(B^N\setminus\{0\}, \RR)$ satisfying $\int_{\mathbb{S}^{N-1}} \psi(r,\theta) d\sigma = 0$, and decompose
$$s(r,\theta) = \sum_{i=0}^\infty s_i(r) \zeta_i(\theta), \qquad \psi(r,\theta) = \sum_{i=0}^\infty \psi_i(r) \zeta_i(\theta),$$
with $s_i, \psi_i\in C^\infty_c\big((0,1)\big)$ for every $i\geq 0$ and for every $r \in (0,1]$. We will prove
\begin{align}
Q_{\eps}^{GL}[v]
	&\geq C  \Big\|v - \sum_{i=1}^N v_i\Big\|_{L^2(B^N)}^2\nonumber\\
		&\quad 
		+ \sum_{i=1}^N \int_0^1 r^{N-1} \Big\{(f_{\eps}')^2 \Big|\Big(\frac{s_i}{f_{\eps}'}\Big)'\Big|^2\nonumber\\
		&\qquad + \frac{2(N-1)}{r^3} \barfeps\,\barfeps' \Big(\frac{s_i}{f_{\eps}'} - \frac{r \psi_i}{f_{\eps}}\Big)^2 
		+ \frac{N-1}{r^2} f_{\eps}^2 \Big|\Big(\frac{r \psi_i}{f_{\eps}}\Big)'\Big|^2
			\Big\} \,dr ,
			\label{Eq:QeGLS}
\end{align}
where $v_i = s_i \zeta_i n +  \psi_i \slashed{D} \zeta_i\in C_c^\infty(B^N\setminus\{0\},\RR^{N})$ for 
$i\geq 0$.

Proposition \ref{Prop:OrthoDecomp} yields
$$
\Qgl[v]=\Qgl[\mathring{w}] + \sum_{i=0}^\infty \Qgl[v_i].
	$$
First, Lemma \ref{Lem:2VarDivFree} yields a constant $C>0$ independent of $\eps$ such that 
$$\Qgl[\mathring{w}] \geq C \|\mathring{w}\|^2_{L^2}$$ for every 
 tangent vector field $\mathring{w}\in C^\infty_c(B^N\setminus\{0\}, \RR^N)$ of vanishing covariant divergence. Second, for the zero mode $v_0 = s_0 \zeta_0 n$, the argument in the proof of Lemma \ref{Lem:2VarZero} yields
 \begin{align*}
 \Qgl[s_0 \zeta_0 n] 
	&= \int_0^1 r^{N-1} \Big\{\barfeps^2 \Big|\Big(\frac{s_0}{\barfeps}\Big)'\Big|^2
	+  \frac{2}{\eps^2} W''(1 - \barfeps^2) f^2_\eps s_0^2\Big\}\,dr\\
	&\geq \int_0^1 r^{N+1} \Big|\Big(\frac{s_0}{\barfeps}\Big)'\Big|^2\, dr\geq \frac{N^2}4 \int_0^1 r^{N-1} s_0^2\, dr=\frac{N^2}{4} \|v_0\|^2_{L^2},
	\end{align*}
where we used $r\leq \barfeps\leq 1$ in $(0,1)$ and the Hardy inequality in $\RR^{N+2}$ (as in \eqref{ineg55}). Third, for the modes  $v_i = s_i \zeta_i n + \psi_i \slashed{D}\zeta_i$ for $1\leq i\leq N$ (so that $\lambda_i = N-1$), we factor 
$s_i = \barfeps' \hat s_i$ and $\psi_i = \frac{\barfeps}{r} \hat \psi_i$, and the computation in the proof of  Lemma \ref{Lem:2VarRest} yields
\begin{align*}
\Qgl[v_i] &=   \int_0^1  r^{N-1}\Big\{(\barfeps')^2 (\hat s_i')^2  
			+ \frac{2(N-1)}{r^3} \barfeps\,\barfeps' (\hat s_i-\hat \psi_i)^2 
			 + \frac{N-1 }{r^2} \barfeps^2 (\hat \psi_i')^2 
				\Big\} \,dr\geq 0.
	\end{align*}
Finally, for the modes  $v_i = s_i \zeta_i n + \psi_i \slashed{D}\zeta_i$ for $i\geq N+1$, we factor $s_i = \barfeps \tilde s$ and $\psi_i = \barfeps \tilde \psi_i$, and the computation in the proof of  Lemma \ref{Lem:2VarRest} (see \eqref{Eq:new2}) yields 
\[
\Qgl[v_i] \geq C\|v_i\|^2_{L^2(B^N)} \text{ for every } i\geq N+1
\]
for some $C > 0$ independent of $\eps$ and $i$. Combining the above estimates, we get \eqref{Eq:QeGLS}.
\end{proof}

\subsection{The extended model: Stability-instability dichotomy of the non-escaping vortex solutions}

Let $N \geq 2$. Assume $W \in C^2((-\infty,1])$ and $\tilde W \in C^2([0,\infty))$ satisfy \eqref{Eq:WCond'} and \eqref{Eq:TWCond'}. Let $\bar m_{\eps} = (\barfeps n, 0)$ be the in-plane radially symmetric critical point of $E_{\eps,\eta}$ in $\mcA$, where $f_\eps$ is given by Theorem \ref{Thm:GLExist}. Let $\bar Q_{\eps,\eta}$ be the second variation of $E_{\eps,\eta}$ at $\bar m_{\eps}$: For $V = (v, q) \in H_0^1(B^N, \RR^N) \times H_0^1(B^N,\RR) \cong H_0^1(B^N,\RR^{N+1})$, 
\begin{align*}
\bar Q_{\eps,\eta}[V] 
	&	= \Qgl[v]+ \bar Q_{\eps,\eta}[(0,q)],
	\nonumber\\ 
\bar Q_{\eps,\eta}[(0,q)] 
	&= \int_{B^N} \Big[ |\nabla q|^2
		- \frac{1}{\eps^2}W'(1 - \barfeps^2) q^2
		+ \frac{1}{\eta^2} \tilde W'(0) q^2
		\Big]\,dx\nonumber\\
	&=\int_{B^N} \Big[ L_\eps^{GL} q \cdot q
		+ \frac{1}{\eta^2} \tilde W'(0) q^2
		\Big]\,dx,
\end{align*}
where $\Qgl$ is the second variation at the critical point $u_\eps = \barfeps n$ of the Ginzburg--Landau energy $E_\eps^{GL}$ given in \eqref{second_gl} and $L_\eps^{GL}$ is defined by \eqref{Eq:LeGLDef}. 

\begin{proof}[Proof of Theorem \ref{Thm:ExtendedMain}(c)] 
We will only discuss the positive definiteness of $\bar Q_{\eps,\eta}$. As in the proof of Theorem \ref{Thm:ExtendedMain}(a), in the case when $\bar Q_{\eps,\eta}$ is positive definite, we have that $\bar Q_{\eps, \eta}[V] \geq C\|V\|_{H^1(B^N)}^2$ for $V \in H_0^1(B^N,\RR^{N+1})$ and that $\bar m_\eps$ is a local minimizer of $E_{\eps, \eta}$ in $\mcA$, more precisely, 
for some constants $\delta>0$ and $\tilde C>0$ (depending on $\eps$ and $\eta$),
$$E_{\eps, \eta}[\bar m_{\eps} + V] \geq E_{\eps, \eta}[\bar m_{\eps}] + \tilde C\|V\|^2_{H^1(B^N)}
\,\,  \textrm{for all } V \in H_0^1(B^N,\RR^{N+1}) \textrm{ with }\|V\|_{H^1(B^N)}<\delta.$$

By Theorem \ref{Thm:GLStab}, $Q_\eps^{GL}$ is positive definite. Therefore, $\bar Q_{\eps,\eta}$ is positive definite if and only if $\bar Q_{\eps,\eta}[(0,\cdot)]$ is positive definite, i.e.
\[
\ell(\eps) + \frac{1}{\eta^2} \tilde W'(0) > 0,
\]
where $\ell(\eps)$ is the first eigenvalue of $L_\eps^{GL}$ on $B^N$ with zero Dirichlet boundary value. Recalling that we are assuming that \eqref{Eq:feegeeH1}--\eqref{Eq:20III21-feegeeBC} has no escaping solutions, we deduce from Theorem \ref{Thm:ExtendedExist}(a) and (b), Lemma \ref{Lem:eigen-gl} and the fact that $\tilde W'(0) \geq 0$ that the above inequality fails if and only if $2 \leq N \leq 6$, $W'(1) > 0$, $\tilde W'(0) > 0$, $0 < \eps < \eps_0$ and $\eta  = \eta_0(\eps)$. In this case, 
$\ell(\eps) + \frac{1}{\eta^2} \tilde W'(0) =0$, 
$\bar Q_{\eps,\eta}$ is non-negative semi-definite with the kernel 
\[
\Big\{(0,q): q \in H_0^1(B^N), L_{\eps}^{GL} q = \ell(\eps) q\Big\},
\]
which is one-dimensional and generated by $(0,q_{\eps})$ for any first eigenfunction $q_{\eps}$ of $L_\eps^{GL}$. 
\end{proof}

\subsection{The $\Sphere^N$-valued GL model: Stability of the escaping vortex solution}

Assume that $N\geq 2$ and $\tilde W \in C^2([0,\infty))$. Let $m_{\eta} = (\tilde f_{\eta} n, g_{\eta})$ be the escaping radially symmetric critical point of $E_{\eta}^{MM}$ in $\mcA^{MM}$  with $\tilde f_\eta > 0$ and $g_\eta > 0$, and let $Q_{\eta}^{MM}$ be the second variation of $E_{\eta}^{MM}$ at $m_{\eta}$: For $V = (v, q) \in H_0^1(B^N, \RR^N) \times H_0^1(B^N,\RR) \cong H_0^1(B^N,\RR^{N+1})$ with $V \cdot m_{\eta} = 0$, 
\begin{align*}
Q_{\eta}^{MM}[V] 
	&= \frac{d^2}{dt^2}\Big|_{t = 0} E_{\eta}^{MM}\Big[\frac{m_{\eta} + tV}{|m_{\eta} + tV|}\Big]\nonumber\\
	&= \int_{B^N} \Big[|\nabla V|^2  - \lambda(r) |V|^2 
		+ \frac{1}{\eta^2}\tilde W'(g_{\eta}^2) q^2
		+  \frac{2}{\eta^2} \tilde W''(g_{\eta}^2) g_{\eps,\eta}^2 q^2
		\Big]\,dx
\end{align*}
where $\lambda \in C^1([0,1])$ is given by \eqref{Eq:lamDef}. In particular, $Q_\eta^{MM}$ is continuous in $H_0^1(B^N,\RR^{N+1})$.

\begin{proof}[Proof of Theorem \ref{Thm:MStab}] By the instability of the equator map proved in Theorem \ref{Thm:MMExist}(b), we only need to prove the stability and local minimality of the escaping solution $m_\eta$.

\medskip
\noindent\underline{Proof of the positive definiteness of $Q_{\eta}^{MM}$.} Let $W(t) = t^2$ and let $\eps_0$ and $\eta_0 \in C^0([0,\eps_0))$ be as in Theorem \ref{Thm:ExtendedExist}; those are well-defined as $W'(1) > 0$. If $\tilde W'(0)>0$, then $\eta_0$ is increasing and $\lim_{\eps \rightarrow \eps_0} \eta_0(\eps) = \infty$ (see Remark \ref{Rem:eta0}), so $\eta_0$ has an increasing inverse $\eta_0^{-1}:[0,\infty) \rightarrow [0,\eps_0)$. If $\tilde W'(0)=0$, then $\eta_0(\eps)=0$ for all $\eps\in (0, \eps_0)$ and by abuse of notation, we set $\eta_0^{-1}(\eta)=\eps_0$ for every $\eta>0$. In both cases, by Theorem \ref{Thm:ExtendedExist}, for $0 < \eps < \eta_0^{-1}(\eta)$, \eqref{Eq:feegeeH1}--\eqref{Eq:20III21-feegeeBC} has an escaping solution $(f_{\eps,\eta},g_{\eps,\eta})$ with $f_{\eps,\eta} > 0$ and $g_{\eps,\eta} > 0$. By Remark \ref{Rem:EqInst}, $(f_{\eps,\eta},g_{\eps,\eta}) \rightarrow (\tilde f_\eta, g_\eta)$ in $\mcB$ as $\eps\to 0$, and so uniformly on compact subsets of $(0,1]$.

We would like to deduce the positive definiteness of $Q_\eta^{MM}$ from the positive definiteness of the second variation $Q_{\eps,\eta}$ of the escaping critical point $m_{\eps,\eta} = (f_{\eps,\eta}n,g_{\eps,\eta})$ of $E_{\eps,\eta}$ (established in Theorem \ref{Thm:ExtendedMain}(a)).

Fix some $V = (v,q) \in C_c^\infty(B^N \setminus \{0\},\RR^{N+1})$ with $V \cdot m_\eta = 0$ in $B^N$. We write $v = sn + \mathring{w}+\slashed{D} \psi$ with $s\in C_c^\infty(B^N \setminus \{0\})$, $\mathring{w}\in C^\infty_c(B^N\setminus\{0\}, \RR^N)$ being a tangent vector field (i.e.,  $\mathring{w} \cdot n = 0$) having vanishing covariant divergence $\slashed{D} \cdot \mathring{w}(r,\cdot) = 0$ on $\bS^{N-1}$ and $\psi \in C^\infty_c(B^N\setminus\{0\}, \RR)$ satisfying $\int_{\mathbb{S}^{N-1}} \psi(r,\theta) d\sigma = 0$. 

For $0 < \eps < \eta_0^{-1}(\eta)$, define $V_\eps = (v,q_\eps) \in C_c^\infty(B^N \setminus \{0\},\RR^{N+1})$ by
\[
q_\eps = q - \frac{f_{\eps,\eta} - \tilde f_\eta}{g_{\eps,\eta}} s - \frac{g_{\eps,\eta} - g_\eta}{g_{\eps,\eta}} q.
\]
Then $\mathrm{supp} V_\eps \subset\mathrm{supp} V\subset B^N\setminus \{0\}$, and $V_\eps \rightarrow V$ uniformly in $\bar B^N$ and in $H^1(B^N)$ as $\eps \rightarrow 0$ and $V_\eps \cdot m_{\eps,\eta} = 0$ in $B^N$. We decompose
\begin{align*}
s(r,\theta) &= \sum_{i=0}^\infty s_i(r) \zeta_i(\theta), \qquad \psi(r,\theta) = \sum_{i=0}^\infty \psi_i(r) \zeta_i(\theta),\\
q(r,\theta) &= \sum_{i=0}^\infty q_i(r) \zeta_i(\theta), \qquad q_\eps(r,\theta) = \sum_{i=0}^\infty q_{\eps,i}(r) \zeta_i(\theta),
\end{align*}
define $\mathring{V} = (\mathring{w},0)$, $V_i = (s_i \zeta_i n +  \psi_i \slashed{D} \zeta_i, q_i \zeta_i)$ and $V_{\eps,i} = (s_i \zeta_i n +  \psi_i \slashed{D} \zeta_i, q_{\eps,i} \zeta_i)$ as in Proposition \ref{Prop:OrthoDecomp}. Note that $V_{\eps, i}\to V_i$ uniformly in $\bar B^N$ and in $H^1(B^N)$ as $\eps\to 0$ for every $i\geq 0$, 
\[
0 = V \cdot m_{\eta} = s \tilde f_{\eta} + q g_{\eta} = \sum_{i=0}^\infty (s_i \tilde f_{\eta} + q_{i}g_{\eta})\zeta_i
\]
and so $s_i \tilde f_{\eta} + q_{i}g_{\eta} = 0$ for all $i \geq 0$. By the positivity inequality \eqref{Eq:ExtPD1} for $Q_{\eps,\eta}$, we have 
\begin{align}
Q_{\eps,\eta}[V_\eps]
	&\geq C  \Big\|V_\eps - \sum_{i=0}^N V_{\eps,i}\Big\|_{L^2(B^N)}^2
		+ \int_0^1 r^{N-1} f_{\eps,\eta}^2 \Big|\Big(\frac{s_0}{f_{\eps,\eta}}\Big)'\Big|^2\,dr
		\nonumber\\
		&\quad 
		+ (N-1)\sum_{i=1}^N \int_0^1 r^{N-3} \Big\{  f_{\eps,\eta}^2 \Big|\Big(\frac{r \psi_i}{f_{\eps,\eta}}\Big)'\Big|^2
			+ \frac{2}{r} f_{\eps,\eta}\,f_{\eps,\eta}' \Big(\frac{s_i}{f_{\eps,\eta}'} - \frac{r \,\psi_i}{f_{\eps,\eta}}\Big)^2
			\Big\} \,dr .
		\label{Eq:MMS-1}
\end{align}

\medskip
\noindent\underline{Claim:} $Q_{\eps,\eta}[V_\eps] \rightarrow Q_\eta^{MM}[V]$ as $\eps \rightarrow 0$. Indeed, since $f_{\eps,\eta}$ converges to $\tilde f_\eta$ in $H^1_{loc}(0,1)$, we deduce that, for any open set $K$ compactly supported in $B^N \setminus \{0\}$ and $(\varphi_\eps) \subset H_0^1(K)$ converging in $H^1$ to $\varphi \in H_0^1(K)$, by multiplying from \eqref{Eq:20III21-fee} and \eqref{Eq:MM-fee} with $\varphi_\eps/f_{\eps, \eta}$ and $\varphi / \tilde f_{\eta}$ respectively,
\[
\lim_{\eps \rightarrow 0} \int_{B^N} \frac{1}{\eps^2} W'(1 - f_{\eps,\eta}^2 - g_{\eps,\eta}^2) \varphi_\eps\,dx = \int_{B^N} \lambda(r) \varphi\,dx .
\]
Recalling the expressions of $Q_{\eps,\eta}[V_\eps]$ and  $Q_\eta^{MM}[V]$ together with the fact that $s f_{\eps,\eta} + q_\eps g_{\eps,\eta} = V_\eps \cdot m_{\eps,\eta} = 0$, $\mathrm{supp} V_\eps \subset \mathrm{supp} V \subset B^N \setminus \{0\}$, and $|V_\eps|^2 \rightarrow |V|^2$ in $H_0^1(\textrm{supp} V)$, the claim is readily seen from the above identity.

Passing $\eps \rightarrow 0$ in \eqref{Eq:MMS-1} using the claim on the left hand side and Fatou's lemma on the right hand side, we obtain
\begin{align}
Q_{\eta}^{MM}[V]
	&\geq C  \Big\|V - \sum_{i=0}^N V_{i}\Big\|_{L^2(B^N)}^2
		+ \int_0^1 r^{N-1} \tilde f_{\eta}^2 \Big|\Big(\frac{s_0}{\tilde f_{\eta}}\Big)'\Big|^2\,dr
		\nonumber\\
		&\quad 
		+ (N-1)\sum_{i=1}^N \int_0^1 r^{N-3} \Big\{  \tilde f_{\eta}^2 \Big|\Big(\frac{r \psi_i}{\tilde f_{\eta}}\Big)'\Big|^2
			+ \frac{2}{r} \tilde f_{\eta}\,\tilde f_{\eta}' \Big(\frac{s_i}{\tilde f_{\eta}'} - \frac{r \,\psi_i}{\tilde f_{\eta}}\Big)^2
			\Big\} \,dr 
		\label{Eq:MMS-2}
\end{align}
for any $V \in C_0^\infty(B^N\setminus \{0\}, \RR^{N+1})$ satisfying $V \cdot m_\eta = 0$ in $B^N$.

Suppose next that $V  \in H_0^1(B^N,\RR^{N+1})$ with $V \cdot m_\eta = 0$ in $B^N$. Pick a sequence $\{V_j\} \subset C_c^\infty(B^N \setminus \{0\},\RR^{N+1})$ which converges in $H^1(B^N,\RR^{N+1})$ to $V$. Let $\tilde V_j = V_j - (V_j \cdot m_\eta) m_\eta \in C_c^\infty(B^N \setminus \{0\},\RR^{N+1})$. Then $\{\tilde V_j\}$ also converges in $H^1(B^N,\RR^{N+1})$ to $V$. Applying \eqref{Eq:MMS-2} to $\tilde V_j$ (since $\tilde V_j \cdot m_\eta = 0$), and sending $j \rightarrow \infty$ (using the continuity of $Q_\eta^{MM}$ on the left hand side and Fatou's lemma on the right hand side), we see that \eqref{Eq:MMS-2} holds for $V \in H_0^1(B^N, \RR^{N+1})$ satisfying $V \cdot m_\eta = 0$ in $B^N$. Moreover, if $Q_{\eta}^{MM}[V] = 0$, then $V = \sum_{i = 0}^N V_i$, and $\frac{s_0}{\tilde f_\eta}$, $\frac{r\psi_i}{\tilde f_\eta}$ are constant and $\frac{s_i}{\tilde f_{\eta}'} - \frac{r \,\psi_i}{\tilde f_{\eta}} = 0$ for $1 \leq i \leq N$. Recalling also that $s_i \tilde f_{\eta} + q_{i}g_{\eta} = 0$ in $(0,1)$ and $s_i(1) = \psi_i(1) = 0$ for all $i \geq 0$, we deduce that $V \equiv 0$. This proves the required positive definiteness of $Q_\eta^{MM}$.

\medskip
\noindent\underline{Proof of the local minimality of $m_\eta$.}

We need to relate the functional $E_\eta^{MM}$ in a neighborhood of $m_\eta$ to the second variation $Q_\eta^{MM}$ notwithstanding the fact that $H^1(B^N,\Sphere^N)$ is not a manifold.\marginnote{add references}

Consider a map $m_\eta + V \in \mcA^{MM}$, and write $V = (v,q)$ and $\tilde V := V - (V \cdot m_\eta) m_\eta = (\tilde v, \tilde q)$ so that $V, \tilde V \in H_0^1(B^N,\RR^{N+1})$ and $\tilde V \cdot m_\eta = 0$. By the Euler-Lagrange equation for $m_{\eta}$ (as a critical point for $E_{\eta}^{MM}$ in $\mcA^{MM}$) and $V \cdot m_\eta = - \frac{1}{2}|V|^2$ (since $|m_\eta + V|^2 = |m_\eta|^2 = 1$),
\begin{align}
&E_{\eta}^{MM}[m_{\eta} + V] - E_{\eta}^{MM}[m_{\eta}] -\frac{1}{2} Q_{\eta}^{MM}[\tilde V]\nonumber\\
	&\quad= \frac{1}{2}\int_{B^N} \Big\{(|\nabla V|^2 - |\nabla \tilde V|^2) - \lambda(r) (|V|^2 - |\tilde V|^2)\nonumber\\
		&\qquad
		+ \frac{1}{\eta^2} \Big[\tilde W'(g_\eta^2) + 2\tilde W''(g_\eta^2) g_\eta^2 \Big](q^2 - \tilde q^2)\Big\}\,dx
		+ 	\int_{B^N} h(x, V(x))\, dx,
			\label{Eq:MLM-1}
\end{align}
\begin{align*}
 h(x, y)	&=   \frac{1}{2\eta^2} \Big\{ \tilde W((g_{\eta}(x) + y_{N+1})^2) - \tilde W(g_{\eta}^2(x)) 
			 - \tilde W'(g_{\eta}^2(x)) (2g_{\eta}(x) y_{N+1} + y_{N+1}^2)\\
			 &\quad - 2 \tilde W''(g_{\eta}^2(x)) g_{\eta}^2(x) y_{N+1}^2\Big\}.
\end{align*}
As in the proof of Theorem \ref{Thm:ExtendedMain}(a), the positive definiteness of $Q_\eta^{MM}$ implies that there is a constant $c > 0$ depending only on $\eta$, $W$ and $\tilde W$ such that
\[
Q_\eta^{MM}[\tilde V] \geq c\|\nabla \tilde V\|_{L^2(B^N)}^2 \, \,  \textrm{for every } \tilde V \in H_0^1(B^N,\RR^{N+1}) \textrm{ with } \tilde V \cdot m_\eta = 0.
\]
Since $h\in C^0(\bar B^N, C^2(\RR^{N+1}))$, $h(x,0)=0$, $\nabla_y h(x, 0)=0$, $\nabla^2_y h(x,0)=0$ and $h$ satisfies the growth assumptions in Lemma \ref{Lem:RA} for $p=2$ (due to the convexity of $\tilde W$), by Lemma \ref{Lem:RA}, for any $a > 0$, there exists $\delta > 0$ such that
\[
\int_{B^N} h(x, V(x))\, dx \geq -a \|\nabla V\|_{L^2(B^N)}^2 \text{ whenever } V\in H^1_0(B^N, \RR^{N+1}), \|V\|_{H^1(B^N)} \leq \delta.
\]

Let us consider the first integral on the right hand side of \eqref{Eq:MLM-1}. We start with the term $|V|^2 - |\tilde V|^2$, using the fact that $V \cdot m_\eta = - \frac{1}{2}|V|^2$,
\[
|V|^2 - |\tilde V|^2 = |V\cdot m_\eta|^2 = \frac{1}{4}|V|^4.
\]
Likewise, since $|q| \leq |V|$, $|\tilde q| \leq |\tilde V| \leq |V|$, $0 < g_\eta \leq 1$ and $q - \tilde q = (V \cdot m_\eta) g_\eta$,
\[
|q^2 - \tilde q^2| = |q - \tilde q||q + \tilde q| \leq 2 |V \cdot m_\eta| |V| = |V|^3.
\]
Next, the term $|\nabla V|^2 - |\nabla \tilde V|^2$ is estimated as follows, using the fact that $\nabla (V - \tilde V) = \nabla((V \cdot m_\eta) m_\eta) = -\frac{1}{2} \nabla (|V|^2 m_\eta)$ and $ -m_\eta\cdot \partial_j \tilde V=\partial_j m_\eta\cdot \tilde V$ for $1\leq j\leq N$,
\begin{align*}
|\nabla V|^2 - |\nabla \tilde V|^2 
	&= |\nabla (V - \tilde V)|^2 + 2\nabla(V -\tilde V) : \nabla \tilde V
		=  |\nabla (V - \tilde V)|^2 -\nabla(|V|^2 m_\eta) : \nabla \tilde V\\
	&=  |\nabla (V - \tilde V)|^2 + \sum_{j=1}^N \partial_j (|V|^2)  \tilde V \cdot \partial_j m_\eta  -|V|^2 \nabla m_\eta : \nabla \tilde V\\
	&\geq |\nabla (V - \tilde V)|^2 -C |V|^2(|\nabla V|+|V|^2)
\end{align*}
for some $C = C(\|\nabla m_\eta\|_{C^1(\bar B^N)})$, where we have used $|V| \leq |m_\eta + V| + |m_\eta| = 2$ and $|\nabla \tilde V|= |\nabla V+\frac12\nabla(|V|^2 m_\eta)| \leq C (|\nabla V| + |V|^2)$.

Putting things together in \eqref{Eq:MLM-1} with $a = \frac{1}{8}\min(c,1)$, by the Cauchy-Schwarz and triangle inequalities, we get for all $m_\eta + V \in \mcA^{MM}$ with $\|V\|_{H^1(B^N)} \leq \delta$ that
\begin{align*}
E_{\eta}^{MM}[m_{\eta} + V] - E_{\eta}^{MM}[m_{\eps,\eta}]
	&\geq \frac{c}{2}\|\nabla \tilde V\|_{L^2(B^N)}^2 + \frac{1}{2} \|\nabla (V - \tilde V)\|_{L^2(B^N)}^2
		- a\|\nabla V\|_{L^2(B^N)}^2
	\\	
	&\qquad 
		- C(\|\nabla V\|_{L^2(B^N)}\|V\|_{L^4(B^N)}^2  + \|V\|_{L^4(B^N)}^4 + \|V\|_{L^3(B^N)}^3)\\
	&\geq \frac{\min(c,1)}{8} \|\nabla V\|_{L^2(B^N)}^2
		-  \tilde C( \|V\|_{L^4(B^N)}^4 + \|V\|_{L^3(B^N)}^3).
\end{align*}
Note also that, since $|V| \leq 2$ and by the Sobolev embedding theorem for $V\in H^1_0(B^N)$, we have for any fixed $2 < p < \min(3,\frac{2N}{N-2})$ that
\[
\|V\|_{L^4(B^N)}^4 + \|V\|_{L^3(B^N)}^3 \leq C_p\|V\|_{L^p(B^N)}^p \leq C_{N,p} \|\nabla V\|_{L^2(B^N)}^p.
\]
Putting the last two estimates together, for small $\delta>0$,  we obtain for some $\hat C>0$:
$$E^{MM}_{\eta}[m_{\eta} + V] \geq E^{MM}_{\eta}[m_{\eta}]  + \hat C\|\nabla V\|^2_{L^2(B^N)}
\,\,  \textrm{if } m_\eta+V \in \mcA^{MM} \textrm{ with }\|V\|_{H^1(B^N)}<\delta$$
yielding the desired local minimality of $m_\eta$ for $E_\eta^{MM}$ in $\mcA^{MM}$.
\end{proof}

\appendix

\section{Radially symmetric vector-valued maps}\label{App}

In the sequel, let $SO(N)$ denote the group of $N \times N$ special orthogonal matrices, equipped with the Haar measure. Naturally, $SO(N) \times B^N$ is equipped with the product measure.

\begin{definition}\label{Def:RSMap}
Let $N \geq 2$ and $k \geq 0$. A measurable map $m: B^N \rightarrow \RR^{N+k}$ is said to be $SO(N)$-equivariant, or simply radially symmetric, if 
\[
m(Rx) = \tilde Rm(x) \text{ for almost all } (R,x) \in SO(N) \times B^N,
\]
where $\tilde R = \begin{pmatrix}R &0_{N \times k}\\0_{k \times N}&I_{k \times k}\end{pmatrix}  \in SO(N + k)$, 
and $0_{i \times j}$ and $I_{k \times k}$ denote respectively the $i \times j$ zero matrix and the $k \times k$ identity matrix.
\end{definition}

\begin{lemma}\label{Lem:RSRep}
Let $N \geq 2$, $k \geq 0$ and $m \in L^1_{loc}(B^N, \RR^{N+k})$.

\begin{enumerate}[(a)]
\item If $N \geq 3$, then $m$ is radially symmetric if and only if there exist functions $f, g_1, \ldots, g_k \in L^1_{loc}(0,1)$ such that
\[
m(x) = \Big(f(|x|) \frac{x}{|x|}, g_1(|x|), \ldots, g_k(|x|)\Big) \text{ for almost all } x \in B^N.
\]

\item If $N = 2$, then $m$ is radially symmetric if and only if there exist functions $f_1, f_2 , g_1, \ldots, g_k \in L^1_{loc}(0,1)$ such that
\[
m(x) = \Big(f_1(|x|) \frac{x}{|x|} + f_2(|x|) \frac{x^\perp}{|x|} , g_1(|x|), \ldots, g_k(|x|)\Big) \text{ for almost all } x \in B^2,
\]
where $(x_1,x_2)^\perp = (-x_2, x_1)$.
\end{enumerate} 
\end{lemma}

\begin{proof}
It is clear that if $m$ has the stated form, then $m$ is radially symmetric. For the converse, suppose that $m$ is radially symmetric.

Let us make an observation on mollifications of a radially symmetric map. Let $(\varrho_\eps)$ be a sequence of smooth radially symmetric  mollifiers (i.e. $\varrho_\eps(x) = \varrho_\eps(|x|)$) satisfying $\textrm{supp}\,\varrho_\eps \subset (-\eps,\eps)$ and let $m_\eps = m * \varrho_\eps$ in $B_{1-\eps}$ where $B_r$ is the ball centered at zero of radius $r>0$. We claim that $m_\eps$ is radially symmetric in $B_{1-\eps}$. Indeed, by Fubini's theorem, for almost all $R \in SO(N)$, we have
\[
m(Rx) = \tilde Rm(x) \text{ for almost all } x \in B^N.
\]
Therefore, for almost all $R \in SO(N)$ and for all $0 < |x| < 1 - \eps$,
\begin{align*}
m_{\eps}(Rx) 
	&= \int_{B^N} m(y) \varrho_\eps(Rx - y)\,dy
		= \int_{B^N} m(Rz) \varrho_\eps(Rx - Rz)\,dz\\
	&= \int_{B^N} \tilde R m(z) \varrho_\eps(x - z)\,dz
		= \tilde R m_\eps(x),
\end{align*}
i.e. $m_\eps$ is radially symmetric in $B_{1-\eps}$.

By the above claim, it suffices to consider continuous $m$ in our proof. In this case, 
\begin{equation}
m(Rx) = \tilde Rm(x) \text{ for all } (R,x) \in SO(N) \times B^N.
	\label{Eq:EVCont}
\end{equation}
Clearly \eqref{Eq:EVCont} implies that, for $1 \leq j \leq k$ and $x \in B^N$, $m_{N + j}(Rx) = m_{N + j}(x)$ for all $R \in SO(N)$ and so $m_{N + j}(x) = g_j(|x|)$ for some $g_j \in C(0,1)$. We thus assume without loss of generality that $k = 0$, i.e., $m:B^N\to \RR^N$.

Let $e_N = (0, \ldots, 0, 1)$.  For $r \in (0,1)$, we write $m(re_N) = (a(r),b(r))$ where $a(r) \in \RR^{N-1}$ and $b(r) \in \RR$. Since $m$ is continuous, $a$ and $b$ are continuous in $(0,1)$.

\medskip
\noindent\underline{Case (a):} $N \geq 3$. Taking $R$ of the form $R = \begin{pmatrix}S & 0_{(n-1)\times 1}\\0_{1 \times (n-1)} & 1\end{pmatrix}$  where $S \in SO(N-1)$, we obtain from \eqref{Eq:EVCont} that
\[
a(r) = S a(r) \text{ for all } S \in SO(N-1).
\]
As $N \geq 3$, there exists $S(r) \in SO(N-1)$ so that $S(r)a(r) = -a(r)$ and so the above implies that $a(r) = 0$. In particular, $m(re_N) = b(r) e_N$ for every $r\in (0,1)$. Now if $|x| = r \in (0,1)$, we select $R \in SO(N)$ such that $R(re_N) = x$ and obtain from \eqref{Eq:EVCont} that
\[
m(x) = m(R(re_N)) = R m(re_N) = b(r) Re_N = b(r) \frac{x}{r}.
\]
The conclusion follows with $f(r) = b(r)$.

\medskip
\noindent\underline{Case (b):} $N = 2$. In this case, $a(r)$ is a scalar so that
\[
m(re_2) = -a(r) e_2^\perp + b(r) e_2.
\]

Now if $x = (r\cos\varphi,r\sin\varphi)$ for some $r > 0$ and $\varphi \in [0,2\pi)$, setting 
\[
R_\varphi := \begin{pmatrix}\sin\varphi & \cos\varphi\\-\cos\varphi& \sin\varphi\end{pmatrix} \in SO(2),
\] 
then we have
\[
R_\varphi(re_2) = x \text{ and }R_\varphi(re_2^\perp)  = x^\perp.
\]
We thus obtain from \eqref{Eq:EVCont} that
\[
m(x) = m(R_\varphi(re_2)) = R_\varphi m(re_2) = - a(r) R_\varphi e_2^{\perp} + b(r) R_\varphi e_2 = -a(r)\frac{x^\perp}{r} + b(r)\frac{x}{r}.
\]
The conclusion follows with $f_1(r) = b(r)$ and $f_2(r) = -a(r)$.
\end{proof}

\begin{remark}
In a similar fashion as in Definition \ref{Def:RSMap}, one can also define $O(N)$-equivariant maps. It is easy to see from the above lemma that, for $N \geq 3$ and $k \geq 0$, $SO(N)$-equivariant maps are $O(N)$-equivariant. For $N = 2$ and $k \geq 0$, $m \in L^1_{loc}(B^2;\RR^{2+k})$ is $O(2)$-equivariant if and only if there exist functions $f, g_1, \ldots g_k \in L^1_{loc}(0,1)$ such that
\[
m(x) = \Big(f(|x|) \frac{x}{|x|}  , g_1(|x|), \ldots, g_k(|x|)\Big) \text{ for almost all } x \in B^2.
\]
This is because the map $x \mapsto f_2(|x|) \frac{x^\perp}{|x|}$ is $O(2)$-invariant if and only if $f_2 = 0$, in view of the fact that $(Rx)^\perp = -R(x^\perp)$ with $R$ being the reflection about the $x_1$-axis, i.e. $R(x_1,x_2) = (x_1,-x_2)$.
\end{remark}

\begin{lemma}\label{Lem:ONSym}
Suppose $N \geq 2$, $\eps > 0$ and $W \in C^2((-\infty,1])$. If $m$ is a bounded\footnote{If $W$ satisfies the condition \eqref{Eq:WCond'}, then  the boundedness of $m$ is a consequence of Corollary \ref{Cor:fFinite}.} radially symmetric critical point of $E_{\eps}^{GL}$ in $\mcA^{GL}$, then $m \in C^2(\bar B^N)$ and takes the form
\[
m(x) = f(|x|) \frac{x}{|x|}
\]
for some $f \in C^2([0,1])$ with $\frac{f}{r} \in C^2([0,1])$. In particular, $f(0) = 0$ and $m$ is $O(N)$-equivariant.
\end{lemma}
 
\begin{lemma}\label{Lem:FullONSym}
Suppose $N \geq 2$, $\eps, \eta > 0$, $W \in C^2((-\infty,1])$ and $\tilde W \in C^2([0,\infty))$. If $m$ is a bounded\footnote{If $W$ and $\tilde W$ satisfy the conditions \eqref{Eq:WCond'}-\eqref{Eq:TWCond'}, then  the boundedness of $m$ follows from Lemma~\ref{lem:f2g2}.} radially symmetric critical point of $E_{\eps,\eta}$ in $\mcA$, then $m \in C^2(\bar B^N)$ and takes the form
\[
m(x) = (f(|x|) \frac{x}{|x|} ,g(|x|))
\]
for some $f, g \in C^2([0,1])$ with $\frac{f}{r} \in C^2([0,1])$. In particular, $f(0) = 0$, $g'(0) = 0$ and $m$ is $O(N)$-equivariant.
\end{lemma}

We will only give the proof of the latter one. The proof of the other one requires minor modifications and is omitted.

\begin{proof}[Proof of Lemma \ref{Lem:FullONSym}.] As a bounded radially symmetric critical point of $E_{\eps,\eta}$, $m$ satisfies
\begin{equation}
\left\{\begin{array}{rcl} -\Delta m - \frac{1}{\eps^2} W'(1 - |m|^2) m + \frac{1}{\eta^2} \tilde W'(m_{N+1}^2) m_{N+1} e_{N+1} &=& 0 \text{ in } B^N \setminus \{0\},\\
m(x) & =& x \text{ on } \partial B^N.
\end{array}\right.
	\label{Eq:mPDE}
\end{equation}
Due to $m\in H^1\cap L^\infty(B^N)$ (in particular, $W'(1 - |m|^2),  \tilde W'(m_{N+1}^2)\in L^\infty(B^N)$), it follows that \eqref{Eq:mPDE} holds in all of $B^N$, and, by elliptic regularity theory, $m \in C^2(\bar B^N)$.

On the other hand, using Lemma \ref{Lem:RSRep} and the regularity of $m$, we write
\begin{equation}
m(x) = \left\{\begin{array}{ll}  \Big(f_1(|x|) \frac{x}{|x|} + f_2(|x|) \frac{x^\perp}{|x|} , g(|x|)\Big) &\text{ if } N = 2,\\
	\Big(f_1(|x|) \frac{x}{|x|} , g(|x|)\Big) &\text{ if } N \geq 3,
\end{array}\right.
	\label{Eq:M2-24V21}
\end{equation}
where $f_1, f_2 \in C^2\cap L^\infty((0,1])$ and $g \in C^2([0,1])$ with $f_1(0) = f_2(0) = 0$ and $g'(0) = 0$. 

To conclude, we show that $\frac{f_1}{r} \in C^2([0,1])$ and, when $N = 2$, $f_2 = 0$ in $(0,1)$.

Let us show that $f_2 = 0$ in $(0,1)$ when $N = 2$. We use ideas from the proof of \cite[Proposition 2.3]{INSZ_AnnIHP}. From \eqref{Eq:mPDE}, we have that 
\[
\nabla \cdot (- m_2 \nabla m_1 + m_1 \nabla m_2) =  (m_1, m_2)^\perp  \cdot \Delta (m_1,m_2)= 0 \text{ in } B^2.
\]
Integrating over balls $B_r$ of radius $r \in (0,1)$, the Gauss formula yields
\begin{equation}
\int_{\partial B_r} (m_1, m_2)^\perp \cdot  \partial_r (m_1,m_2)\,dS = \int_{\partial B_r} (- m_2 \partial_r m_1 + m_1 \partial_r m_2 )\,dS = 0.
	\label{Eq:M1-24V21}
\end{equation}
Using \eqref{Eq:M2-24V21} in \eqref{Eq:M1-24V21}, we obtain
\begin{equation}
-f_1'f_2 + f_2' f_1 = 0 \text{ in } (0,1).
	\label{Eq:M3-24V21}
\end{equation}
Since $f_1(1) = 1$, we have that $f_1 > 0$ in some interval $(r_1, 1)$ with $0 \leq r_1 < 1$. Dividing \eqref{Eq:M3-24V21} by $f_1^2$ in $(r_1, 1)$, we get $(f_2/f_1)' = 0$, and using the fact that $f_2(1) = 0$, we have $f_2 = 0$ in $(r_1,1)$. In particular $f_2'(1) = 0$. Now, by \eqref{Eq:mPDE}, we have that
\be
\label{try}
f_2'' + \frac{N-1}{r} f_2' + c(r) f_2 = 0 \text{ in } (0,1),
\ee
where $c(r) :=  - \frac{N-1}{r^2}  + \frac{1}{\eps^2} W'(1 - f_1^2 - f_2^2 - g^2)$ belongs to $C^1((0,1])$. Since $f_2(1) = f_2'(1) = 0$, standard uniqueness results for ODEs implies that $f_2 = 0$ in $(0,1)$ as desired.

Let us show next that $\frac{f_1}{r} \in C^2([0,1])$ for any $N \geq 2$. By \eqref{Eq:mPDE} and \eqref{Eq:M2-24V21}, we have
\[
f_1'' + \frac{N-1}{r} f_1' +  \Big(- \frac{N-1}{r^2}  + \frac{1}{\eps^2} W'(1 - f_1^2 - g^2)\Big) f_1 = 0 \text{ in } (0,1).
\]
Setting $v = \frac{f_1}{r}$ and $d = \frac{1}{\eps^2} W'(1 - f_1^2 - g^2)= \frac{1}{\eps^2} W'(1 -|m|^2)\in C^1([0,1])$ (as $m\in C^2(\bar B^N)$), we then have 
\[
v'' + \frac{N+1}{r} v' + d(r) v(r) = 0 \text{ in } (0,1).
\]
Considering $v$ as a radially symmetric function on the $(N+2)$-dimensional ball $B^{N+2}$, we have that $v$ satisfies $\Delta v + dv = 0$ in $B^{N+2} \setminus \{0\}$. On the other hand, since $m  \in H^1(B^N)$, we have $r^{\frac{N-1}{2}} f_1', r^{\frac{N-3}{2}}f_1 \in L^2(0,1)$ and so $v \in H^1(B^{N+2})$. It follows that $\Delta v + dv = 0$ in $B^{N+2}$ and since $d\in  C^1([0,1])$, we deduce that $v \in C^2(B^{N+2})$. The conclusion follows.
\end{proof}

\begin{lemma}\label{Lem:MMONSym}
Suppose $N \geq 2$, $\eta > 0$, and $\tilde W \in C^2([0,1])$. If $m$ is a radially symmetric critical point of $E_{\eta}^{MM}$ in $\mcA^{MM}$, then $m$ takes the form
\begin{equation}
m(x) = (f(|x|) \frac{x}{|x|} ,g(|x|))
	\label{Eq:mF-25VI21}
\end{equation}
for some $f, g \in C^2_{loc}((0,1])$ with $f^2 + g^2 = 1$ and $r^{\frac{N-1}{2}}(|f'| + |g'|) + r^{\frac{N-3}{2}}|f| \in L^2(0,1)$. In particular, $m$ is $O(N)$-equivariant. Furthermore, either $\frac{f}{r}, g \in C^2([0,1])$ or both $(f,g) \equiv (1,0)$ and $N \geq 3$, where in the former case one has also that $m \in C^2(\bar B^N)$, $f(0) = 0$ and $g'(0) = 0$.
\end{lemma}

\begin{proof} We adapt the proof of Lemma \ref{Lem:FullONSym}. Without loss of generality, we may assume that $\tilde W(0) = 0$. As a critical point of $E_{\eta}^{MM}$ in $\mcA^{MM}$, $m$ satisfies
\begin{equation}
\left\{\begin{array}{rcl} -\Delta m - \lambda(x) m + \frac{1}{\eta^2} \tilde W'(m_{N+1}^2) m_{N+1} e_{N+1} &=& 0 \text{ in } B^N,\\
m(x) & =& x \text{ on } \partial B^N,
\end{array}\right.
	\label{Eq:mMMPDE}
\end{equation}
where $\lambda = |\nabla m|^2 + \frac{1}{\eta^2} \tilde W'(m_{N+1}^2) m_{N+1}^2 \in L^1(B^N)$. By Lemma \ref{Lem:RSRep}, $m$ takes the form \eqref{Eq:M2-24V21}. In particular, $\lambda=\lambda(r)\in L^1_{loc}((0,1])$, which together with \eqref{Eq:mMMPDE} (recast as ODEs for $f_1, f_2, g$) implies that $f_1'', f_2'', g'' \in L^1_{loc}((0,1])$ where $f_2$ is absent when $N\geq 3$. This in turn implies that $\lambda\in C^0((0,1])$ and then again, by regularity theory, $f_1, f_2, g \in C^2((0,1])$ (and hence $m \in C^2(\bar B^N \setminus \{0\})$. Next, as in the proof of Lemma \ref{Lem:FullONSym}, when $N = 2$, we prove that \eqref{Eq:M1-24V21}-\eqref{Eq:M3-24V21} hold also here yielding $f_2=0$ in $(0,1)$. We have thus shown that $m$ has the form \eqref{Eq:mF-25VI21} where $f^2 + g^2 = 1$, $r^{\frac{N-1}{2}}(|f'| + |g'|) + r^{\frac{N-3}{2}}|f| \in L^2(0,1)$, and $f, g \in C^2((0,1])$.

\medskip
\noindent
\underline{Step 1:} {\it We prove that $f, g \in C([0,1])$.} We distinguish the cases $N = 2$ and $N \geq 3$. 

\medskip

\noindent {\it Case 1: $N = 2$}. It is known that the continuity of $m$ in $\bar B^2$ can be proved using Wente's lemma (see e.g. H\'elein \cite{Helein} or Carbou \cite[Theorem 1]{Carbou97-CVPDE}). However, in this ODE setting, the continuity of $f$ (and hence of $g$) in $[0,1]$ is a consequence of the fact that $r^{\frac{1}{2}}|f'|  + r^{-\frac{1}{2}}|f| \in L^2(0,1)$, since
\[
|f^2(r_1) - f^2(r_2)| 
	\leq 2\int_{r_2}^{r_1} |f'(r)||f(r)|dr 
	\leq \int_{r_2}^{r_1} (r|f'(r)|^2 + \frac{1}{r}|f(r)|^2)\,dr
		 \stackrel{r_1,r_2 \rightarrow 0}{\longrightarrow} 0.
\]
Also, since $r^{-\frac{1}{2}}|f| \in L^2(0,1)$, we also have that $f(0) = 0$. It follows that $m \in C(\bar B^2)$. 

\medskip

\noindent {\it Case 2: $N \geq 3$}. As $f, g\in C^2((0,1])$ and $f^2 + g^2 = 1$, we can find a lifting $\theta \in C^2((0,1])$ such that $r^{\frac{N-1}{2}}|\theta'| \in L^2(0,1)$, $f = \sin \theta$, $g = \cos\theta$ in $(0,1]$ and $\theta(1) = \pi/2$. (To prepare for Steps 2 and 3 later on, we note that the existence of such a lifting $\theta$ also holds for $N = 2$ where we have in addition to the above that $\theta \in C([0,1])$, $r^{-1/2}\sin\theta \in L^2(0,1)$ and $\theta(0) \in \pi\mathbb{Z}$.)

A direct computation using \eqref{Eq:mMMPDE} gives
\begin{equation}
\theta'' + \frac{N-1}{r} \theta' - \frac{N-1}{r^2} \sin \theta\,\cos\theta  + \frac{1}{\eta^2} \tilde W'(\cos^2\theta) \sin\theta\cos\theta = 0 \text{ in } (0,1).
	\label{Eq:T1-21VI21}
\end{equation}
Set $F(r) =  [(N-1)  - \frac{1}{\eta^2} r^2\tilde W'(\cos^2\theta(r)) ]\sin \theta(r)\,\cos\theta(r) \in L^\infty(0,1)$ so that \eqref{Eq:T1-21VI21} is equivalent to $(r^{N-1} \theta')' = F(r) r^{N-3}$. Therefore, for some constant $c$,
\[
\theta'(r) = \frac{c}{r^{N-1}} +  \frac{1}{r^{N-1}} \int_0^r F(s)\,s^{N-3}\,ds = \frac{c}{r^{N-1}} + O(\frac{1}{r}) \text{ as } r\rightarrow 0.
\]
Using that $r^{\frac{N-1}{2}}|\theta'| \in L^2(0,1)$, we deduce that $c = 0$ and
\begin{equation}
\theta'(r) = \frac{1}{r^{N-1}} \int_0^r F(s)\,s^{N-3}\,ds.
	\label{Eq:TInt-0}
\end{equation}
It follows that, for some positive constant $C$ independent of $r$,
\begin{equation}
|\theta'(r)| \leq \frac{C}{r} \quad\text{ and }\quad |\theta(r)| \leq C(1 + |\log r|) \text{ in }(0,1).
	\label{Eq:TInt-1}
\end{equation}

\medskip

\noindent 
{\it Claim:} We prove that $\theta \in C([0,1])$ and $\theta(0)= \frac{k\pi}2$ for some $k\in \ZZ$.

\medskip

\noindent 
{\it Proof of Claim}: Indeed, let 
\[
P(r) = r^2 (\theta')^2 + (N-1) \cos^2\theta - \frac{r^2}{\eta^2} \tilde W(\cos^2\theta).
\] 
By \eqref{Eq:TInt-1}, $P \in L^\infty(0,1)$. Multiplying \eqref{Eq:T1-21VI21} by $2r^2 \theta'$, we see that 
\begin{equation}
P'(r) = - 2(N-2)r (\theta')^2 - \frac{2r}{\eta^2} \tilde W(\cos^2\theta).
	\label{Eq:Pohozaev}
\end{equation}
In particular, the function $\tilde P(r) := P(r) + \int_0^r \frac{2s}{\eta^2} \tilde W(\cos^2\theta(s))\,ds$ satisfies $\tilde P\in L^\infty(0,1)$ and $\tilde P'(r) = - 2(N-2)r (\theta')^2 \leq 0$. It follows that $r (\theta')^2 =  \frac{1}{2(N-2)} |\tilde P'| \in L^1(0,1)$ and $\tilde P, P \in W^{1,1}(0,1)\subset  C([0,1])$.

By \eqref{Eq:TInt-0} and integrating by parts,
\begin{align*}
\theta'(r) 
	&= \frac{F(r)}{(N-2) r} -  \frac{1}{(N-2) r^{N-1}} \int_0^r F'(s)\,s^{N-2}\,ds.
\end{align*}
Since $|F'(r)|\leq C(|\theta'(r)|+r)$ for every $r\in (0,1)$, we obtain
\begin{align*}
|F(r)| 
	&= \Big|(N-2)r \theta'(r) + \frac{1}{r^{N-2}} \int_0^r F'(s)\,s^{N-2}\,ds\Big|\\
	&\leq Cr^2 +  C r |\theta'(r)| +  \frac{C}{r^{N-2}} \int_0^r |\theta'(s)|\,s^{N-2}\,ds.
\end{align*}
Noting that, by Cauchy-Schwarz' inequality,
\[
\int_0^r |\theta'(s)| s^{N-2}\,ds \leq Cr^{N-2} \Big(\int_0^r s|\theta'(s)|^2\,ds\Big)^{1/2},
\]
we deduce from the above bound for $|F|$ that
\begin{align*}
\int_0^r |F(s)|\,s^{N-3}\,ds
	&\leq Cr^{N} + C\underbrace{\int_0^r |\theta'(s)|\,s^{N-2}\,ds}_{\leq Cr^{N-2} (\int_0^r s|\theta'(s)|^2\,ds)^{1/2}} + C\int_0^r \underbrace{\frac{1}{s} \int_0^s  |\theta'(t)|\,t^{N-2}\,dt}_{\leq Cs^{N-3}(\int_0^r t|\theta'(t)|^2\,dt)^{1/2}}\,ds\\
	&\leq Cr^{N} + Cr^{N-2} \Big(\int_0^r s|\theta'(s)|^2\,ds\Big)^{1/2}.
\end{align*}
Returning to \eqref{Eq:TInt-0}, since $r|\theta'(r)|^2\in L^1(0,1)$, we have that
\[
r|\theta'(r)| \leq Cr^2 + C \Big(\int_0^r s|\theta'(s)|^2\,ds\Big)^{1/2} \rightarrow 0 \text{ as } r \rightarrow 0.
\]
Recalling the expression of $P$ and its continuity, we deduce that $\cos^2\theta$ and hence $\theta$ belong to $C([0,1])$. By \eqref{Eq:TInt-0} and the continuity of $F$, $r\theta'(r) = \frac{1}{N-2} F(0) + o(1)$ for small $r > 0$. We hence have that $F(0) = 0$, i.e. $\theta(0) = \frac{k \pi}{2}$ for some $k \in \mathbb{Z}$. 
Step 1 is now completed.

\medskip
\noindent\underline{Step 2:} {\it We prove that if $k$ is odd, then $(f,g) \equiv (1,0)$ and $N \geq 3$.}

When $k$ is odd, $f(0) \neq 0$. We saw in Step 1 that this is possible only if $N \geq 3$.

In the absence of $\tilde W$ (i.e. for the harmonic map problem), the assertion that $(f,g) \equiv (1,0)$ can be dealt as in \cite{JagerKaul83-JRAM} as follows: \eqref{Eq:Pohozaev} implies that $P' \leq 0$, which leads to $0 = P(0) \geq P(r) \geq P(1) = (\theta'(1))^2 \geq 0$. Thus $\theta'(1) = 0$; since $\theta(1)=\frac\pi2$, uniqueness results for second order ODEs give that $\theta \equiv \frac{\pi}{2}$.

To account for the presence of $\tilde W$ in \eqref{Eq:T1-21VI21}, we argue as follows. By \eqref{Eq:Pohozaev}, $P'(r) \leq 2ar$ for some constant $a > 0$. Since $r\theta'(r) \rightarrow 0$ as $r \rightarrow 0$ and $k$ is odd, $P(r) \rightarrow 0$ as $ r\rightarrow 0$. Hence $P(r) \leq ar^2$. By \eqref{Eq:Pohozaev}, we have $(r^{-2}P)' \leq 0$ and since $\cos \theta(1)=0$, $\tilde W(0)=0$, 
\begin{equation}
P(r) \geq P(1) r^2 = (\theta'(1))^2 r^2 \geq 0 \text{ in } (0,1).
	\label{Eq:PoC-1}
\end{equation}
Also by \eqref{Eq:Pohozaev}, we have that
\[
(r^{-1}P)' \leq  - \frac{(N-1)}{r^2} \cos^2\theta - \frac{1}{\eta^2} \tilde W(\cos^2\theta).
\]
Using the fact that $\cos\theta(0) = 0$, $\tilde W(0) = 0$ and $\tilde W \in C^1$, in particular, $|\tilde W(t)|\leq \tilde c t$ for $t\in [0,1]$, we thus have that $(r^{-1}P)' \leq 0$ in some interval $(0,r_0)$. But as $r^{-1}P(r) \rightarrow 0$ as $r \rightarrow 0$ (as $0\leq P(r)\leq a r^2$), we deduce that 
\begin{equation}
P(r) \leq 0 \text{ in } (0,r_0)
	\label{Eq:PoC-2}
\end{equation}
and so, $P\equiv 0$ in $(0,r_0)$. Putting together \eqref{Eq:PoC-1} and \eqref{Eq:PoC-2}, we have that $\theta'(1) = 0$. By uniqueness results for ODEs, we then have that $\theta \equiv \frac{\pi}{2}$, i.e. $(f,g) \equiv (1,0)$.

\medskip
\noindent\underline{Step 3:} {\it We prove that if $\theta(0)\in \pi \ZZ$ and $N\geq 2$, then $\frac{f}{r}, g \in C^2([0,1])$.} Since $\theta(0)\in \pi \ZZ$,
\[
F(r) = (N-1)  d(r) (\theta(r) - \theta(0)) \text{ where } d(r) = 1 + O(r^2 + |\theta(r) - \theta(0)|^2) \text{ as } r \rightarrow 0.
\]
We can then recast \eqref{Eq:T1-21VI21} in the form
\[
L(\theta - \theta(0)) := (\theta - \theta(0))'' + \frac{N-1}{r} (\theta - \theta(0))' - \frac{(N-1) d(r)}{r^2} (\theta - \theta(0)) = 0.
\]
It is straightforward to check that, for $\delta \in (0,1)$, there exists $r_\delta > 0$ such that
\[
L(r^{-(N-1)+\delta}) < 0  \text{ and }  L(r^{1 - \delta}) < 0 \text{ in } (0,r_\delta).
\]
Thus, by the maximum principle (see e.g. \cite[Lemma B.1]{ODE_INSZ}), we have that
\[
\frac{|\theta(r_\delta) - \theta(0)|}{r_\delta^{1-\delta}} r^{1-\delta}  \pm (\theta(r) - \theta(0)) \geq 0 \text{ in } (0,r_\delta).
\]
This shows that $r^{-(1-\delta)} |\theta - \theta(0)| \in L^\infty(0,1)$ for all $\delta \in (0,1)$.

Taking $\delta = 1/2$ above, we have that $d(r) = 1 + O(r)$. Then, for some large $A > 0$ and small $r_0 > 0$, we have
\[
L(r - Ar^2) < 0 \text{ and } r - Ar^2 > 0 \text{ in } (0,r_0).
\]
Again, by the maximum principle, we then have that 
\[
\frac{|\theta(r_0) - \theta(0)|}{r_0 - Ar_0^2} (r - Ar^2)   \pm (\theta(r) - \theta(0)) \geq 0 \text{ in } (0,r_0).
\]
We thus have that $r^{-1}(\theta - \theta(0)) \in L^\infty(0,1)$. This yields $F(r)=O(r)$ and by \eqref{Eq:TInt-0}, 
\[
\theta' \in L^\infty(0,1).
\]
Since $f(0)=\sin \theta(0)=0$, we get $\frac{f}{r}\in L^\infty(0,1)$. Returning to $m$, as $|\nabla m|^2=(\theta')^2+\frac{(N-1)f^2}{r^2}$, we see that $m \in C^{0,1}(B^N)$ and $\lambda\in L^\infty(B^N)$ (given in \eqref{Eq:mMMPDE}), and by bootstrapping \eqref{Eq:mMMPDE}, $m \in C^2(B^N)$ and $\lambda\in C^1(B^N)$. By the same argument in Lemma \ref{Lem:FullONSym}, it follows that $\frac{f}{r},g \in C^2([0,1])$, $f(0) = 0$ and $g'(0) = 0$ as desired.
\end{proof}

\section{Some properties of the $\RR^N$-valued GL vortex radial profile}

\begin{proposition}\label{Prop:B.1}
Suppose that $N \geq 2$, $W \in C^2((-\infty,1])$ satisfies \eqref{Eq:WCond'} and let $f_\eps:[0,1]\to [0,1]$ be given by Theorem \ref{Thm:GLExist} and $f_\eps^{-1}:[0,1]\to [0,1]$ its inverse. Then: 
\begin{enumerate}[(i)]
\item For $0 < \tilde\eps \leq \eps$, $f_\eps(\eps r) \geq f_{\tilde \eps}(\tilde\eps r)$ \text{ for } $0 < r < 1/\eps$.

\item If $W'(1) > 0$ and $t_0 := \sup \{0 \leq t < 1: W(t) = 0\}$, then $t_0 < 1$, $\lim_{\eps\rightarrow 0} \frac{f_\eps^{-1}(\sqrt{1 - t_0})}{\eps} = \infty$, and, for every $\delta \in (0,1-t_0)$, $\lim_{\eps\rightarrow 0} \frac{f_\eps^{-1}(\sqrt{1 - t_0 - \delta})}{\eps} \in (0,\infty)$. In particular, for every $a > 0$, there exists $\eps_{a} > 0$ such that 
\[
f_\eps^2 \leq 1 - t_0 \text{ in } [0,a \eps] \text{ for every } \eps \in (0, \eps_{a}],
\]
and, for every $\delta \in (0,1-t_0)$, there exists $C_\delta > 0$ such that
\[
1 - t_0 - \delta \leq f_\eps^2\text{ in } [C_\delta\eps,1] \text{ for every } \eps \in (0,1/C_\delta].
\]
\end{enumerate}
\end{proposition}

\begin{proof}
For $\eps > 0$, define
\[
\hat f_\eps(r) = \left\{\begin{array}{ll}
f_\eps(\eps r) & \text{ if } r \in (0,1/\eps),\\
1 & \text{ if } r \in (1/\eps,\infty).
\end{array}\right.
\]
Note that 
\[
\hat f_ \eps'' + \frac{N-1}{r} \hat f_\eps' - \frac{N-1}{r^2} \hat f_\eps 
	= -W'(1 - \hat f_\eps^2) \hat f_\eps  \text{ in } (0,1/{\eps})
\]
and, the function $\hat v_\eps := \frac{\hat f_\eps}{r}$, considered as a radially symmetric function in $\RR^{N+2}$ satisfies
\begin{equation}
\Delta \hat v_\eps =  -W'(1 - \hat f_\eps^2) \hat v_\eps \leq 0 \text{ in } B(0,1/{\eps}).
	\label{Eq:Delhef}
\end{equation}
As at the end of the proof of Proposition \ref{Prop:ODEMonotonicity}, we deduce that $\hat v_\eps$ is non-increasing in $(0,1/{\eps})$ and so in $(0,\infty)$.

\medskip
\noindent\underline{Proof of (i).} This is equivalent to prove that $\hat f_\eps \geq \hat f_{\tilde \eps}$ for $0 < \tilde\eps \leq \eps$. This is a direct consequence of the comparison principle\footnote{Though the comparison principle \cite[Proposition 3.5]{ODE_INSZ} was stated with the assumption that $W' > 0$ in $(0,1)$ and $W''(0) > 0$, it is straightforward to see that it remains valid under the weaker condition that $W' \geq 0$ in $(0,1)$. Alternatively, one can first apply \cite[Proposition 3.5]{ODE_INSZ} for the unique radial profiles corresponding to the strictly convex potentials $t \mapsto W(t) + \delta t^2$ with $\delta > 0$ and then send $\delta \rightarrow 0$.} \cite[Proposition 3.5]{ODE_INSZ} and the fact that $\hat f_\eps'(0) = \hat v_\eps(0) > 0$ (since $\frac{\hat f_\eps}{r} = \hat v_\eps$ is non-increasing), $\hat f_\eps(1/{\tilde\eps}) = \hat f_{\tilde \eps}(1/{\tilde \eps}) = 1$, and
\begin{align*}
\hat f_{\tilde \eps}'' + \frac{N-1}{r} \hat f_{\tilde \eps}' - \frac{N-1}{r^2} \hat f_{\tilde \eps} 
	&= -W'(1 - \hat f_{\tilde \eps}^2) \hat f_{\tilde \eps}  \text{ in } (0,1/{\tilde\eps}),\\
\hat f_ \eps'' + \frac{N-1}{r} \hat f_\eps' - \frac{N-1}{r^2} \hat f_\eps 
	&\leq -W'(1 - \hat f_\eps^2) \hat f_\eps  \text{ in } (0,1/{\tilde\eps}).
\end{align*}

\medskip
\noindent\underline{Proof of (ii).} By \eqref{Eq:WCond'}, we have $t_0 < 1$, $W > 0$ and $W' > 0$ in $(t_0,1]$. We need to prove
\begin{equation}
\lim_{\eps\rightarrow 0} \hat f_\eps^{-1}(\sqrt{1 - t_0}) = \infty \text{ and } \lim_{\eps\rightarrow 0} \hat f_\eps^{-1}(\sqrt{1 - t_0 - \delta}) \in (0,\infty).
	\label{Eq:B1e}
\end{equation}
By (i), $\{\hat f_\eps\}$ is non-increasing as $\eps \rightarrow 0$ and hence converges pointwise to some limit function $\hat f_*$. In particular, $\hat f_*(0) = 0$, $0 \leq \hat f_* \leq 1$ in $(0,\infty)$, $\hat f_*$ is continuous at $0$,  and, by the monotonicity of $\hat f_\eps$, $\hat f_*$ is non-decreasing. By the equation of $\hat f_\eps$ and the bound $0 \leq \hat f_\eps \leq 1$, for every compact interval $[1/C,C] \subset (0,\infty)$, the family $\{\hat f_\eps\}_{0 < \eps < 1/C}$ is bounded in $C^{3}([1/C,C])$. By the Arzel\`a-Ascoli theorem, it follows that $\hat f_* \in C^2((0,\infty))$, $\hat f_\eps$ converges to $\hat f_*$ in $C^2_{loc}((0,\infty))$ as $\eps \rightarrow 0$ and 
\[
\hat f_ *'' + \frac{N-1}{r} \hat f_*' - \frac{N-1}{r^2} \hat f_* 
	= -W'(1 - \hat f_*^2) \hat f_*  \text{ in } (0,\infty).
\]
Since $W' > 0$ in $(t_0,1]$, one can argue as in Step 3 of the proof of \cite[Proposition 2.4]{ODE_INSZ} to show that $W'(1 - \hat f_*(\infty)^2) \hat f_*(\infty) = 0$, which implies that $\hat f_*(\infty) \in \{0\} \cup [\sqrt{1 - t_0}, 1]$. Moreover, using again that $W' > 0$ in $(t_0,1]$, we can argue as in Steps 4 and 5 of the proof of \cite[Proposition 2.4]{ODE_INSZ} to show that $\hat f_* \not\equiv 0$ and so $\hat f_*(\infty) \in [\sqrt{1 - t_0}, 1]$. Differentiating the equation for $\hat f_*$ and applying the strong maximum principle, we have that $\hat f_*' > 0$ in $(0,\infty)$.

\medskip
\noindent \underline{Claim}: $\hat f_*(\infty) = \sqrt{1 - t_0}$. Once this claim is proved, since $\{\hat f_\eps^{-1}\}$ is non-decreasing as $\eps \rightarrow 0$, the desired estimate \eqref{Eq:B1e} follows.

\noindent \underline{Proof of the claim}: Indeed, suppose by contradiction that this does not hold, i.e. $\hat f_*(\infty) > \sqrt{1 - t_0}$. Then we can select $r_0 \in (0,\infty)$ so that $\hat f_*(r_0) = \sqrt{1 - t_0}$, $\hat f_* \in [\sqrt{1-t_0},1]$ and so $W'(1 - \hat f_*^2) = 0$ in $[r_0,\infty)$. It follows that $
\hat f_ *'' + \frac{N-1}{r} \hat f_*' - \frac{N-1}{r^2} \hat f_*  = 0$ in $[r_0,\infty)$ and so
\[
\hat f_*(r) = c_1 r + c_2 r^{1-N} \text{ in } [r_0,\infty) \text{ for some constants } c_1, c_2.
\]
Since $\hat f_*$ is bounded, we must have $c_1 = 0$, which implies that $\hat f_*(\infty) = 0$, which gives a contradiction. The claim is proved.
\end{proof}

\section{A sharp Poincar\'e inequality for solenoidal vector fields on the sphere}\label{App:C}
\begin{lemma}\label{Lem:SolPoincare}
Suppose $N \geq 3$ and let $\slashed{D}$ and $d\sigma$ denote the covariant derivative and the volume form on the standard sphere $\Sphere^{N-1}$. For every smooth divergence-free vector field $v$ on $\Sphere^{N-1}$, i.e., $\slashed{D}\cdot v=0$ on $\Sphere^{N-1}$, one has
\[
\int_{\Sphere^{N-1}}|\slashed{D} v|^2\,d\sigma = (N-2)\int_{\Sphere^{N-1}}|v|^2\,d\sigma + 2\int_{\Sphere^{N-1}} |Sym(\slashed{D} v)|^2\,d\sigma.
\]
In particular,
\[
\int_{\Sphere^{N-1}}|\slashed{D} v|^2\,d\sigma \geq (N-2)\int_{\Sphere^{N-1}}|v|^2\,d\sigma,
\]
and equality holds if and only if $v$ is a Killing field, i.e. $Sym(\slashed{D} v) = 0$.
\end{lemma}

\begin{proof} In the following computation, we raise and lower indices using the standard metric $g$ on the round sphere, i.e. $\slashed{D}^i = g^{ij}\slashed{D}_j$, $v_i = g_{ij} v^j$, etc. Also, repeated upper-lower indices are summed from $1$ to $N-1$. As the commutator $[\slashed{D}^j, \slashed{D}_i]v_j=Ric_{ki} v^k$, integration by parts yields:
\begin{align*}
\int_{\Sphere^{N-1}} \slashed{D}_i v_j \slashed{D}^j v^i\,d\sigma 
	&= -\int_{\Sphere^{N-1}} \slashed{D}^j \slashed{D}_i v_j  v^i\,d\sigma 
	= -\int_{\Sphere^{N-1}} \Big(\slashed{D}_i  \underbrace{ \slashed{D}^jv_j }_{=0} + \underbrace{Ric_{ki}}_{=(N-2)g_{ki}} \,v^k\Big) v^i\,d\sigma \\
	&= -(N-2)\int_{\Sphere^{N-1}} |v|^2\,d\sigma.
\end{align*}
It follows that
\[
4\int_{\Sphere^{N-1}} |Sym(\slashed{D} v)|^2\,d\sigma=\int_{\Sphere^{N-1}} |\slashed{D}_i v_j + \slashed{D}_j v_i|^2\,d\sigma
	= 2\int_{\Sphere^{N-1}}\Big[|\slashed{D} v|^2 - (N-2)|v|^2\Big]\,d\sigma,
\]
which clearly gives the assertion.
\end{proof}

\section{Miscellaneous}

\begin{lemma}\label{Lem:RA}
Suppose $N \geq 2$, $M \geq 1$, and $2 \leq p < \infty$ if $N = 2$ and $2 \leq p \leq \frac{2N}{N-2}$ if $N \geq 3$. Let $\Omega$ be a bounded smooth open subset of $\RR^N$ and $h \in C^0(\Omega\times \RR^M)$ satisfies 
\be
\label{rrr}
\lim_{|y| \rightarrow 0, \, y\neq 0} \sup_{x\in \Omega}\frac{|h(x, y)|}{|y|^2} = 0
\ee
and, for some $C > 0$,
\be
\label{add-now}
h(x, y) \geq -C|y|^2(|y|^{p-2} + 1) \text{ for all } x\in \Omega, y \in \RR^M.
\ee
Then
\[
\liminf_{\substack{\|v\|_{H^1(\Omega,\RR^M)} \rightarrow 0\\v\neq 0, \, v\in H_0^1(\Omega,\RR^M)}} \frac{\int_{\Omega} h(x, v(x))\,dx}{\|v\|_{H^1(\Omega,\RR^M)}^2} \geq 0.
\]
\end{lemma}

Note that by the Sobolev embedding theorem and the lower bound of $h$, the integral $\int_{\Omega} h(x, v(x))\,dx \in \RR \cup \{+\infty\}$ makes sense for $v \in H_0^1(\Omega,\RR^M)$.

\begin{proof}
Suppose by contradiction that the conclusion fails. Then there exist $t_j \rightarrow 0^+$ and $v_j \in H_0^1(\Omega,\RR^M)$ with $\|v_j\|_{H^1} = 1$ such that, for some $\eps > 0$ independent of $j$,
\begin{equation}
\int_{\Omega} \frac{1}{t_j^2} h(x, t_j v_j(x))\,dx \leq -\eps < 0.
	\label{Eq:RA-1}
\end{equation}
Without loss of generality, we may also assume that $v_j$ converges weakly in $H^1$ and a.e. in $\Omega$ to some $v \in H_0^1(\Omega,\RR^M)$.

Fix some small $\delta > 0$. By Egorov's theorem, we can select a measurable set $A \subset  \Omega$ such that $v_j$ converges uniformly to $v$ in $A$ and $|\Omega\setminus A| \leq \delta/2$. Also, since $v \in L^2(\Omega)$, then for large $K=K(\delta)\geq 1$, we can select a measurable set $B \subset A$ such that $|v| \leq K$ in $B$ and $|A \setminus B| \leq \delta/2$. In particular, we have $|v_j| \leq 2K$ in $B$ for all large $j$. Hence, by \eqref{rrr}, 
\[
\lim_{j \rightarrow \infty} \int_B \frac{1}{t_j^2} |h(x, t_j v_j(x))| \,dx  = 0.
\]

Let $q = \frac{2N}{N-2}$ if $N \geq 3$ and $q$ be arbitrary in $(p,\infty)$ if $N = 2$. Using the bound $h(x, y) \geq -C 
|y|^2 (|y|^{p-2} + 1)$, H\"older's inequality, the Sobolev embedding theorem for $\|v_j\|_{H^1} = 1$ and the fact that $|\Omega \setminus B| \leq \delta$, we have for some constant $C' > 0$ (independent of $\delta$) that
\[
\int_{\Omega \setminus B}  \frac{1}{t_j^2} h(x, t_j v_j(x)) \,dx
	\geq - C \int_{\Omega \setminus B}  (t_j^{p-2} |v_j|^p + |v_j|^2) \,dx\\
	\geq - C' \Big(t_j^{p-2} \delta^{1 - \frac{p}{q}} +  \delta^{1 - \frac{2}{q}} \Big).
\]
Putting together the last two estimates, we get
\[
\liminf_{j \rightarrow \infty} \int_\Omega \frac{1}{t_j^2} h(x, t_j v_j(x)) \,dx  \geq  - C' \limsup_{j \rightarrow \infty}\Big(t_j^{p-2} \delta^{1 - \frac{p}{q}} +   \delta^{1 - \frac{2}{q}}\Big).
\]
Clearly, when $\delta$ is sufficiently small, this gives a contradiction to \eqref{Eq:RA-1}.
\end{proof}

\def\cprime{$'$}

\end{document}